\documentclass[a4paper,12pt,reqno]{amsart}
\usepackage{amsfonts,amsmath,amssymb,amstext,amsxtra,mathrsfs}
\usepackage{euscript}
\usepackage{inputenc}
\usepackage{color,graphicx,psfrag,subfigure,xcolor}
\usepackage{enumitem}
\usepackage[english]{babel}
\usepackage[T1]{fontenc}
\usepackage{dsfont}
\usepackage{pdftexcmds}
\usepackage{multirow}

\usepackage{scalerel,stackengine}
\stackMath
\newcommand\reallywidehat[1]{%
\savestack{\tmpbox}{\stretchto{%
  \scaleto{%
    \scalerel*[\widthof{\ensuremath{#1}}]{\kern-.6pt\bigwedge\kern-.6pt}%
    {\rule[-\textheight/2]{1ex}{\textheight}}
  }{\textheight}%
}{0.5ex}}%
\stackon[1pt]{#1}{\tmpbox}%
}
\usepackage{cancel, ulem}
\usepackage{pgf,tikz}
\usetikzlibrary{arrows}

\usepackage{hyperref}
\usepackage{tikz}

\setlist[itemize]{leftmargin=*}

\definecolor{darkgreen}{rgb}{0,0.6,0}

\newtheorem{theorem}{Theorem}[section]

\newtheorem{lemma}[theorem]{Lemma}
\newtheorem{corollary}[theorem]{Corollary}
\newtheorem{remark}[theorem]{Remark}

\DeclareRobustCommand\circled[1]{\tikz[baseline=(char.base)]{\node[shape=circle,draw,inner sep=2pt](char){#1};}}

\newcommand\bc{\mathbf{c}\vsf}
\newcommand\C{\mathbb{C}}
\newcommand\const{\operatorname{const.}}

\renewcommand\epsilon{\varepsilon}
\renewcommand\H{\mathbb{H}}
\newcommand\Hn{\mathbb{H}^{\ssf n}}

\newcommand\msb{\hspace{-.5mm}}
\newcommand\msf{\hspace{.5mm}}
\newcommand\N{\mathbb{N}}
\renewcommand\O{\mathbb{O}}
\renewcommand\phi{\varphi}

\renewcommand\Re{\operatorname{Re}}
\newcommand\R{\mathbb{R}}
\newcommand\Rn{\mathbb{R}^n}

\newcommand\ssb{\hspace{-.3mm}}
\newcommand\ssf{\hspace{.3mm}}
\newcommand\supp{\operatorname{supp}}

\newcommand\TQ{\mathbb{T}_{\ssb Q}}

\newcommand\vsb{\hspace{-.1mm}}
\newcommand\vsf{\hspace{.1mm}}
\newcommand\Z{\mathbb{Z}}

\title[Schr\"odinger equation for the fractional Laplacian]
{The Schr\"odinger equation with fractional Laplacian on hyperbolic spaces and homogeneous trees}

\begin{document}
\author[J.-P. Anker]{Jean-Philippe Anker}
\address{Universit\'{e} d’Orl\'{e}ans, Universit\'{e} de Tours \& CNRS,
Institut Denis Poisson (UMR 7013), 
B\^{a}timent de Math\'{e}matiques, Rue de Chartres,
B.P. 6759, 45067 Orl\'eans cedex 2, France}
\email{anker@univ-orleans.fr}
\author[G. Palmirotta]{Guendalina Palmirotta}
\address{Universit\"at Paderborn, Warburgerstra{\ss}e 100, 33098 Paderborn, Germany} \email{gpalmi@math.uni-paderborn.de}
\author[Y. Sire]{Yannick Sire}
\address{Johns Hopkins University, Krieger Hall, 3400 North Charles Street, Baltimore, MD, 21231, USA}
\email{ysire1@jhu.edu}

\begin{abstract}
We investigate dispersive and Strichartz estimates for the  Schr\"odinger equation involving the fractional Laplacian in real hyperbolic spaces and their discrete analogues, homogeneous trees. Due to the Knapp phenomenon, the Strichartz estimates on Euclidean spaces for the fractional Laplacian exhibit loss of derivatives. A similar phenomenon appears on real hyperbolic spaces. However, such a loss disappears on homogeneous trees, due to the triviality of the estimates for small times.
\end{abstract}



\maketitle



\section{Introduction}

The aim of the present work is to
derive dispersive and Strichartz estimates for Schr\"odinger equations associated to the \textit{fractional Laplacian} on negatively curved manifolds like real hyperbolic spaces and their discrete counterparts, homogeneous trees. Specifically, consider the following Cauchy problem for the  fractional Schr\"{o}dinger equation:
\begin{equation}\label{eq:NLSgeneral}
\begin{cases}
    i\partial_tu + (-\Delta)^{\alpha/2}u =F(x,t)   &(x,t)\in M \times \R \\
    u|_{t=0} = u_0,
\end{cases}
\end{equation}
where $M$ stands for either $\H^n$ ($n \geq 2$), the real hyperbolic space with its standard metric; $\mathbb{T}_Q$ ($Q \geq 2$), the homogeneous tree with $Q+1$ edges; or
$\R^n$ ($n \geq 2$), the Euclidean space with the flat metric.
The operator $(-\Delta)^{\alpha/2}$ represents the Fourier multiplier of order $\alpha \in (0,2)$ associated to powers of the corresponding Laplacian on $M$. Here $F(x,t)$ is a nonhomogeneous term defined on $M \times \mathbb R$ and $u_0$, defined on $M$, stands for the initial data.
\smallskip

Positive powers of the Laplace-Beltrami operator, known as the fractional Laplacian, appear in several areas of mathematical physics such as relativistic theories (see e.g.,  \cite{CarmonaMastersSimon1990, DaubechiesLieb1983, FrankLiebSeiringer2007, LiebYau1987, LiebYau1988}). 
Such operators are also reminiscent of stochastic processes with pure jumps since they are the infinitesimal generators of stable Levy processes (see the book by Bertoin \cite{Bertoin1996} for a detailed account). E.g. an attempt to reinterpret Feynman's path integral into the framework of Levy processes has been undertaken in \cite{Laskin2002}. \\
Additionally, significant progress has been made in understanding such operators using techniques from partial differential equations. While a comprehensive review is beyond the scope of this work, we refer readers to the books and surveys \cite{BookFRRO, SurveyCaffarelli, SurveyCIME, Vazquez2014} for further insights.\\
\smallskip

We now outline our main results for $M=\{\H^n, \TQ\}$, namely the dispersive estimates for Equation \eqref{eq:NLSgeneral}. As our focus is the harmonic analysis of the multipliers $(-\Delta)^{\alpha/2}$ on rank 1 symmetric spaces of noncompact type, we state below the boundedness properties in $L^p-L^q$ of the multiplier in question. 
\begin{theorem}[{Dispersive estimates on hyperbolic spaces}]\label{thm:DispersiveEstimateHn_CaseAlphaLargeIntro}
Consider $M=\H^n.$
\begin{itemize}
\item
Assume that \,$n\ssb\ge\ssb2$\ssf, $1\!<\msb\alpha\msb<\msb2$\ssf,
\vsf$0\ssb\le\msb\sigma\msb\le\msb\frac n2$\vsf,
\vsf$2\msb<\msb q\msb\le\msb\infty$ and let
\begin{equation*}
m=\max\ssf\bigl\{2\ssf\tfrac{n\vsf-\vsf\sigma}\alpha,
\tfrac{n\vsf-\vsf2\ssf\sigma}{\alpha\vsf-1}\bigr\}
=\begin{cases}
\,2\ssf\frac{n\vsf-\vsf\sigma}\alpha
&\text{if \,$\sigma\msb\ge\msb(1\msb-\msb\frac\alpha2)\ssf n$\ssf,}\\
\msf\tfrac{n\vsf-\vsf2\ssf\sigma}{\alpha\vsf-1}
&\text{if \,$\sigma\msb\le\msb(1\msb-\msb\frac\alpha2)\ssf n$\ssf.}
\end{cases}
\end{equation*}
Then the following dispersive estimates hold for \,$t\ssb\in\ssb\R^*$ on \,$\H^n:$
\begin{equation*}
\bigl\|\ssf(-\Delta)^{-\ssf(\frac12-\frac1q)\ssf\sigma}\msf
e^{\hspace{.4mm}i\hspace{.4mm}t\ssf(-\Delta)^{\alpha/2}}\ssf
\bigr\|_{L^{q^{\vsf\prime}}\!\to\vsf L^q}
\lesssim\ssf|t|^{-(\frac12-\frac1q)\vsf m}
\end{equation*}
for \,$t$ small, say \,$0\msb<\msb|t|\msb<\msb1$\vsf, and
\begin{equation*}
\bigl\|\ssf(-\Delta)^{-\ssf\sigma/2}\msf
e^{\hspace{.4mm}i\hspace{.4mm}t\ssf(-\Delta)^{\alpha/2}}\ssf
\bigr\|_{L^{q^{\vsf\prime}}\!\to\vsf L^q}
\lesssim\ssf|t|^{-\frac32}
\end{equation*}
for \,$t$ large, say \,$|t|\msb\ge\msb1$\vsf. 
\item 
Assume that \,$n\ssb\ge\ssb3$\ssf, \vsf$0\!<\msb\alpha\msb<\msb1$\vsf,
$(1\!-\msb\frac\alpha2)\ssf n\ssb\le\msb\sigma\msb\le\msb n$ and \,$2\ssb<\ssb q\ssb\le\ssb\infty$\ssf.
Then the following dispersive estimates hold for
\,$t\msb\in\msb\R^*$ on \,$\H^n:$
\begin{equation*}
\bigl\|\ssf(-\Delta)^{-\ssf(\frac12-\frac1q)\ssf\sigma}\msf e^{\hspace{.4mm}i\hspace{.4mm}t\ssf(-\Delta)^{\alpha/2}}\ssf\bigr\|_{L^{q^{\vsf\prime}}\!\to\vsf L^q}\ssb\lesssim|t|^{-\vsf(\frac12-\frac1q)\ssf2\vsf\frac{n-\vsf\sigma}\alpha}
\end{equation*}
for \,$t$ small, say \,$0\ssb<\ssb|t|\ssb<\msb1$\vsf, and
\begin{equation*}
\bigl\|\ssf(-\Delta)^{-\ssf\sigma/2}\msf
e^{\hspace{.4mm}i\hspace{.4mm}t\ssf(-\Delta)^{\alpha/2}}
\ssf\bigr\|_{L^{q^{\vsf\prime}}\!\to\vsf L^q}\ssb
\lesssim|t|^{-\frac32}  
\end{equation*}
for \,$t$ large, say \,$|t|\ssb\ge\msb1$\vsf.
\end{itemize}
\end{theorem}

In Theorem \ref{thm:DispersiveEstimateHn_CaseAlphaLargeIntro}, in the case $0<\alpha<1$, we did not state the estimates whenever $n=2$ since the numerology gets more tedious. We refer the reader to the dedicated Section~\ref{sect:case_alpha_small} for more details on this case.\\

Note that the $L^p\to L^q$ boundedness of such multipliers, for fixed time $t$\ssf, was extensively investigated by Cowling, Giulini and Meda \cite{GiuliniMeda1990, CowlingGiuliniMeda1993, CowlingGiuliniMeda1995, CowlingGiuliniMeda2001, CowlingGiuliniMeda2002}.\\
For specific $\alpha$, some results are known:
For $\alpha = 2$, Equation \eqref{eq:NLSgeneral} has been the subject of extensive study. Focusing specifically on symmetric spaces, key references include \cite{AnkerPierfelice2009, IonescuStaffilani2009, AnkerPierfeliceVallarino2011}, while for higher-rank symmetric spaces, we refer to \cite{AnkerMedaPierfeliceVallarinoZhang2023}. In addition, space-time linear estimates in the Euclidean space are the well-known Strichartz results of Ginibre-Velo \cite{GinibreVelo1979, GinibreVelo1985} and Keel-Tao \cite{KeelTao1998}. \\
For $\alpha \in (0,2)$, the Strichartz and dispersive estimates are due to Cho, Ozawa and Xia \cite{ChoOzawaXia2011,ChoKohSeo2015}. Guo and Wang in \cite{GuoWang2014} derived finer estimates for $1< \alpha < 2$. This loss of derivatives is reminiscent of the Knapp phenomenon (see e.g., \cite[Cor.~3.10.]{GuoWang2014}). \\
For $\alpha = 1$, Equation \eqref{eq:NLSgeneral}, often referred to as the \textit{half-wave equation}, has been investigated in works such as \cite{AnkerPierfeliceVallarino2012, AnkerPierfelice2014, AnkerPierfeliceVallarino2015}.\\
\smallskip

In an influential paper \cite{Gromov1987}, Gromov investigates the so-called hyperbolic groups and provides a discrete analogue of the real hyperbolic space, namely the homogeneous trees introduced before. The next theorem provides dispersive estimates in such a setting.

\begin{theorem}[{Dispersive estimate on homogeneous trees}] \label{thm:DispersiveEstimateTQIntro}
Consider $M=\mathbb T_Q$. Let  \,$0\msb<\msb\alpha\msb\le\msb2$
and \,$2\msb<\msb q,{\tilde{q}}\msb\le\msb\infty$\ssf.
Then the following dispersive estimate holds for
\,$t\msb\in\msb\R^*:$
\begin{equation*}
\bigl\|\msf
e^{\hspace{.4mm}i\hspace{.4mm}t\hspace{.4mm}(-\Delta)^{\alpha/2}}
\ssf\bigr\|_{\ssf\ell^{\ssf q^{\ssf\prime}}\msb{(\mathbb{T}_Q)}\ssf\to\ssf\ell^{\ssf\tilde{q}}{(\mathbb{T}_Q)}}
\lesssim\ssf(1\!+\msb|t|\vsf)^{-\frac32}\msf.
\end{equation*}
\end{theorem}

The previous statements follow from a fine analysis of the kernel of the propagator $e^{i \,t \,(-\Delta)^{\alpha/2}}$. A standard argument then gives Stichartz estimates from the dispersive ones. 
As a quick inspection shows, the phase in the oscillatory integral of the linear solution changes convexity according to the powers $\alpha$. This requires to consider separately the two regimes $\alpha \in (0,1)$ and $\alpha \in (1,2)$. The kernel analysis is substantially more involved whenever $\alpha$ is small. In particular, we observe degeneracies in this case which need to push the phase analysis an order more.  \smallskip
 
This phenomenon is dominant in the continuous setting of $\H^n$ but disappears in the discrete one of $\mathbb T_Q$ since in this case the {\sl local} analysis of the kernel becomes trivial. We point out as well to the reference \cite{DinhJDE} for an investigation on closed manifolds and to \cite{DinhJFA} for one on asymptotically Euclidean manifolds. 
\smallskip

More generally, it has already been observed in the Euclidean case that space-time estimates exhibit a loss of derivatives (see \cite{GuoWang2014}) which is reminiscent of the Knapp phenomenon. However, it was observed by Guo and Wang that such a loss can be removed by assuming radial symmetry. More precisely, it was also shown in \cite{GuoWang2014} by Guo and Wang that one can obtain optimal Strichartz estimates (i.e., without loss) if one restricts to \ssf$\alpha\ssb\in\msb\bigl(\frac2{2\ssf n\vsf-1},\vsb2\bigr)$\vsf. In particular, the number \vsf$\frac{2\ssf n}{2\ssf n\vsf-1}$ \vsf is larger than $1$ and there is a gap between the Strichartz estimates for the wave operator \ssf$\alpha\ssb=\msb1$ \vsf and the ones occurring for larger powers. The loss in question occurs at small scales in space and then can be removed in the case of homogeneous trees. It is worth mentioning that the range of admissible pairs for the Strichartz estimates on \ssf$\Hn$ is larger than the one on Euclidean spaces. This was of course already observed for \ssf$\alpha\ssb=\ssb2$\ssf. However, one still observes a loss of derivatives for general classes of data. A possible way to remove the loss would be to consider the case of data depending only on the geodesic distance to a given pole of \ssf$\Hn$.
\vspace{-2mm}

The case of real hyperbolic space introduces several points departing substantially from the analysis in the Euclidean space.

\begin{enumerate}
\item
First, we observe that the change of convexity in the phase is essential in dealing with the case of powers close to zero. This introduces a phase analysis much more technical than in the case close to \ssf$2$\ssf. 
\item
Second, as already mentioned, the range of admissible exponents for the Strichartz estimates are broader than in the Euclidean case of Cho, Ozawa and Xia \cite{ChoOzawaXia2011}.
\item
Interestingly enough, we notice that one can {\sl remove} the loss of derivatives (the exponent \ssf$\sigma$ \vsf in the previous theorems) provided \ssf$\alpha\ssb<\msb1$ \vsf and paying the price of a much weaker dispersion.  
\end{enumerate}

The loss of derivatives introduces difficulties for the well-posedness theory for the nonlinear problem. Some partial results can be found in the work by Hong and the last author \cite{HS}, as well as several subsequent contributions by many authors. As previously mentioned, the case of radial data is more favorable and a concentration-compactness/rigidity {\sl \`a la Kenig-Merle}, for the energy-critical {nonlinear Schr\"{o}dinger (NLS)} is performed in \cite{GSWZ}.
\smallskip

In the present article, we refrain from developing nonlinear applications of our estimates. Our original motivation is to understand the structure of the kernel of the propagator on Riemannian symmetric spaces of rank one and noncompact type. We also exclude the case \ssf$\alpha\ssb=\msb1 $ \vsf in the continuous setting since it requires different techniques and is also contained in the literature (the half-wave theory, see \cite{Tataru2001, AnkerPierfeliceVallarino2011, MetcalfeTaylor2011, MetcalfeTaylor2012, AnkerPierfelice2014}).
\smallskip

Our paper is organized as follows: Section~\ref{sect:harmonic_analysis} is devoted to a summary of classical notations and preliminary tools of the harmonic analysis of symmetric spaces. Section~\ref{sect:kernel_analysis} provides the refined kernel estimates, which are at the core of the proofs of the previous theorems~\ref{thm:DispersiveEstimateHn_CaseAlphaLargeIntro} and \ref{thm:DispersiveEstimateTQIntro}. Section~\ref{sect:estimates_hyperbolic_spaces} gives the proofs of the dispersive and Strichartz estimates for the real hyperbolic spaces. Finally, Section~\ref{sect:estimates_trees} deals with the case of homogeneous trees. \\

{\textit{Notation}.
\begin{itemize}
\item
\,Given two non-negative functions \ssf$f$ \ssf and \ssf$g$ \ssf defined on \vsf$M$, we write \ssf$f\ssb\lesssim\ssb g$ \ssf if there exists a positive constant \ssf$C$ \ssf such that \ssf$f(x)\ssb\le\ssb C\ssf g(x)$ \ssf for all \ssf$x\ssb\in\msb M$. The expression \ssf$f\ssb\asymp\ssb g$ \ssf means that \ssf$f(x)\ssb\lesssim\ssb g(x)$ \ssf and \ssf$g(x)\ssb\lesssim\ssb f(x)$ \ssf for all \ssf$x\ssb\in\msb M$.
\item
\,The Lebesgue spaces is denoted by $(L^p:=L^p(\R),\|\cdot\|_{L^p})$, $p\ssb\ge\ssb2$\ssf, and the time-space Strichartz spaces $L^p(\R;L^q(M))$, $p,q\ssb\ge\ssb2$\ssf, of \ssf$f$ \ssf on \ssf$\R\msb\times\msb M$ \ssf is defined by
$$
\|f\|_{L^p(\R;L^q(M))}:=\Biggl[\msf\int_\R\Bigl(\msf\int_M|f(t,x)|^q\,\mathrm{d}x\Bigr)^{\frac{p}{q}}\,\mathrm{d}t\,\Biggr]^{\frac1p}.
$$
\,In the discrete setting, we will denote the Lebesgue space by its lower-case letter $(\ell^{\ssf p},\|\cdot\|_{\ell^{\vsf p}})$.
\end{itemize}

\section{Harmonic analysis tools on hyperbolic spaces} \label{sect:harmonic_analysis}

This section is devoted to some preliminary results about harmonic analysis on hyperbolic spaces together with some definitions used throughout the paper. 
\smallskip

\subsection{(Real) hyperbolic spaces} \label{sect:real_Hn}
We define $M= \H^n${, for $n\ge2$,} as the upper branch of a hyperboloid in $\R^{n+1}$ with the metric induced by the Lorenzian metric in $\R^{n+1}$ given by $\mathrm{d}s^2-|\mathrm{d}x|^2$. More precisely, we take
\begin{equation*}\begin{split}
\H^{\vsf n}&=\bigl\{(s,x)\msb\in\msb\R\msb\times\ssb\R^{\vsf n}:s^2\!-\ssb|x|^2\msb=\msb1\vsf\bigr\}\\
&=\bigl\{(s,x)\msb\in\msb\R\msb\times\msb\R^{\vsf n}:(s,x)\msb=\msb(\cosh r\vsf,\vsb{(\sinh r)}\msf\omega)\text{\hspace{2mm} where\hspace{2mm}}r\ssb\geq\ssb 0\ssf,\,\omega\msb\in\msb\mathbb S^{\vsf n-1}\vsf\bigr\}\ssf,
\end{split}
\end{equation*}
with metric
 $g_{\mathbb H^n}=\mathrm{d}r^2+\sinh^2 r\,\mathrm{d}\omega^2,$
where $\mathrm{d}\omega^2$ is metric on {the hypersphere} $\mathbb S^{n-1}$.   
Notice that via a stereographic projection, one obtains the half-space model
$$\H^n=\{(x_1,\ldots,x_n) {\in \R^n}: x_n>0\},$$
with the metric
$g_{\mathbb{H}^n}=\frac{\mathrm{d}x_1^2+\ldots+\mathrm{d}x_n^2}{x_n^2}.$
By choosing coordinates $x=(\tilde x,x_n)$, $\tilde x=(x_1,\ldots,x_{n-1})$, we denote the volume element by
$$\mathrm{d}V_{\mathbb H^n}(x){:}=(x_n)^{{-}n}\;\mathrm{d}\tilde x \;\mathrm{d}x_n.$$
The fact that we consider for simplicity the real case plays no role here except for simplicity of the presentation and all the formulas extend to {\sl all} hyperbolic spaces, i.e., any rank {one} Riemannian symmetric spaces of noncompact type. 

\subsection{Fractional Laplacian on hyperbolic spaces}
Under the parametrization of Sect.~\ref{sect:real_Hn} the Laplace Beltrami operator is given by
\begin{equation}
\label{eq:Laplace-Beltrami}\Delta_{\H^n} =x_n^2\Delta -(n-2)x_n\partial_n.\end{equation}
Here $\Delta$ denotes the Euclidean Laplacian in coordinates $x_1,\ldots,x_n {\in \R^n}$ {and $\partial_n=\frac{\partial}{\partial x_n}$ is the partial derivative with respect to $x_n$.}
Before defining its fractional representation, let us first recall in $\R^n$ the Fourier transform{, which} is given by
$$\hat{f}(\xi)=\int_{\R^n} f(x)e^{-2\pi i x\cdot \xi} \;\mathrm{d}x {, \;\; \xi \in \R^n}.$$
Notice that the functions  ${h_\xi(x):=}e^{-2\pi i x\cdot \xi}$ are generalized (in the sense that they do not belong to $L^2$) eigenfunctions of the Laplacian associated to the eigenvalue $-4\pi^2 |\xi|^2$. Moreover, the following inversion formula holds
$$f(x)=\int_{\R^n} \hat{f}(\xi)e^{2\pi i x\cdot \xi} \;\mathrm{d}\xi{, \;\; x \in \R^n}.$$

Similarly, in $\H^n$ we consider the generalized eigenfunctions of the Laplace Beltrami operator:
$$h_{\lambda,\theta}(x)=[x,(1,\theta)]^{i\lambda-\frac{n-1}{2}}, \quad x\in \H^n,$$
where $\lambda\in \R$ and $\theta \in \mathbb S^{n-1}$ and we denoted $[\cdot,\cdot]$ the Lorentzien inner product, i.e.,
$
[(s,x),(\tilde s,\tilde x)]=s\,\tilde s - x \tilde x.
$
Notice that 
$$\Delta_{\H^n} h_{\lambda,\theta}=-\left(\lambda^2+\tfrac{(n-1)^2}{4}\right)h_{\lambda,\theta}.$$
In analogy with the definition in $\R^n$,
the Fourier transform can be defined as$$\hat{f}(\lambda, \theta)=\int_{\H^n} f(x)\,h_{\lambda,\theta}(x)\,\mathrm{d}x,$$
for $\lambda\in\R$, ${\theta}\in \mathbb S^{n-1}$. Moreover, the following inversion formula holds\,:
$$
f(x)=\int_{-\infty}^{\infty}\int_{\mathbb{S}^{n-1}}\bar{h}_{\lambda,\theta}(x)\hat{f}(\lambda,\theta)\,\frac{\mathrm{d}\theta\,\mathrm{d}\lambda}{|{\mathbf{c}}(\lambda)|^2}\,,
$$
where {the Plancherel density}
$$
\frac1{|{\mathbf{c}}(\lambda)|^2}
={\const}\,
\frac{\bigl|\Gamma\bigl(i\vsf\lambda\ssb+\ssb\frac{n\vsf-1}2\bigr)\bigr|^2}
{\bigl|\Gamma(i\vsf\lambda)\bigr|^2}
{\asymp\begin{cases}
\,\lambda^2
&\text{if \ssf$|\lambda|\msb\le\msb1$}\\
\msf|\lambda|^{n-1}
&\text{if \ssf$|\lambda|\msb\ge\msb1$}\\
\end{cases}}
$$
{involves the Harish-Chandra $\mathbf{c}$-function}.
It is easy to check by integration by parts for compactly supported functions, and consequently for  every $f\in L^2(\H^n)$, that
\begin{align*}
\widehat{\Delta_{\H^n}f}(\lambda,\theta)
=\int_{\mathbb{H}^n}f(y)\,\Delta_{\H^n}h_{\lambda,\theta}(y)\,\mathrm{d}y
=-\ssf\bigl(\lambda^2\!+\ssb\tfrac{(n-1)^2}4\bigr)\msf\hat{f}(\lambda, \theta).\end{align*}
Having in mind the theory of spherically symmetric multipliers,
we define the \textit{fractional Laplacian} $(-\Delta_{\mathbb{H}^n})^{\alpha/2}$, with \ssf$0\ssb<\ssb\alpha\ssb<\ssb2$\ssf, on the hyperbolic space $\mathbb{H}^n$  by
$$
\reallywidehat{(-\Delta_{\mathbb{H}^n})^{\frac\alpha2}\ssb f\vsf}(\lambda,\theta)=\bigl(\lambda^2\!+\ssb\tfrac{(n-1)^2}4\bigr)^{\ssb\frac\alpha2}\hat{f}(\lambda,\theta)\,.
$$

Hence we can write the Schr\"odinger solution as
$$
\widehat{u}(\lambda\vsf,\vsb\theta\ssf;\vsb t)
=\reallywidehat{e^{\ssf i\ssf t\ssf(-\Delta_{\mathbb{H}^n})^{\ssb\frac\alpha2}}\ssb u_0}(\lambda\vsf,\vsb\theta\ssf)
=e^{\ssf i\ssf t\ssf\left(\smash{\lambda^2+\frac{(n-1)^2}4}\vphantom{\frac oo}\right)^{\ssb\frac\alpha2}}\widehat{u_0}(\lambda\vsf,\vsb\theta\ssf)\,.
$$


The (mild) solution to \eqref{eq:NLSgeneral} on \ssf$\Hn$ is given by
\begin{equation}\label{eq:SolutionHn}
u({x,t})
=\ssf\underbrace{
e^{\hspace{.4mm}i\hspace{.4mm}t\ssf(-\Delta_{\ssf x})^{\ssb\frac\alpha2}}\msb u_0\ssf(x)
\vphantom{\int_|}}_{\text{homogeneous}}\ssf
{-\,i}\underbrace{\int_{\,0}^{\,t}\!
e^{\hspace{.4mm}i\hspace{.4mm}(t-s)\ssf(-\Delta_{\ssf x})^{\ssb\frac\alpha2}}\msb F({x,s})\,\mathrm{d}s
\vphantom{\int_|}}_{\text{inhomogeneous}}\ssf .
\end{equation}

Our estimates and their proofs extend straightforwardly
to all {Riemannian symmetric spaces $G/K$ of noncompact type of rank $1$ (where $G$ is a noncompact semisimple connected Lie group with finite centre and $K$ is its maximal compact subgroup), i.e., to the four} hyperbolic spaces: {$\H^{\vsf n}\hspace{-1mm}=\msb\text{\rm H}^{\vsf N\!}(\R)$, $\text{\rm H}^{\vsf N\!}(\C)$, $\text{\rm H}^{\vsf N\!}(\mathds{H})$, $\text{\rm H}^{\ssf 2\vsb}(\O)$, where $\mathds{H}$ (resp. $\O$) denotes the field of quaternions (resp. octonions)} in Table~\ref{tab:SymmSpacesRank1} and more generally to all Damek--Ricci spaces.

\vspace{2mm}

\begin{table}[h!]
\centerline{\begin{tabular}{|c|c|c|c|c|}
\hline
\hspace{10mm}
&\,$\H^{\vsf n}\hspace{-1mm}=\msb\text{\rm H}^{\vsf N\!}(\R)$
&\hspace{5mm}$\text{\rm H}^{\vsf N\!}(\C)$\hspace{5mm}
&\hspace{5mm}$\text{\rm H}^{\vsf N\!}(\mathds{H})$\hspace{5mm}
&\hspace{5mm}$\text{\rm H}^{\ssf 2\vsb}(\O)$\hspace{5mm}$\vphantom{\frac||}$\\
\hline
$G$& $\mathrm{SO}(n,1)^\circ$& $\mathrm{SU}(n,1)$& $\mathrm{Sp}(n,1)$& $\mathrm{F}_{4(-20)}$\\
\hline
$K$& $\mathrm{SO}(n)$& $\mathrm{S}[\mathrm{U}(n)\times\mathrm{U}(1)]$& $\mathrm{Sp}(n)$& $\mathrm{Spin}(9)$\\
\hline
$n$&$N$&$2\ssf N$&$4\ssf N$&$16\vphantom{\frac||}$
\\\hline
$\rho$&$\frac{N\vsb-1}2$&$N$&$2\ssf N\hspace{-1mm}+\!1$&$11\vphantom{\frac||}$
\\\hline
\end{tabular}}
\vspace{0.5cm}
\caption{{Symmetric spaces of rank 1. Here $\mathrm{F}_{4(-20)}$ is an exponential noncompact Lie group, $\mathrm{Spin}(9)$ is the spinor group of dimension $9$, $\mathrm{SO}(n,1)^\circ$ is the} {connected} {component} {of} {the identity} {in} {the orthogonal} {Lorentz group $\mathrm{SO}(n,1)$}, $\mathrm{SO}(n)$ is the special orthogonal group, $\mathrm{U}(n)$ is the unitary group, and $\mathrm{SU}(n)$ is the special unitary group}
    \label{tab:SymmSpacesRank1}
\end{table}
\vspace{1mm}

\subsection{Asymptotic expansions of the spherical function}

{We now state several well-known results about asymptotic expansions of the spherical function $\phi_\lambda$ which will be used throughout the proof of the kernel estimates in the next Section~\ref{sect:kernel_analysis}. To facilitate the reading, we collect those formulae into several lemmata. }

{We start the following integral formula due to Harish--Chandra, which holds for all $r >0$\,} (see for instance \cite[Prop.~3.1.4]{GangolliVaradarajan1988}, \cite[Ch.~IV, Theorem 4.3]{Helgason1984} {or \cite[p.~40]{Koornwinder1984}}). We have  
\begin{equation}\label{eq:HCformula}\begin{aligned}
\phi_\lambda(r)
=\!\int_K\!\hspace{1mm}e^{\ssf(\vsf i\vsf\lambda\ssf-\vsf\rho)\ssf H(a_{\vsf\vsf r}k)}\;\mathrm{d}k
&=\!\int_K\!\hspace{1mm}e^{\ssf-\ssf(\vsf i\vsf\lambda\ssf+\vsf\rho)\ssf H(a_{\vsf-r}k)}\ssf\;\mathrm{d}k\\
&{=\tfrac{\Gamma(\ssb\frac n2\ssb)}{\sqrt{\pi}\,\Gamma(\ssb\frac{n\vsf-1}2\ssb)}\ssf\int_0^{\ssf\pi}\ssb(\sin\theta)^{n-2}\msf
(\cosh r\ssb\pm\sinh r\cos\theta\vsf)^{-\rho\ssf\mp\ssf i\vsf\lambda}\,\mathrm{d}\theta},
\end{aligned}\end{equation}
where 
{$
H(a_{\ssf r}\ssf k_\theta)=\log\ssf(\cosh r\ssb+\vsb\sinh r\cos\theta\vsf)\, \in [-r,r],
$
with
$$
a_{\ssf r}=\left(\,\begin{matrix}
\cosh r&0&0&\sinh r\\0&1&0&0\\0&0&I_{n-2}&0\\\sinh r&0&0&\cosh r
\end{matrix}\,\right)
\quad\text{and}\quad
k_\theta=\left(\,\begin{matrix}
1&0&0&0\\0&\cos\theta&0&-\sin\theta\\0&0&I_{n-2}&0\\0&\sin\theta&0&\cos\theta
\end{matrix}\,\right)
$$
in the hyperboloid model of \ssf$\H^{\vsf n}$.}
\smallskip

The following result is a large scale asymptotic (see for instance \cite[Formula (2.17)]{Koornwinder1984}). 

\begin{lemma}[Large scale asymptotic]\label{largescaleExp}
The following large scale con\-ver\-ging expansion
of the spherical functions holds: Let $r_0 >0$ be fixed. Then for every $r>r_0$, we have that
{\begin{equation}\label{eq:HCexpansion}
\phi_\lambda(r)
=\bc(\lambda)\,\Phi_\lambda(r)+\bc(-\lambda)\,\Phi_{-\lambda}(r)
\qquad\forall\;\lambda\!\in\!\C\!\smallsetminus\!i\ssf\Z\ssf,
\end{equation}
holds, where}
\begin{equation}\label{eq:Philambda}\begin{aligned}
\Phi_{\lambda}(r)&\textstyle
\,=(\ssf2\sinh r)^{\ssf i\lambda-\rho}\,{}_2F_1\bigl(
\frac\rho2\!-\!i\ssf\frac\lambda2,-\frac{\rho-1}2\!-\!i\ssf\frac\lambda2\ssf;
1\!-\!i\ssf\lambda\ssf;-\sinh^{-2}\ssb r\bigr)\\
&=\,(\ssf2\sinh r)^{-\rho}\,e^{\ssf i\ssf\lambda\ssf r}\,
\sum\nolimits_{\ssf\ell=0}^{+\infty}\ssf
\Gamma_{\ssb\ell\ssf}(\lambda)\,e^{-2\ssf\ell\ssf r}\,.
\end{aligned}\end{equation}

The coefficients \ssf$\Gamma_{\ssb\ell\ssf}(\lambda)$
\ssf in this expansion are
inhomogeneous symbols of order $0$
on \ssf$\R\ssf$.
More precisely,
there are constants \msf$\gamma\!\ge\!0$
\ssf and \ssf$C_j\!\ge\msb0$ $(j\!\in\!\N)$, such that
\begin{equation}\label{eq:EstimateGamma}
\bigl|\ssf\bigl(\tfrac\partial{\partial\lambda}\bigr)^j\msf
\Gamma_{\ssb\ell\ssf}(\lambda)\ssf\bigr|
\le C_j\,\ell^{\msf\gamma}\msf
(\ssf1\!+\ssb|\lambda|\ssf)^{-\vsf j}
\qquad\forall\;\ell\!\in\!\N^*,\,\forall\;\lambda\!\in\!\R\,.
\end{equation}
\end{lemma}

\begin{remark}
Notice that we actually have \ssf$\Gamma_{\ssb0}\ssb\equiv\ssb1$\ssf,
while the other \ssf$\Gamma_{\ssb\ell}$ are inhomogeneous symbols of order $-1$ (see for instance \cite[Lem.~1]{AnkerPierfeliceVallarino2015}).
\end{remark}

The following small scale asymptotic is due to Stanton and Tomas \cite[Thm.~2.1]{StantonTomas1978}. 

\smallskip
\begin{lemma}[Small scale asymptotics]\label{smallscaleExp}
The following small scale expansion holds:
Let $r_1>0$ be fixed. Then for every $0\le r<r_1$, we have
\begin{equation}\label{eq:StantonTomas}
\phi_\lambda(r)
\sim\sum\nolimits_{\vsf m=0}^{\ssf+\infty}\ssf
r^{\ssf2\ssf m}\,\tilde{b}_{\ssf m\vsb}(r)\,j_{\vsf m\vsf+\frac n2-1}(\lambda\ssf r)\,,
\end{equation}
where
\begin{eqnarray*}
j_\nu(z)
=\tfrac{\Gamma(\vsf\nu\ssf+\vsf1)\vphantom{\frac00}}
{\sqrt{\pi\ssf}\,\Gamma(\vsf\nu\ssf+\frac12)\vphantom{\frac|0}}
\int_{-1}^{\ssf+1}\hspace{-1mm}\,
(1\!-\msb u^2\vsf)^{\vsf\nu\vsf-\frac12}\msf
e^{\msf i\ssf z\ssf u} \; \mathrm{d}u
&=&\sum\nolimits_{\ssf m=0}^{\ssf+\infty}\,
\tfrac{(-1)^m\,\Gamma(\nu\vsf+1)\vphantom{\frac0|}}
{m\ssf!\;\Gamma(m\vsf+\vsf\nu\vsf+1)\vphantom{\frac|0}}\,
\bigl(\tfrac z2\bigr)^{\ssb2\ssf m}\\
&=& \vsb
\Gamma(\nu\msb+\!1)\ssb\left(\frac z2\right)^{\msb-\nu}\msb J_\nu(z)\ssf, 
\end{eqnarray*}
is a modified Bessel function
and
\msf$\tilde{b}_{\ssf0\vsb}(r)\msb=\msb
\bigl(\tfrac{r\vphantom{|}}{\sinh r}\bigr)^{\msb\frac{n-1}2}$
\ssf is the Jacobian of the exponential map, raised to the power \ssf$-\frac12$\ssf.
More precisely, for every $M\hspace{-1mm}\in\msb\N^*$\ssb,
\begin{equation*}
\phi_\lambda(r)\ssf=\ssf\sum\nolimits_{\ssf0\le m<M}\ssb
r^{\ssf2\ssf m}\,\tilde{b}_{\ssf m\vsb}(r)\,j_{\vsf m\vsf+\frac n2-1}(\lambda\ssf r)\ssf
+\msf r^{\ssf2\vsf M}\widetilde{R}_{\vsf M}(\lambda,r)\,,
\end{equation*}
where the coefficients $\tilde{b}_{\ssf m\vsb}(r)$ are smooth even functions and
\begin{equation*}
|\ssf\widetilde{R}_{\vsf M}(\lambda,r)\ssf|
\le\widetilde{C}_{\vsb M}\,(1\!+\msb|\lambda\ssf r|\vsf)^{-\frac{n-1}2-M}\,.
\end{equation*}
\end{lemma}

\begin{remark}
    Such asymptotics are closely related to the dual Abel transform. {Note that in} \cite{AnkerPierfelice2009, AnkerPierfeliceVallarino2011, AnkerPierfeliceVallarino2012, AnkerPierfeliceVallarino2015, AnkerPierfelice2014}, the inverse Abel transform was used instead of the asymptotic expansions we used in the present article.
\end{remark}

As noticed by Ionescu \cite[Prop.~A2.(b)]{Ionescu2000},
by combining \eqref{eq:StantonTomas} with the classical asymptotics
(see for instance \cite[10.17.3]{DLMF})
\begin{equation*}
j_\nu(z)\ssf\sim\msf
e^{\msf i\ssf z}\msf\sum\nolimits_{\ssf m=0}^{\ssf+\infty}
\beta_{\ssf m\vsf,\vsf\nu}\,(i\ssf z)^{-\vsf m-\vsf\nu-\frac12}
+\ssf e^{-\ssf i\ssf z}\msf\sum\nolimits_{\ssf m=0}^{\ssf+\infty}\ssf
\beta_{\ssf m\vsf,\vsf\nu}\,(\ssb-\msf i\ssf z)^{-\vsf m-\vsf\nu-\frac12}\,,
\end{equation*}
one obtains,
for every $M\hspace{-1mm}\in\msb\N^*$\ssb,
for every $\Lambda\msb>\msb0$\ssf,
for every \ssf$\lambda\!\in\msb\R^*$ \ssb
and for every \ssf$0\msb<\msb r\msb\le\!1$ \vsf
such that $|\lambda\ssf r|\ssb\ge\msb\Lambda$\ssf,
\begin{equation}\label{eq:STI}
\phi_\lambda(r)
=\vsf b_{\vsf M}(\lambda,r)\,e^{\msf i\ssf\lambda\msf r}\msb
+\vsb b_{\vsf M}(\vsb-\vsb\lambda,r)\,e^{-\ssf i\ssf\lambda\msf r}\msb
+\vsb R_{\ssf M}(\lambda,r)\,,
\end{equation}
where
\begin{equation}\label{eq:EstimateBM}
|\bigl(\tfrac\partial{\partial\lambda}\bigr)^{\vsb j}\msf b_{\vsf M}(\lambda,r)\vsf|
\le C_{\vsb M}\,|\lambda\ssf r|^{-\frac{n-1}2}\,|\lambda|^{-j}
\qquad\forall\;0\msb\le\ssb j\msb\le\msb M
\end{equation}
and
\begin{equation}\label{eq:EstimateRM}
|\ssf R_{\vsf M}(\lambda,r)\ssf|
\le\vsf C_{\vsb M}\,|\lambda\ssf r|^{-M\vsb-\frac{n-1}2}\msf.
\end{equation}


\section{Kernel analysis} \label{sect:kernel_analysis}
{This section is devoted to the results in the continuous setting. The analysis relies on fine kernel estimates and we observe the dichotomy between $\alpha \in (0,1)$ and $\alpha \in (1,2)$ which amounts to investigate the two different behaviours for the phase of the oscillatory integral.}
The integral expression in \eqref{eq:SolutionHn} involves the following propagator
\begin{equation*}
e^{\hspace{.4mm}i\hspace{.4mm}t\ssf(-\Delta)^{\alpha/2}}\msb f\ssf(x)
=f\ssb*k_{\ssf t}\ssf(x)
=\!\int_{\Hn}\hspace{-1mm}f(y)\,k_{\ssf t}({d\vsf(x,y)}\ssb)\,\mathrm{d}y,
\end{equation*}
which is the radial convolution operator defined by
the inverse spherical Fourier transform
\begin{equation}\label{eq:KernelHn}
k_{\ssf t}(r)=\const\int_{-\infty}^{\vsf+\infty}\hspace{-1mm}
e^{\hspace{.4mm}i\hspace{.4mm}t\ssf(\lambda^2+\ssf\rho^2)^{\alpha/2}}
\msf\phi_\lambda(r)\,\tfrac{\mathrm{d}\lambda}{|\vsf\bc(\lambda)|^2}
\qquad\forall\;t\ssb\in\ssb\R^*,\,\forall\;r\ssb\ge\ssb0\ssf.
\end{equation}
Here \msf$\rho^{\ssf2}\!=\!\smash{\bigl(\frac{n\vsf-1}2\bigr)^{\ssb2}}$
\ssf is the bottom of the \ssf$L^2$ spectrum of \ssf$-\Delta$ \ssf on \ssf$\Hn$.
Notice that, in comparison with the Hankel transform
(i.e., the Fourier transform of radial functions in \ssf$\Rn$),
the modified Bessel functions are replaced in \eqref{eq:KernelHn}
by the spherical functions \ssf$\phi_\lambda(r)$
and the Plancherel density \ssf$\lambda^{n-1}$
by \ssf$|\bc(\lambda)|^{-2}$,
where
\begin{equation*}
\bc(\lambda)=\ssf\tfrac{\Gamma(2\rho)}{\Gamma(\rho)}\ssf
\tfrac{\Gamma(i\lambda)}{\Gamma(i\lambda+\rho)}\msf.
\end{equation*}

\begin{remark}
Notice that \,$|\phi_\lambda(r)|\msb\le\ssb\phi_0(r)\msb\asymp\msb(1\!+\ssb r)\msf e^{-\rho\ssf r}$ \ssf for every \ssf$\lambda\msb\in\msb\R$ \ssf and \ssf$r\msb\ge\msb0$\ssf.
Moreover, one should (often) factorize
\ssf$\bc(\lambda)^{-1}\!=\ssb\lambda\ssf
\bigl(\vsf i\ssf\tfrac{\Gamma(\rho)}{\Gamma(2\rho)}\ssf
\tfrac{\Gamma(i\lambda\vsf+\rho)}{\Gamma(i\lambda\vsf+1)}\vsf\bigr)$
and use the fact that the parenthesis is an inhomogeneous symbol on $\R$ of order $\frac{n-3}2$\vsf. 
\end{remark}

More generally we shall consider the operator 
\msf$(-\Delta_x)^{-\ssf\sigma/2}\msf
e^{\hspace{.4mm}i\hspace{.4mm}t\ssf(-\Delta_x)^{\vsf\alpha\vsb/2}}$ {for an additional smoothness}
{$\sigma\msb\in\msb\C$ with $\Re\sigma\msb\ge\ssb0$\ssf,} and its kernel
\begin{equation}\label{eq:KernelSigma}
k_{\ssf t}^{\ssf\sigma}\ssb(r)=\ssf\const\int_{-\infty}^{\vsf+\infty}\hspace{-1mm}
(\lambda^{\vsb2}\!+\msb\rho^{\ssf2}\vsf)^{-\frac\sigma2}\,
e^{\hspace{.4mm}i\hspace{.4mm}t\hspace{.4mm}
(\lambda^{\vsb2}+\ssf\rho^{\vsf2}\vsf)^{\vsf\alpha\vsb/2}}
\msf\phi_\lambda(r)\,\tfrac{\mathrm{d}\lambda}{|\vsf\bc(\lambda)|^2}\;.
\end{equation}
This will lead us to analyze oscillatory integrals
\begin{equation*}
\int_{-\infty}^{\vsf+\infty}\hspace{-1mm}a\vsf(\lambda)\,
e^{\hspace{.4mm}i\hspace{.4mm}t\msf\psi\vsf(\lambda)} \,\mathrm{d}\lambda
\end{equation*}
involving the phase
\begin{equation}\label{eq:phaseHn}
\psi(\lambda)=\psi_{R}(\lambda)=
(\lambda^{\vsb2}\!+\msb\rho^{\ssf2}\vsf)^{\alpha/2}\ssb-\ssb R\msf\lambda\,,
\end{equation}
with \msf$t\msb>\msb0$ \ssf and \ssf$R\ssb=\ssb\frac rt\msb\ge\ssb0$\ssf, {and} {amplitudes} {$a(\lambda)$} {involving $(\lambda^2+\rho^2)^{-\sigma/2}$ 
and the $\bc$-function}.

}

Without loss of generality, we may assume that
\msf$t\msb>\msb0$\ssf.
The function $\psi$ will be the phase of the oscillatory integral associated to the propagator. As a consequence, the following technical lemmata are the basis of the kernel analysis (see Section~\ref{sect:Proof}). 

\begin{lemma}[Phase for $1<\alpha<2$]\label{lem:phase1Hn}
Let \msf$1\!<\msb\alpha\msb<\msb2$\ssf.
Then \eqref{eq:phaseHn} has a single stationary point \msf$\lambda_1$\vsf,
which is nonnegative and comparable to
\vspace{-1mm}
\begin{equation*}\begin{cases}
\,R
&\text{if }\,R\msb\le\!1\ssf,\\
\,R^{\vsf\frac 1{\alpha\vsf-1}}
&\text{if }\,R\msb\ge\!1\ssf.\\
\end{cases}\end{equation*}

Moreover, {the following properties hold\,$:$}

{\rm(i)}
\,$\psi^{\msf\prime\prime\ssb}(\lambda)$\ssf,
which is positive and comparable to 
\,$(\vsf|\lambda|\ssb+\msb1)^{-(2-\alpha)}$\ssf,
is an inhomogeneous symbol of order \,$-\ssf(2\ssb-\ssb\alpha)$ on \msf$\R$\ssf.

{\rm(ii)}
\,{Assume that \,$R\ssb>\ssb0$\ssf.
Then,} for any fixed \,$0\msb<\msb\beta\msb<\!1$\vsf,
\,$|\ssf\psi^{\msf\prime\vsb}(\lambda)\vsf|$ \ssf
is comparable to \msf$\tfrac{|\lambda|\vsf+\ssf\lambda_1}
{(\vsf|\lambda|\vsf+\ssf\lambda_1\vsb+\vsf1)^{\vsf2-\alpha}}$ \ssf
when \,$\lambda\msb\in\ssb{\R\msb\smallsetminus}\msb
(\vsf\beta\vsf\lambda_1,\vsb\beta^{-1}\ssb\lambda_1)\ssf$.
\end{lemma}

\begin{proof}
Let us compute the first two derivatives
\vspace{-.5mm}
\begin{equation}\label{eq:Psiprime}
\psi^{\msf\prime\vsb}(\lambda)\vsb=\alpha\,\lambda\msf
(\lambda^{\vsb2}\!+\msb\rho^{\ssf2}\vsf)^{\vsf-\vsf(1-\frac\alpha2)}\ssb-\ssb R
\end{equation}
\vspace{-7mm}

and
\begin{equation}\label{eq:Psidoubleprime}
\psi^{\msf\prime\prime\ssb}(\lambda)\vsb=\alpha\,
(\lambda^{\vsb2}\!+\msb\rho^{\ssf2}\vsf)^{\vsb-\vsf(2\vsf-\frac\alpha2)}\msf
[\ssf(\alpha\msb-\!1)\ssf\lambda^{\vsb2}\!+\msb\rho^{\ssf2}\msf]\msf.
\end{equation}
(i) is an immediate consequence of \eqref{eq:Psidoubleprime}.
All claims about the stationary point \ssf$\lambda_1$
follow from the equation \msf$\theta(\lambda_1\vsb)\msb=\msb R$ \ssf
and from the behavior of the function
\begin{equation}\label{eq:theta}
\theta\vsf(\lambda)\vsb=\alpha\msf\lambda\msf
(\lambda^{\vsb2}\!+\msb\rho^{\ssf2}\vsf)^{\vsb-\vsf(1-\frac\alpha2)}
\end{equation}
(see Figure \ref{fig:Theta}),
which is odd, strictly increasing and comparable
to
\vspace{.5mm}

\centerline{$\begin{cases}
\,\lambda
&\text{on \,}[\ssf0\ssf,\ssb1\vsf]\ssf,\\
\,\lambda^{\vsf\alpha\ssf-1}
&\text{on \,}[\vsf1\vsf,\ssb+\infty)\ssf.\\
\end{cases}$}\vspace{-.5mm}
Actually \msf$\theta(\lambda)\msb\le\ssb\alpha\ssf\lambda^{\alpha-1}$ \vsf on \ssf$[\ssf0,+\infty{)}$\,,
\ssf hence \ssf$\lambda_1\msb\ge\ssb(R\ssf/\alpha)^{1/(\alpha-1)}$\vsf.
{Finally, assume that \ssf$R\ssb>\ssb0$ \ssf and l}et us estimate
\vspace{-.5mm}

\centerline{$
\psi^{\msf\prime\vsb}(\lambda)=\vsf\theta\vsf(\lambda)\msb-\vsb\theta_1\ssf,
$}\vspace{.5mm}

where \ssf$\theta_1\!=\ssb\theta(\lambda_1)$\ssf.
On the one hand,
\vspace{.5mm}

\centerline{$
|\psi^{\msf\prime\vsb}(\lambda)\vsf|\le\vsf\theta\vsf(|\lambda|)\msb+\vsb\theta_1\ssf.
$}\vspace{.5mm}

On the other hand, notice that, {for $\beta \in (0,1)$,}
\begin{equation*}
\tfrac{\theta\vsf(\beta\vsf\lambda)}{\theta\vsf(\lambda)}
=\beta^{\ssf\alpha\vsf-1}\ssf\bigl(1\!-\msb\tfrac{(\beta^{-2}-1)\rho^{\vsf2}}
{\lambda^{\vsb2}+\ssf\beta^{-2}\vsb\rho^{\vsf2}}\bigr)^{1-\frac\alpha2}\ssb
<\beta^{\ssf\alpha\vsf-1}
\qquad\forall\,\lambda\msb>\msb0\ssf.
\end{equation*}
Hence

\centerline{$
\theta\vsf(\lambda)\msb-\vsb\theta_1\ssb
\ge\theta\vsf(\lambda)\msb-\vsb\theta\vsf(\beta\vsf\lambda)
\ge[\vsf1\!-\msb\beta^{\ssf\alpha\vsf-1}\vsf]\,\theta\vsf(\lambda)
$}\vspace{.5mm}

if \ssf$\lambda\msb\ge\msb\beta^{-1}\vsb\lambda_1$\vsf, while

\centerline{$
\theta_1\!-\vsb\theta\vsf(\lambda)
\ge\theta\vsf(\lambda_1)\msb-\vsb\theta\vsf(\beta\vsf\lambda_1)
\ge[\vsf1\!-\msb\beta^{\ssf\alpha\vsf-1}\vsf]\,\theta_1
$}\vspace{.5mm}

if \msf$0\msb<\msb\lambda\msb\le\msb\beta\vsf\lambda_1$\vsf.
Moreover,

\centerline{$
\theta_1\!-\vsb\theta\vsf(\lambda)\ssb
=\theta_1\!+\vsb\theta\vsf(|\lambda|)
\qquad\forall\,\lambda\!\le\msb0\msf.
$}\vspace{.5mm}

In conclusion,
\ssf$|\ssf\psi^{\msf\prime\vsb}(\lambda)\vsf|$ \ssf is comparable to
\vspace{1mm}

\centerline{$
\theta\vsf(|\lambda|)\msb+\ssb\theta_1\ssb
\asymp\smash{\tfrac{\vphantom{\frac00}|\lambda|\vsf+\ssf\lambda_1}
{(\vsf|\lambda|\vsf+\ssf\lambda_1\vsb+\vsf1)^{\vsf2-\alpha}
\vphantom{\frac00}}}
$}

when \msf$\lambda\msb\notin\msb
(\vsf\beta\vsf\lambda_1,\vsb\beta^{-1}\vsb\lambda_1)$\ssf.
\end{proof}


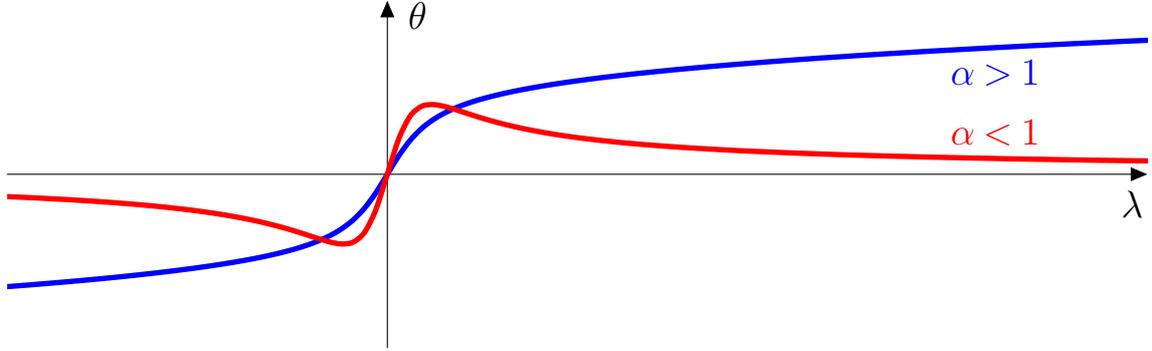
\begin{figure}[h!]
\centering
\begin{tikzpicture}[line cap=round,line join=round,>=triangle 45,x=1.0cm,y=1.0cm]
\draw[->,color=black] (-5.,0.) -- (10.,0.);
\foreach\x in {-5.,-4.,-3.,-2.,-1.,1.,2.,3.,4.,5.,6.,7.,8.,9.,10.}
\draw[shift={(\x,0)},color=black] (0pt,-2pt);
\draw[->,color=black] (0.,-2.3000076462765966) -- (0.,2.300007646276595);
\foreach\y in {-2.,2.}
\draw[shift={(0,\y)},color=black] (-2pt,0pt);
\clip(-5.,-2.3000076462765966) rectangle (10.,2.300007646276595);
\draw[line width=2.pt,color=blue,smooth,samples=100,domain=-5.0:10.000000000000002] plot(\x,{(\x)*((\x)^(2.0)+1.0/4.0)^(-3.0/8.0)});
\draw[line width=2.pt,color=red,smooth,samples=100,domain=-5.0:10.000000000000002] plot(\x,{(\x)*((\x)^(2.0)+1.0/4.0)^(-7.0/8.0)});
\begin{scriptsize}
\draw[color=black] (9.8,-0.4) node {\large$\lambda$};
\draw[color=black] (0.4,2.1) node {\large$\theta$};
\draw[color=blue] (8.,1.35) node {\large$\alpha\ssb>\ssb1$};
\draw[color=red] (8.,0.55) node {\large$\alpha\ssb<\ssb1$};
\end{scriptsize}
\end{tikzpicture}
\caption{The function \eqref{eq:theta}}
\label{fig:Theta}
\end{figure}

The function \eqref{eq:theta} behaves differently
when $0\msb<\msb\alpha\msb<\!1$\vsf.
It is still odd and positive on \ssf$(\vsf0\vsf,\ssb+\infty)\ssf$.
But now it increases between $0$ and \ssf$\lambda_{\vsf0}\msb
:=\msb\tfrac{\rho\vphantom{|}}{\sqrt{1-\ssf\alpha\ssf}}\!>\msb0$\ssf,
where it reaches its maximum
\ssf$\theta_{\vsf0}\msb=\ssb\theta\vsf(\lambda_{\vsf0})\!>\msb0\ssf$,
and decreases between \ssf$\lambda_{\vsf0}$ \vsf and \ssf$+\infty$\ssf,
where it tends to \ssf$0$\ssf.
Consequently the equation \msf$\theta(\lambda)\msb=\msb R$ \ssf
may have \ssf$0$\ssf, $1$ \vsf or \ssf$2$ \ssf solutions.

\begin{lemma}[Phase for $0<\alpha<1$]\label{lem:phase2Hn}
Let \,$0\msb<\msb\alpha\msb<\!1$\vsf.

{\rm(i)} If \,$R\ssb>\ssb\theta_0$\ssf,
\eqref{eq:phaseHn} has no stationary point.
More precisely,
\vspace{-1mm}
\begin{equation}\label{eq:Estimate1eq_Psiprime}
|\vsf\psi^{\msf\prime\ssb}(\lambda)|\ssb\ge\ssb R\ssb-\ssb\theta_{\vsf0}
\quad\forall\,\lambda\!\in\msb\R\msf.
\end{equation}
\vspace{-6mm}

{\rm(ii)} If \,$R\ssb=\ssb\theta_0$\ssf,
\eqref{eq:phaseHn} has a single stationary point at \msf$\lambda_{\vsf0}$\ssf,
where \msf$\psi^{\msf\prime}$ \msb and \msf$\psi^{\msf\prime\prime}$
\ssb both vanish.

{\rm(iii)} If \,$0\msb<\msb R\msb<\msb\theta_0$\ssf,
\eqref{eq:phaseHn} has two stationary points\,$:$
\vspace{1mm}

\centerline{$\begin{cases}
\,\lambda_1\!\in\msb(0\vsf,\ssb\lambda_{\vsf0})\ssf,
\text{ which is comparable to }\ssf R\ssf,\\
\,\lambda_{\vsf2}\msb\in\msb(\lambda_{\vsf0}\vsf,\ssb+\infty)\ssf,
\text{ which is comparable \ssf
to }\ssf\smash{R^{\ssf-\frac1{1-\vsf\alpha}}}\ssf.\\
\end{cases}$}\vspace{1mm}

Moreover, for every \,$0\ssb<\ssb\beta\ssb<\msb1$\vsf,
\vspace{-1mm}
\begin{equation}\label{eq:Estimate2eq_Psiprime}
|\vsf\psi^{\msf\prime\ssb}(\lambda)|
\asymp(\min\ssf\{\lambda\vsf,\ssb\lambda_{\vsf2}\})^{-(1-\vsf\alpha)}
\qquad\forall\,\lambda\msb\in\ssb[\ssf\lambda_{\vsf0}\vsf,\ssb+\infty)\msb
\smallsetminus\msb(\vsf\beta\vsf\lambda_{\vsf2},\beta^{-1}\ssb\lambda_{\vsf2})\ssf.
\end{equation}

{\rm(iv)} If \,$R\ssb=\ssb0$\ssf,
\eqref{eq:phaseHn} has a single stationary point at the origin.
More precisely,
\vspace{-1mm}
\begin{equation*}
|\vsf\psi^{\msf\prime\ssb}(\lambda)|
\asymp|\lambda|\msf(|\lambda|\msb+\!1)^{-\vsf(2\vsf-\vsf\alpha)}
\qquad\forall\,\lambda\!\in\msb\R\msf.
\end{equation*}
{\rm(v)} Contrarily to \,$\psi(\lambda)$ and \,$\psi^{\ssf\prime}\ssb(\lambda)$\ssf,
\,$\psi^{\ssf\prime\prime}\ssb(\lambda)$ and \,$\psi^{\ssf\prime\prime\prime}\ssb(\lambda)$
don't depend on \ssf$R$\ssf. Moreover,

$\bullet$ \;$\psi^{\ssf\prime\prime}\ssb(\lambda)$
is an even inhomogeneous symbol of order \,$-\ssf(2\ssb-\ssb\alpha)$ on \,$\R$\ssf,

$\bullet$ \,away from \,$\lambda\msb=\msb\pm\ssf\lambda_{\vsf0}$\ssf,
where it vanishes,  \,$|\psi^{\ssf\prime\prime}\ssb(\lambda)|$
is comparable to \,$(\vsf|\lambda|\msb+\!1)^{-(2-\alpha)}$\ssf,

$\bullet$ \;$\psi^{\ssf\prime\prime\prime}\ssb(\lambda)$ is an odd function on \,$\R$\ssf,
which vanishes at \,$\lambda\msb=\msb0$ and
\,$\lambda\msb=\msb\pm\ssf\sqrt{\ssf3\ssf}\lambda_{\vsf0}$\ssf.
\end{lemma}

\begin{remark}
    {Actually, in $\mathrm{(iii)}$, we have
    \begin{center}
\msf$\theta(\lambda)\msb\le\ssb\alpha\ssf(\lambda^2\!+\msb\rho^{\vsf2})^{-(1-\vsf\alpha)\vsb/\vsf2}$
\vsf on \ssf$[\ssf0,+\infty\ssf{)}$\,, \ssf hence
\ssf$\lambda_{\vsf2}\msb\le\!\sqrt{\lambda_{\vsb2}^2\!+\msb\rho^{\vsf2}\ssf}\!\le\ssb(R\ssf/\alpha)^{-1/(1-\vsf\alpha)}$\vsf.
\end{center}}
\end{remark}

\begin{proof}[Proof of Lemma~\ref{lem:phase2Hn}]
Almost all claims are straightforward consequences of the expressions \eqref{eq:Psiprime},
\eqref{eq:Psidoubleprime} and of the above behavior of \eqref{eq:theta}.
The only exceptions are the last point, which follows from
\vspace{.5mm}

\centerline{$
\psi^{\msf\prime\prime\prime\ssb}(\lambda)\ssb
=\vsb\alpha\msf(2\msb-\msb\alpha)\msf\lambda\,
[\ssf(1\!-\msb\alpha)\ssf\lambda^{\vsb2}\!-\msb3\ssf\rho^{\ssf2}\msf]\ssf
(\lambda^2\!+\msb\rho^{\ssf2}\vsf)^{\frac\alpha2\vsb-\vsf3}\ssf,
$}\vspace{.5mm}

and \eqref{eq:Estimate2eq_Psiprime}, which is proved as (ii) in Lemma \ref{lem:phase1Hn}.
\end{proof}






Hence using the lemmata above, we estimate the kernel for the dichotomy between $\alpha \in (0,1)$ and $\alpha \in (1,2)$.
\begin{theorem}[{Kernel estimates}]\label{thm:KernelEstimate}
{\rm(i)} Assume that \msf$1\!<\msb\alpha\msb<\msb2$
and \,$0\msb\le\ssb\sigma\msb\le\msb\frac n2$\ssf.
Let $\rho=\tfrac{n-1}2$.
Then the following estimates hold,
for \msf$t\msb\in\msb\R^*$ \ssb and \,$r\msb\ge\msb0:$
\vspace{.5mm}

$\bullet$
\,{\rm Large scale\,:}
{$
|\ssf k_{\ssf t}^{\ssf\sigma}\ssb(r)\vsf|
\lesssim\left\{\,\begin{aligned}
&|t|^{-\frac32}\msf(1\!+\msb r)\,e^{-\rho\ssf r} 
&&\text{if \,}|t|\ssb\ge\ssb\max\msf\{1,r\}
&&\text{\footnotesize{$($Subcases 1.1.1 and 2.1.1\msf$)$},}\\
&|t|^{-\frac12\frac{n\vsf-\vsf2\vsf\sigma}{\alpha\vsf-1}}\,
r^{\ssf\frac12\frac{n\vsf-\vsf2\vsf\sigma}{\alpha\vsf-1}-\frac12}\,
e^{-\rho\ssf r}
&&\text{if \,$r\msb\ge\ssb\max\msf\{1,|t|\}$}
&&\text{\footnotesize{$($Subcase 1.1.2\msf$)$}.}
\end{aligned}\right.
$}
\vspace{.5mm}

$\bullet$
\,{\rm Small scale\,:}
{$
|\ssf k_{\ssf t}^{\ssf\sigma}\ssb(r)\vsf|
\lesssim\left\{\,\begin{aligned}
&|t|^{-\frac{n\vsf-\vsf\sigma}\alpha}
&&\text{if \,$r^{\ssf\alpha}\msb\le\ssb|t|\ssb<\msb1$}
&&\text{\footnotesize{$($Subcase 2.1.1\msf$)$},}\\
&|t|^{-\frac12\frac{n\vsf-\vsf2\vsf\sigma}{\alpha\vsf-1}}\,
r^{\ssf\frac12\frac{n\vsf-\vsf2\vsf\sigma}{\alpha\vsf-1}-\frac n2}
&&\text{if \,$|t|\ssb\le\ssb r^{\ssf\alpha}\msb<\msb1$}
&&\text{\footnotesize{$($Subcase 2.1.2\msf$)$}.}\\
\end{aligned}\right.$}
\vspace{1mm}

{\rm(ii)} Assume that \msf$0\msb<\ssb\alpha\ssb<\msb1$
,\,$\frac n2\msb\le\ssb\sigma\msb\le\ssb n$ and $N\!>\msb\frac{n\vsf+1}2\msb-\msb\sigma$\ssf.
Then the following estimates hold,
for \msf$t\msb\in\msb\R^*$ \ssb and \,$r\msb\ge\msb0:$
\vspace{.5mm}

$\bullet$
\,{\rm Large scale\,:}

\centerline{$
|\ssf k_{\ssf t}^{\ssf\sigma}\ssb(r)\vsf|
\lesssim\left\{\,\begin{aligned}
&|t|^{-\frac32}\ssb+|t|^{-\frac12\frac{2\vsf\sigma-\vsf n}{1-\vsf\alpha}}\,r^{\frac12\frac{2\vsf\sigma-\vsf n}{1-\vsf\alpha}{{-\frac n2}}}
&&\text{if \,$0\ssb\le\ssb r\msb\le\msb1\msb\le\ssb|t|$}
&&\text{\footnotesize{$($Subcase 2.2.1\msf$)$}},\\
&\smash{\msb\bigl(\tfrac r{|t|}\bigr)^{\ssf\min\ssf\{\frac32,\frac12\frac{2\vsf\sigma-\vsf n}{1-\vsf\alpha}\}}}\msf r^{-\frac12}\,e^{-\rho\ssf r}
&&\text{if \,$r\msb\ge\msb1$ and \,$\tfrac r{|t|}\msb\le\msb\tfrac12\ssf\theta_0$} &&\text{\footnotesize{$($Subcase 1.2.3\msf$)$}},\\
&r^{-\frac13}\,e^{-\rho\ssf r}
&&\text{if \,$r\msb\ge\msb1$ and \,$\tfrac12\ssf\theta_0\msb<\msb\tfrac r{|t|}\msb<\msb2\msf\theta_0$}
&&\text{\footnotesize{$($Subcase 1.2.2\msf$)$},}\\
&r^{-N}\msf e^{-\rho\ssf r}
&&\text{if \,$r\msb\ge\msb1$ and \,$\tfrac r{|t|}\msb\ge\msb2\ssf\theta_0$}
&&\text{\footnotesize{$($Subcase 1.2.1\msf$)$}.}\\
\end{aligned}\right.$}
\vspace{.5mm}

$\bullet$
\,{\rm Small scale\,:}

\centerline{$
|\ssf k_{\ssf t}^{\ssf\sigma}\ssb(r)\vsf|
\lesssim\left\{\,\begin{aligned}
&|t|^{-\frac{n\vsf-\vsf\sigma}\alpha}
&&\text{if \,$|t|\ssb\le\ssb r^{\vsf\alpha}\msb<\msb1$}
&&\text{\footnotesize{$($Subcase 2.2.3\msf$)$},}\\
&|t|^{-\frac{n\vsf-\vsf\sigma}\alpha}\msb+|t|^{-\frac12\frac{2\vsf\sigma-\vsf n}{1-\vsf\alpha}}\,r^{\frac12\frac{2\vsf\sigma-\vsf n}{1-\vsf\alpha}-\frac n2}
&&\text{if \,$r^{\vsf\alpha}\msb\le\ssb|t|\ssb<\msb1$}
&&\text{\footnotesize{$($Subcase 2.2.2\msf$)$}.}
\end{aligned}\right.$}

\end{theorem}


\begin{remark}
{\rm(i)} Assume that \,$1\msb<\ssb\alpha\ssb<\ssb2$
and \,$0\msb\le\ssb\sigma\msb\le\msb\frac n2$\ssf.
Then the following inequalities are equivalent\,$:$
\vspace{.5mm}

\centerline{$
\sigma\ssb\ge\msb\bigl(1\!-\msb\tfrac\alpha2\bigr)\ssf n\,,\quad
\tfrac{n\vsf-\vsf\sigma}\alpha\msb\le\msb\tfrac n2\,,\quad
\tfrac12\tfrac{n\vsf-\vsf2\vsf\sigma}{\alpha\vsf-1}\msb\le\msb\tfrac{n\vsf-\vsf\sigma}\alpha\,,\quad
\tfrac12\tfrac{n\vsf-\vsf2\vsf\sigma}{\alpha\vsf-1}\msb\le\msb\tfrac n2\,.
$}\vspace{1mm}

Moreover, under these conditions, we have
\vspace{1mm}

\centerline{$
|\ssf k_{\ssf t}^{\ssf\sigma}\ssb(r)\vsf|
\lesssim|t|^{-\frac12\frac{n\vsf-\vsf2\vsf\sigma}{\alpha\vsf-1}}\,
r^{\ssf\frac12\frac{n\vsf-\vsf2\vsf\sigma}{\alpha\vsf-1}-\frac n2}
\le\ssf|t|^{-\frac12\frac{n\vsf-\vsf2\vsf\sigma}{\alpha\vsf-1}}\,
|t|^{\ssf\frac1{2\ssf\alpha}\frac{n\vsf-\vsf2\vsf\sigma}{\alpha\vsf-1}-\frac n{2\ssf\alpha}}
=\ssf|t|^{-\frac{n\vsf-\vsf\sigma}\alpha}
$}\vspace{1mm}

in the range \,$0\ssb<\msb|t|\msb\le\ssb r^{\ssf\alpha}\msb<\msb1$\vsf.


{\rm(ii)} Assume that \,$0\ssb<\ssb\alpha\ssb<\msb1$
and \,$\frac n2\msb\le\ssb\sigma\msb\le\ssb n$\ssf.
Then the following inequalities are equivalent\,$:$
\vspace{-4mm}

\centerline{$
\sigma\ssb\ge\msb\bigl(1\!-\msb\tfrac\alpha2\bigr)\ssf n\,,\quad
\tfrac{n\vsf-\vsf\sigma}\alpha\msb\le\msb\tfrac n2\,,\quad
\tfrac12\tfrac{2\vsf\sigma\vsf-\vsf n}{1-\vsf\alpha}\msb\ge\msb\tfrac{n\vsf-\vsf\sigma}\alpha\,,\quad
\tfrac12\tfrac{2\vsf\sigma\vsf-\vsf n}{1-\vsf\alpha}\msb\ge\msb\tfrac n2\,.
$}\vspace{1mm}

Moreover,
in the range \,$0\ssb<\ssb r^{\ssf\alpha}\msb<\msb|t|\msb<\msb1$\vsf,
we have
\vspace{1mm}

\centerline{$
\max\msf\bigl\{\vsf|t|^{-\frac{n\vsf-\vsf\sigma}\alpha},
|t|^{-\frac12\frac{2\vsf\sigma-\vsf n}{1-\vsf\alpha}}\,
r^{\frac12\frac{2\vsf\sigma-\vsf n}{1-\vsf\alpha}-\frac n2}\bigr\}
=\begin{cases}
\msf|t|^{-\frac{n\vsf-\vsf\sigma}\alpha}
&\text{if \,$\sigma\msb\ge\msb(1\!-\msb\frac\alpha2)\ssf n$\ssf,}\\
\msf|t|^{-\frac12\frac{2\vsf\sigma-\vsf n}{1-\vsf\alpha}}\,
r^{\frac12\frac{2\vsf\sigma-\vsf n}{1-\vsf\alpha}-\frac n2}
&\text{if \,$\sigma\msb<\msb(1\!-\msb\frac\alpha2)\ssf n$\ssf.}\\
\end{cases}
$}

\end{remark}

\begin{remark}
Assume that \,$\sigma\msb=\msb(1\!-\msb\frac\alpha2)\msf n$ and $N\!>\msb\frac{1+n(\alpha-1)}2$\ssf.
Then Theorem \ref{thm:KernelEstimate} boils down to
\begin{equation}\label{eq:SimplifiedEstimates1}
|\ssf k_{\ssf t}^{\ssf\sigma}\ssb(r)\vsf|\,
\lesssim\,
\begin{cases}
\,|t|^{-\frac n2}\,(1\!+\ssb r)^{\frac{n\vsf-1}2}\,e^{-\rho\ssf r}
&\text{if \;}|t|\msb\le\msb1\msb+\ssb r\\
\,|t|^{-\frac32}\,(1\!+\ssb r)\,e^{-\rho\ssf r}
&\text{if \;}|t|\msb\ge\msb1\msb+\ssb r\\
\end{cases}
\end{equation}
in the range \msf$1\!<\msb\alpha\msb<\msb2$ and to
\begin{equation}\label{eq:SimplifiedEstimates2}
|\ssf k_{\ssf t}^{\ssf\sigma}\ssb(r)\vsf|\,
\lesssim\,
\begin{cases}
\,|t|^{-\frac n2}
&\text{if \, $0\msb<\msb|t|\msb\le\msb1$ and \,$0\ssb\le\ssb r\msb\le\msb1$}\\
\,|t|^{\ssf-\min\ssf\{\frac32,\ssf\frac n2\}}\msf(1\!+\ssb r)^{\ssf\min\ssf\{1,\ssf\frac{n\vsf-1}2\}}\msf e^{-\rho\ssf r}
&\text{if \,$|t|\msb\ge\msb1$ and \,$0\ssb\le\msb\frac r{|t|}\msb\le\msb\frac12\ssf\theta_0$}\\
\,r^{-\frac13}\,e^{-\rho\ssf r}
&\text{if \,$r\msb\ge\msb1$ and \,$\frac12\ssf\theta_0\msb<\msb\frac r{|t|}\msb<\msb2\msf\theta_0$}\\
\,r^{-N}\msf e^{-\rho\ssf r}
&\text{if \,$r\msb\ge\msb1$ and \,$\frac r{|t|}\msb\ge\msb2\ssf\theta_0$}\\
\end{cases}
\end{equation}
in the range \,$0\msb<\msb\alpha\msb<\!1$\ssf.
In the limit case \,$\alpha\msb=\msb2$\ssf,
the kernel estimates \eqref{eq:SimplifiedEstimates1} were obtained
in \cite{AnkerPierfelice2009} and \cite{AnkerPierfeliceVallarino2011}.
\end{remark}

\subsection{Proof of Theorem \ref{thm:KernelEstimate}} \label{sect:Proof}

\begin{figure}[h!]
\centering
     \includegraphics[scale=0.5]{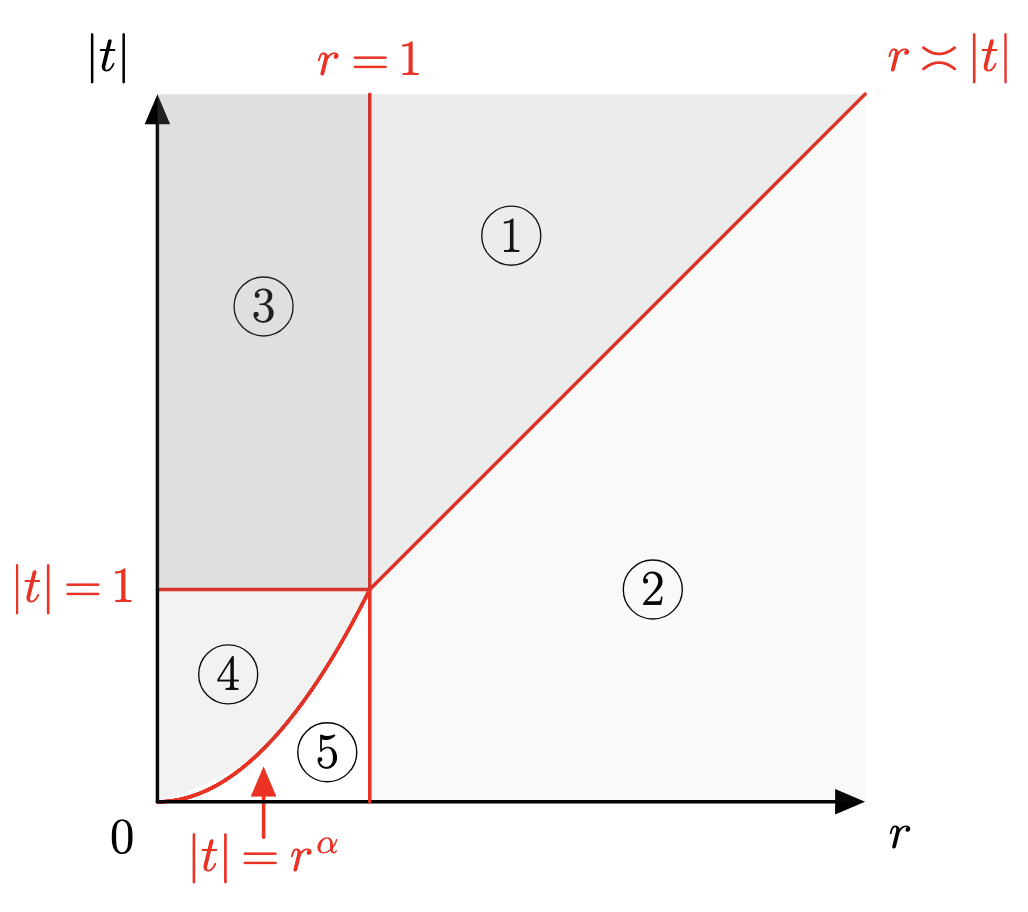}
\caption{Different ranges for the kernel estimates in the case \ssf$1\msb<\ssb\alpha\ssb<\ssb2$.\\
The circled numbers correspond to the different cases in the proof of Theorem~\ref{thm:KernelEstimate}. More precisely, \circled{1} corresponds to Subcase~1.1.1, \circled{2} to Subcase~1.1.2, \circled{3} -- \circled{4} to Subcase~2.1.1 and \circled{5} to Subcase~2.1.2}
\label{fig:RegionsAlphaLarge}
\end{figure}

\begin{figure}[h!]
\centering
     \includegraphics[scale=0.5]{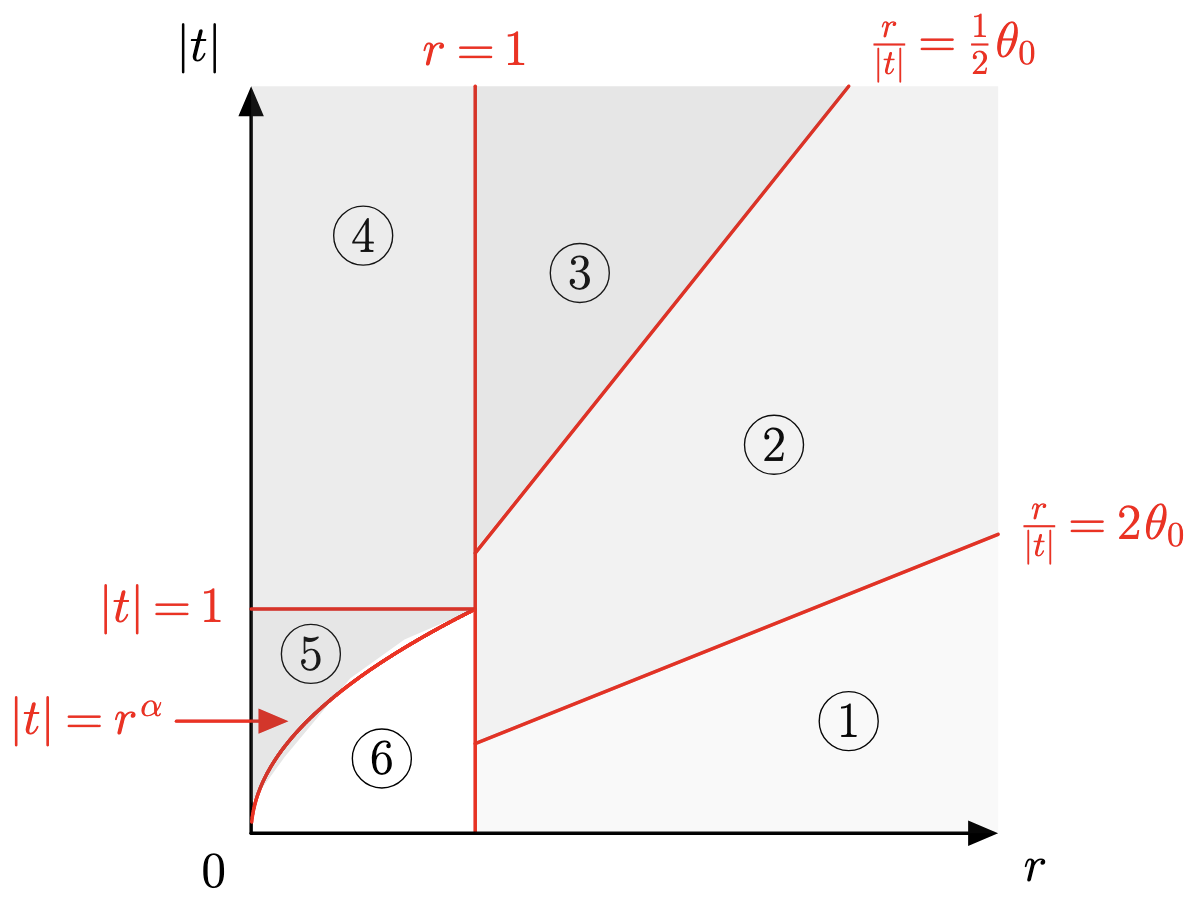}
\caption{Different ranges for the kernel estimates in the case \ssf$0\ssb<\ssb\alpha\ssb<\msb1$.\\
The circled numbers correspond to the different cases in the proof of Theorem~\ref{thm:KernelEstimate}~(ii). More precisely, \circled{1} -- \circled{2} -- \circled{3} correspond to Subcase 1.2, \circled{4} to Subcase 2.2.1, \circled{5} to Subcase 2.2.2 and \circled{6} to Subcase 2.2.3}
\label{fig:RegionsAlphaSmall}
\end{figure}

{ We first explain the global structure of the proof since this is quite technical. Due to the change of behaviour in the phase $\psi$ in Equation \eqref{eq:phaseHn}, it is necessary to consider two regimes $\alpha \in (0,1)$ and $\alpha \in (1,2)$. Of course, the case $\alpha=1$, which is the half-wave has been treated extensively in the literature. 
As explained in the introduction, the change of ``convexity'' of $\psi$ induces different losses which are a major difficulty for nonlinear applications. This also introduces several technical difficulties for the kernel analysis since for $\alpha \in (0,1)$ the phase has {\sl two} stationary points {(see Lemma~\ref{lem:phase2Hn})} and one needs to go to the {\sl third} order. 
Now for each of those ranges in $\alpha$, one needs to consider the different regimes in $r$ and $t$, {which is shown in the corresponding Figures~\ref{fig:RegionsAlphaLarge} and \ref{fig:RegionsAlphaSmall}}.
Because of similarities in the arguments, we prefer to split the proof into several parts according to:
\begin{itemize}
\item[(1)] First case of {\sl large spatial scale} $r \geq r_0>0$ with $r_0$ fixed, { will be treated in Section~\ref{sect:Case1}}. Then we consider the subcases $\alpha \in (0,1)$ and $\alpha \in (1,2).$ 
\item[(2)] Second case of {\sl small spatial scale} $0\leq r \leq r_0$ {can be found in Section~\ref{sect:Case2}}. Then we consider again  the subcases $\alpha \in (0,1)$ and $\alpha \in (1,2).$
\end{itemize}

For each of the cases above, there is additional smallness to consider in the scaled variable~$\frac{r}{t}$. We would like to emphasize that the range $\alpha \in (0,1)$ is the one presenting the most differences with the classical Laplacian estimates. This is because of the structure of the oscillatory integral term that such striking differences occur.
}
\smallskip

{Through the proof, we will} use the following version of the van der Corput Lemma
(see \cite[\!Ch.\ssf VIII\ssf{, Cor.~p.~334}]{Stein1993}) 
when \ssf${L}\msb=\ssb2\text{ or }3$\ssf.
\smallskip
\begin{lemma}\label{lem:vanderCorput}
Let \msf${L}\,{\ge\ssb2}$ be an integer.
Then there exists a constant \,$C\msb>\msb0$ such that
\begin{equation*}
\Bigl|\ssf\int_I a(\lambda)\,
e^{\hspace{.4mm}i\hspace{.4mm}\Psi\vsf(\lambda)}
\;\mathrm{d}\lambda\ssf\Bigr|\ssf
\le\ssf 
C\hspace{1mm}\bigl\{\ssf\|\ssf a\ssf\|_\infty\msb
+\ssb\|\ssf a^{\vsf\prime}\|_1\bigr\}\hspace{1mm}
T^{\ssf-\frac1{L}}\,,
\end{equation*}
for any interval \,$I\!\subset\msb\R$\ssf,
for any \msf$C^{\vsf{L}}$ \ssb function \,$\Psi\ssb:\msb I\msb\longrightarrow\msb\R$
\msf such that \,$|\vsf\Psi^{\vsf({L}\vsf)}|\msb\ge\msb T$ on \msf$I$,
and for any \msf$C^{\vsf1}$ \ssb function
\,$a\ssb:\msb I\msb\longrightarrow\msb\C$\ssf.
\end{lemma}

\subsubsection{Case 1 - large spatial scale} \label{sect:Case1}
If \ssf$r$ \vsf is bounded from below, let us say by 1, we use the large scale expansion provided by {Lemma \ref{largescaleExp}}. 
%
By substituting \eqref{eq:HCexpansion} and \eqref{eq:Philambda} in \eqref{eq:KernelSigma},
we get
\begin{equation*}
k_{\ssf t}^{\ssf\sigma}\ssb(r)=C\,(\sinh r)^{-\rho}\ssf
\sum\nolimits_{\ssf\ell=0}^{\vsf+\infty}
e^{-2\ssf\ell\ssf r}\,k_{\ssf t,\ssf\ell}^{\ssf\sigma}(r)\,,
\end{equation*}
where
\begin{equation}\label{eq:KernelKtellsigmar}
k_{\ssf t,\ssf\ell}^{\ssf\sigma}(r)
=\!\int_{-\infty}^{\vsf+\infty}\hspace{-1mm}\hspace{1mm}
a_{\vsf\ell}(\lambda)\,e^{\hspace{.4mm}i\hspace{.4mm}t\msf\psi_{\vsb\frac rt\ssb}(\lambda)}\,\mathrm{d}\lambda,
\end{equation}
with
\begin{equation}\label{eq:Amplitude}
a_{\vsf\ell}(\lambda)=
\tfrac{\Gamma(i\vsf\lambda\vsf+\vsf\rho\vsf)}{\Gamma(i\vsf\lambda)}\,
(\lambda^{\vsb2}\!+\msb\rho^{\vsf2}\vsf)^{-\frac\sigma2}\,
\Gamma_{\ssb\ell\ssf}(-\lambda\vsf)\,.
\end{equation}

\textbf{Subcase 1.1}.
Assume that \vsf$1\!<\msb\alpha\msb<\msb2$\ssf.
\smallskip

{\textit{Subcase 1.1.1\/}.}
Consider first the case where \msf$\frac rt$
\ssf or equivalently \ssf$\lambda_1$ remains bounded,
let say $\lambda_1\!\le\!1$, {that is case \circled{1} in Figure~\ref{fig:RegionsAlphaLarge}}.
Given {an even} bump function \,$\chi_0\!\in\msb C_c^{\ssf\infty}(\R)$
\ssf such that \ssf$\chi_0\!\equiv\msb1$ \vsf on \ssf$[\vsf{-2}\vsf,\vsb2\ssf]$\ssf,
let us split up
\begin{equation} \label{eq:integral_split}
\int_{-\infty}^{\vsf+\infty}\hspace{-1.5mm}\;\mathrm{d}\lambda
=\int_{-\infty}^{\vsf+\infty}\hspace{-1.5mm}
\hspace{1mm}\chi_0\vsb(\lambda)\;\mathrm{d}\lambda
+\int_{-\infty}^{\vsf+\infty}\hspace{-1.5mm}
[\ssf1\!-\msb\chi_0\vsb(\lambda)\ssf]\;\mathrm{d}\lambda
\end{equation}
in \eqref{eq:KernelKtellsigmar} and
\begin{equation} \label{eq:kernel_split}
k_{\ssf t,\ssf\ell}^{\ssf\sigma}(r)
=\ssf k_{\ssf t,\ssf\ell}^{\ssf\sigma\ssb,\vsf0}\ssb(r)
+\ssf k_{\ssf t,\ssf\ell}^{\ssf\sigma\ssb,\vsf\infty}\ssb(r)
\end{equation}
accordingly.
On the one hand,
after an integration by parts based on
\begin{equation}\label{eq:IBP0}
e^{\hspace{.4mm}i\hspace{.4mm}t\hspace{.4mm}
(\lambda^{\vsb2}+\ssf\rho^{\vsf2}\vsf)^{\vsf\alpha\vsb/2}}\ssb
=\tfrac{-\ssf i}{\alpha\msf t\ssf\lambda}\,
(\lambda^{\vsb2}\!+\msb\rho^{\vsf2\vsf})^{1-\frac\alpha2}\,
\tfrac\partial{\partial\lambda}\,
e^{\hspace{.4mm}i\hspace{.4mm}t\hspace{.4mm}
(\lambda^{\vsb2}+\ssf\rho^{\vsf2}\vsf)^{\vsf\alpha\vsb/2}}\ssf,
\end{equation}
we get
\begin{equation*}
k_{\ssf t,\ssf\ell}^{\ssf\sigma\ssb,\vsf0}\ssb(r)
=\ssf i\,\tfrac rt\int_{-\infty}^{\vsf+\infty}\hspace{-1mm}\hspace{1mm}
e^{\hspace{.4mm}i\hspace{.4mm}t\msf\psi_{\vsb\frac rt\ssb}(\lambda)}\,
a_{\vsf\ell}^{\vsf0}(\lambda) \;\mathrm{d}\lambda
-\ssf\tfrac 1t\int_{-\infty}^{\vsf+\infty}\hspace{-1mm}\hspace{1mm}
e^{\hspace{.4mm}i\hspace{.4mm}t\msf\psi_{\vsb\frac rt\ssb}(\lambda)}\,
\tfrac\partial{\partial\lambda}\msf a_{\vsf\ell}^{\vsf0}(\lambda)\msf\;\mathrm{d}\lambda,
\end{equation*}
where
\begin{equation*}
a_{\vsf\ell}^{\vsf0}(\lambda)=\tfrac1\alpha\,\chi_0\vsb(\lambda)\,
(\lambda^{\vsb2}\!+\msb\rho^{\vsf2}\vsf)^{1-\frac\alpha2-\frac\sigma2}\,
\tfrac{\Gamma(i\vsf\lambda\vsf+\vsf\rho\vsf)}{\Gamma(i\vsf\lambda\vsf+1)}\,
\Gamma_{\ssb\ell\ssf}(-\lambda\vsf)
\end{equation*}
is a smooth function with compact support.
By using Lemma \ref{lem:vanderCorput} with \ssf${L}\msb=\ssb2$\ssf,
together with Lemma \ref{lem:phase1Hn}~(i) and \eqref{eq:EstimateGamma},
we estimate
\begin{equation*}
|\ssf k_{\ssf t,\ssf\ell}^{\ssf\sigma\ssb,\vsf0}\ssb(r)\vsf|
\lesssim(1\!+\msb\ell\ssf)^\gamma\,t^{-\frac32}\,r\,.
\end{equation*}
On the other hand, after \ssf$N$ integrations by parts based on
\begin{equation}\label{eq:IBP1}
e^{\hspace{.4mm}i\hspace{.4mm}t\msf\psi_{\vsb\frac rt\ssb}(\lambda)}\ssb
=\tfrac{-\ssf i\vphantom{|}}{t\,\psi_{\vsb\frac rt\ssb}^{\msf\prime}(\lambda)}\,
\tfrac{\partial\vphantom{|}}{\partial\lambda\vphantom{|}}\,
e^{\hspace{.4mm}i\hspace{.4mm}t\msf\psi_{\vsb\frac rt\ssb}(\lambda)}\ssf,
\end{equation}
we get
\begin{equation*}
k_{\ssf t,\ssf\ell}^{\ssf\sigma\ssb,\vsf\infty}\ssb(r)
=\bigl(\tfrac{i\vphantom{|}}t\bigr)^{\ssb N}\!
\int_{-\infty}^{\vsf+\infty}\hspace{-1.5mm}
e^{\hspace{.4mm}i\hspace{.4mm}t\msf\psi_{\vsb\frac rt\ssb}(\lambda)}\msf
a_{\vsf\ell}^{\infty}(\lambda)\msf \;\mathrm{d}\lambda,
\end{equation*}
where
\begin{equation*}
a_{\vsf\ell}^{\infty}(\lambda)
=\bigl\{\underbrace{\ssf
\tfrac\partial{\partial\lambda}\ssb\circ\ssb
\tfrac{1\vphantom{|}}{\psi_{\vsb\frac rt\ssb}^{\msf\prime}(\lambda)}
\ssb\circ\ssf\dots\ssf\circ\ssb
\tfrac\partial{\partial\lambda}\ssb\circ\ssb
\tfrac{1\vphantom{|}}{\psi_{\vsb\frac rt\ssb}^{\msf\prime}(\lambda)}
\ssf}_{N\text{ times}}\bigr\}\ssf\bigl\{\ssf
[\ssf1\!-\msb\chi_0\vsb(\lambda)\ssf]\,a_{\vsf\ell}(\lambda)
\vsf\bigr\}
\end{equation*}
is an inhomogeneous symbol of order \msf$\frac{n\vsf-1}2
\msb-\msb\sigma\!-\!N\alpha$\ssf,
according to Lemma \ref{lem:phase1Hn}~(ii), Lemma \ref{lem:phase1Hn}~(i),
\eqref{eq:Amplitude} and \eqref{eq:EstimateGamma}.
Hence
\begin{equation*}
|\ssf k_{\ssf t,\ssf\ell}^{\ssf\sigma\ssb,\vsf\infty}\ssb(r)\vsf|
\lesssim(1\!+\msb\ell\ssf)^\gamma\,t^{-N}\ssf,
\end{equation*}
provided that \msf$N\!
>\msb\frac{n\vsf+1-\vsf2\vsf\sigma}{2\vsf\alpha}$\ssf.
By taking \,$N\msb=\ssb\max\msf\bigl\{2\ssf,
\lfloor\frac{n\vsf+1-\vsf2\vsf\sigma}{2\vsf\alpha}\rfloor
\msb+\msb1\bigr\}$
\ssf and by summing up over \ssf$\ell$\ssf,
we conclude that
\begin{equation}\label{eq:EstimateKernel1}
|\ssf k_{\ssf t}^{\ssf\sigma}\ssb(r)\vsf|
\lesssim t^{-\frac32}\,r\,e^{-\rho\ssf r}
\end{equation}
when \msf$t\msb\ge\msb r\!\ge\!1$\ssf.
\smallskip

{\textit{Subcase 1.1.2\/}.}
Consider next the case where \ssf$\lambda_1\!\ge\!1$, {that is case \circled{2} in Figure~\ref{fig:RegionsAlphaLarge}}.
Given \ssf$0\msb<\msb\beta\msb<\!1$
and a bump function \ssf$\chi_1\hspace{-1mm}\in\!C_c^{\vsf\infty\ssb}(\R)$ such that
\vspace{.5mm}

\centerline{$
\chi_1\msf\equiv\msf\begin{cases}
\,1
&\text{on \,}[\ssf\beta,\ssb\tfrac1\beta\ssf]\ssf,\\
\,0
&\text{outside \,}(\frac\beta2,\ssb\tfrac2\beta\ssf)\ssf,\\
\end{cases}$}

let us now split up
\begin{equation*}
a_{\vsf\ell}(\lambda)
=\underbrace{
\chi_1(\lambda_1^{\ssb-1}\lambda)\msf a_{\vsf\ell}(\lambda)
\vphantom{\big|}}_{=\,a_\ell^1(\lambda)}\ssf
+\ssf\underbrace{
[\ssf1\!-\msb\chi_1(\lambda_1^{\ssb-1}\lambda)\ssf]
\msf a_{\vsf\ell}(\lambda)
\vphantom{\big|}}_{=\,a_\ell^\infty(\lambda)}
\end{equation*}
\vspace{-5mm}

and
\begin{equation*}
k_{\ssf t,\ssf\ell}^{\ssf\sigma}(r)
=k_{\ssf t,\ssf\ell}^{\ssf\sigma\ssb,1}\ssb(r)
+k_{\ssf t,\ssf\ell}^{\ssf\sigma\ssb,\vsf\infty}\ssb(r)
\end{equation*}
accordingly.
On the one hand,
by using Lemma \ref{lem:vanderCorput} with \ssf${L}\msb=\ssb2$\ssf,
together with Lemma \ref{lem:phase1Hn}~(i), \eqref{eq:Amplitude} and \eqref{eq:EstimateGamma},
we estimate
\begin{equation*}
|\ssf k_{\ssf t,\ssf\ell}^{\ssf\sigma\ssb,1}\ssb(r)\vsf|\ssf
\lesssim\ssf(1\!+\msb\ell\ssf)^\gamma\,t^{-\frac12}\msf
\lambda_1^{\msb\frac{n\vsf+1-\vsf\alpha}2-\vsf\sigma}
\asymp\ssf(1\!+\msb\ell\ssf)^\gamma\,
t^{-\frac{n\vsf-\vsf2\ssf\sigma}{2\ssf(\alpha\vsf-1)}}\,
r^{\ssf\frac{n\vsf-\vsf2\ssf\sigma}{2\ssf(\alpha\vsf-1)}-\frac12}\ssf.
\end{equation*}
On the other hand,
after \ssf$N$ integrations by parts based on \eqref{eq:IBP1},
\begin{equation*}
k_{\ssf t,\ssf\ell}^{\ssf\sigma\ssb,\vsf\infty}\ssb(r)
=\bigl(\tfrac{i\vphantom{|}}t\bigr)^{\ssb N}\!
\int_{-\infty}^{\vsf+\infty}\hspace{-1.5mm}
e^{\hspace{.4mm}i\hspace{.4mm}t\msf\psi_{\vsb\frac rt\ssb}(\lambda)}\msf
a_{\vsf\ell}^{\infty}(\lambda)\msf\;\mathrm{d}\lambda,
\end{equation*}
\vspace{-4mm}

with
\begin{equation*}
a_{\vsf\ell}^{\infty}(\lambda)
=\bigl\{\underbrace{\ssf
\tfrac\partial{\partial\lambda}\ssb\circ\ssb
\tfrac{1\vphantom{|}}{\psi_{\vsb\frac rt\ssb}^{\msf\prime}(\lambda)}\ssb
\circ\ssf\dots\ssf\circ\ssb
\tfrac\partial{\partial\lambda}\ssb\circ\ssb
\tfrac{1\vphantom{|}}{\psi_{\vsb\frac rt\ssb}^{\msf\prime}(\lambda)}
\ssf}_{N\text{ times}}\bigr\}\ssf\bigl\{\ssf
[\ssf1\!-\msb\chi_1(\lambda_1^{\ssb-1}\lambda)\ssf]\,
a_{\vsf\ell}(\lambda)\vsf\bigr\}\msf.
\end{equation*}
As
\begin{equation*}
|\ssf a_{\vsf\ell}^{\infty}(\lambda)\vsf|
\lesssim(1\!+\msb\ell\ssf)^\gamma\,
(\vsf|\lambda|\msb+\msb\lambda_1)^{-N(\alpha\vsf-1)}\msf
(1\!+\msb|\lambda|\vsf)^{\frac{n-1}2-\vsf\sigma-N}\,,
\end{equation*}
\vspace{-4mm}

according to Lemma \ref{lem:phase1Hn}~(ii), Lemma \ref{lem:phase1Hn}~(i),
\eqref{eq:Amplitude} and \eqref{eq:EstimateGamma},
we have
\begin{equation*}
\int_{\ssf|\lambda|\ssf\le\ssf\lambda_1}
\hspace{-1.5mm}\,|\ssf a_{\vsf\ell}^{\infty}(\lambda)\vsf|\,\mathrm{d}\lambda
\lesssim(1\!+\msb\ell\ssf)^\gamma\msf
\lambda_1^{\ssb-N\vsb(\alpha\vsf-1)}
\end{equation*}
\vspace{-3mm}

and
\begin{equation*}
\int_{\ssf|\lambda|\ssf\ge\ssf\lambda_1}
\hspace{-1.5mm}|\ssf a_{\vsf\ell}^{\infty}(\lambda)\vsf|\,\mathrm{d}\lambda
\lesssim(1\!+\msb\ell\ssf)^\gamma\msf
\lambda_1^{\msb\frac{n+1}2-\vsf\sigma-N\ssb\alpha},
\end{equation*}
provided that \msf$N\!>\msb
\frac{n+1}2\msb-\msb\sigma$\ssf.
Hence
\begin{equation*}
|\ssf k_{\ssf t,\ssf\ell}^{\ssf\sigma\ssb,\vsf\infty}\ssb(r)\vsf|
\lesssim(1\!+\msb\ell\ssf)^\gamma\,t^{\vsf-N}\ssf
\lambda_1^{\ssb-N(\alpha-1)}
\asymp\,(1\!+\msb\ell\ssf)^\gamma\msf r^{-N}\,.
\end{equation*}
By summing up over \ssf$\ell$\ssf,
we conclude that
\begin{equation}\label{eq:EstimateKernel2}
|\ssf k_{\ssf t}^{\ssf\sigma}\ssb(r)\vsf|\lesssim\ssf
t^{-\frac{n\vsf-\vsf2\ssf\sigma}{2\ssf(\alpha\vsf-1)}}\,
r^{\ssf\frac{n\vsf-\vsf2\ssf\sigma}{2\ssf(\alpha\vsf-1)}-\frac12}\,
e^{-\rho\ssf r}
\end{equation}
when \msf$t\msb>\msb0$ \ssf and
\msf$r\!\ge\hspace{-.4mm}\max\msf\{1,t\ssf\}$\ssf.
\medskip

\textbf{Subcase 1.2\/}.
Assume that \ssf$0\msb<\msb\alpha\msb<\!1$\vsf.
The analysis of \eqref{eq:KernelKtellsigmar} depends again on the size of~$\msf\frac rt\ssf$.
\smallskip

{\textit{Subcase 1.2.1\/}.}
Firstly, if \ssf$\frac rt\msb\ge\msb2\ssf\theta_0$ {(see Figure~\ref{fig:RegionsAlphaSmall}, case \circled{1})},
the phase \ssf$\psi_{\vsf\frac rt}$ has no stationary point
and, after \ssf$N$ integrations by parts based on \eqref{eq:IBP1},
\eqref{eq:KernelKtellsigmar} becomes
\vspace{-2mm}
\begin{equation*}
k_{\ssf t,\ssf\ell}^{\ssf\sigma}(r)
=\bigl(\tfrac{i\vphantom{|}}r\bigr)^{\ssb N}\!
\int_{-\infty}^{\vsf+\infty}\hspace{-1.5mm}\hspace{1mm}
e^{\hspace{.4mm}i\hspace{.4mm}t\msf\psi_{\vsb\frac rt\ssb}(\lambda)}\msf
\bigl\{\overbrace{
\ssf\tfrac\partial{\partial\lambda}\ssb\circ\ssb
\tfrac{r\vphantom{|}}{t\msf\psi_{\vsb\frac rt\ssb}^{\msf\prime}(\lambda)}\ssb
\circ\ssf\dots\ssf\circ\ssb
\tfrac\partial{\partial\lambda}\ssb\circ\ssb
\tfrac{r\vphantom{|}}{t\msf\psi_{\vsb\frac rt\ssb}^{\msf\prime}(\lambda)}
\ssf}^{N\text{ times}}\bigr\}\msf
a_{\vsf\ell}(\lambda)\msf\;\mathrm{d}\lambda,
\end{equation*}
where the amplitude is an inhomogeneous symbol of order
\vsf$\frac{n\vsf-1}2\msb-\msb\sigma\!-\!N$\vsb,
according this time to \eqref{eq:Estimate1eq_Psiprime}.
Thus
\vspace{-.5mm}

\centerline{$
|\ssf k_{\ssf t,\ssf\ell}^{\ssf\sigma}(r)\vsf|
\lesssim(1\!+\msb\ell\ssf)^\gamma\,r^{\vsf-N}\ssf,
$}\vspace{.5mm}

provided that $N\!>\msb\frac{n\vsf+1}2\msb-\msb\sigma$\vsf, and

\centerline{$
|\ssf k_{\ssf t}^{\ssf\sigma}\ssb(r)\vsf|
\lesssim\ssf r^{\vsf-N}\ssf e^{-\rho\ssf r}\ssf,
$}\vspace{1mm}

after summing up over \ssf$\ell$\ssf.
\smallskip

{\textit{Subcase 1.2.2\/}.}
Secondly, assume that
\msf$\theta_{\vsf0}/2\msb\le\msb\frac rt\msb\le\msb2\ssf\theta_0$ {(see Figure~\ref{fig:RegionsAlphaSmall}, case \circled{2})}
{and let \msf$0\msb<\msb c_{\vsf1}\!<\!1\!<\msb c_{\ssf2}\!<\!\smash{\sqrt{\ssf3\ssf}}$ \ssf such that \msf$\theta(c_{\vsf1}\lambda_{\vsf0})\msb=\ssb\theta_0/2\ssb=\ssb\theta(c_{\ssf2}\ssf\lambda_{\vsf0})$\ssf.}
Then all stationary points of the phase \ssf$\psi_{\vsf\frac rt}$ are contained in
\ssf$[\ssf c_{\vsf1}\lambda_{\vsf0}\vsf,\ssb c_{\ssf2}\ssf\lambda_{\vsf0}\ssf]$\ssf,
according to Lemma \ref{lem:phase2Hn}.
Let us split up \eqref{eq:integral_split}
in \eqref{eq:KernelKtellsigmar} and \eqref{eq:kernel_split}
accordingly, where \ssf$\chi_0\!\in\!C_c^{\vsf\infty\ssb}(\R)$ is a bump function  such that \ssf$\chi_0\msb=\!1$ on a neighborhood of
\ssf$[\ssf c_{\vsf1}\lambda_{\vsf0}\vsf,\ssb c_{\ssf2}\ssf\lambda_{\vsf0}\ssf]$ \ssf and \ssf$\supp\chi_0\msb\subset\!
(0,\msb\sqrt{\ssf3\ssf}\ssf{\lambda_0}\vsf)$\vsf.
We estimate again
\vspace{.5mm}

\centerline{$
|\ssf k_{\ssf t,\ssf\ell}^{\ssf\sigma\ssb,\vsf0}\ssb(r)\vsf|
\lesssim(1\!+\msb\ell\ssf)^\gamma\,t^{\vsf-\frac13}\ssf,
$}\vspace{.5mm}

by using Lemma \ref{lem:vanderCorput}, this time with \ssf${L}\msb=\ssb3$\ssf, and
\vspace{.5mm}

\centerline{$
|\ssf k_{\ssf t,\ssf\ell}^{\ssf\sigma\ssb,\vsf\infty\ssb}(r)\vsf|
\lesssim(1\!+\msb\ell\ssf)^\gamma\,t^{-N}\ssf,
$}\vspace{.5mm}

by performing \ssf$N$ integrations by parts based on \eqref{eq:IBP1}.
In conclusion,
\vspace{.5mm}

\centerline{$
\bigl|\ssf k_{\ssf t}^{\ssf\sigma}\ssb(r)\vsf\bigr|
\lesssim\ssf r^{\vsf-\frac13}\msf e^{-\rho\ssf r}\ssf,
$}\vspace{.5mm}

as \ssf$t$ \vsf and \ssf$r$ are comparable under the present assumptions.
\smallskip

{\textit{Subcase 1.2.3\/}.}
Thirdly, in the remaining case \ssf$0\msb<\msb\frac rt\msb<\ssb\theta_0/2$ {(see Figure~\ref{fig:RegionsAlphaSmall}, case \circled{3})},
the phase \ssf$\psi_{\vsf\frac rt}$ has two stationary points\,:
\vsf$\lambda_1\!\in\msb(\vsf0,\vsb c_{\vsf1}\lambda_{\vsf0})$ and
\vsf$\lambda_{\vsf2}\msb\in\msb(\vsf c_{\ssf2}\ssf\lambda_{\vsf0}\vsf,\ssb+\infty)\ssf$.
We shall isolate these two points by means of bump functions.
Let \ssf$\chi_0\!\in\msb C_c^{\vsf\infty\ssb}(\R)$ and
\ssf$\chi_2\!\in\msb C_c^{\vsf\infty\ssb}(\R)$ such that
\begin{equation}\label{eq:TwoBumpFunctions}
\begin{cases}
\,\chi_0\msb=\!1\text{ \,on \,}[\ssf-1,\vsb b_{\vsf\vsf1}\lambda_{\vsf0}\ssf]\\
\,\supp\chi_0\msb\subset\msb[-\ssf2\ssf,b_{\vsf\vsf2}\vsf\lambda_{\vsf0}\ssf]
\end{cases}
\quad\text{and}\qquad
\begin{cases}
\,\chi_2\ssb=\!1\text{ \msf on \,}{[\ssf b_{\vsb3}^{-1},\vsb b_{\vsb3}\ssf]}\\
\,\supp\chi_2\msb\subset\msb{[\ssf b_{\vsf4}^{-1},\vsb b_{\vsf4}\ssf]}
\end{cases}\end{equation}
where \ssf{$c_{\vsf1}\!<\ssb b_{\vsf\vsf1}\!<\ssb b_{\vsf\vsf2}\!<\msb1\msb<\ssb b_{\vsf3}\msb<\ssb b_{\vsf4}\msb<\ssb c_2$}\ssf.
Then \ssf$\chi_0$ \ssf and \ssf$\chi_2(\lambda_{\vsf2}^{-1}\msf\cdot\,)$
are smooth bump functions around \ssf$\lambda_1$ and \ssf$\lambda_{\vsf2}$ \ssf {respectively},
whose supports are disjoint and don't contain \ssf$\lambda_{\vsf0}$\ssf.
This follows indeed from the inequalities
\vspace{.5mm}

\centerline{$
b_{\vsf\vsf2}\ssf\lambda_{\vsf0}\msb<\ssb\lambda_{\vsf0}\msb<\ssb{b_{\vsf4}^{-1}\vsf c_{\vsf2}}\ssf\lambda_{\vsf0}\msb<\ssb{b_{\vsf4}^{-1}}\vsf\lambda_{\vsf2}\ssf.
$}\vspace{-.5mm}


Let us split up
\begin{equation*}
\int_{-\infty}^{\vsf+\infty}\hspace{-2mm}\;\mathrm{d}\lambda\,
=\int_{-\infty}^{\vsf+\infty}\hspace{-2mm}\chi_0(\lambda)\;\mathrm{d}\lambda
+\int_{-\infty}^{\vsf+\infty}\hspace{-2mm}\chi_2(\lambda_{\vsf2}^{\ssb-1\ssb}\lambda)\;\mathrm{d}\lambda
+\int_{-\infty}^{\vsf+\infty}\hspace{-2mm}[\vsf1\!-\hspace{-.6mm}\chi_0(\lambda)\msb
-\hspace{-.6mm}\chi_2(\lambda_{\vsf2}^{\ssb-1\ssb}\lambda)\ssf]\;\mathrm{d}\lambda
\end{equation*}
in \eqref{eq:KernelKtellsigmar} and
\begin{equation}\label{eq:ThreeTermsSplitting}
k_{\ssf t,\ssf\ell}^{\ssf\sigma}(r)
=\ssf k_{\ssf t,\ssf\ell}^{\ssf\sigma\ssb,\vsf0}\ssb(r)
+\ssf k_{\ssf t,\ssf\ell}^{\ssf\sigma\ssb,\vsf2}\ssb(r)
+\ssf k_{\ssf t,\ssf\ell}^{\ssf\sigma\ssb,\vsf\infty}\ssb(r)
\end{equation}
accordingly.
We estimate each term as we did in Subcase 1.1,
using Lemma \ref{lem:phase2Hn} instead of Lemma \ref{lem:phase1Hn}.
This way, we obtain
\begin{equation}\label{eq:ThreeEstimates}\begin{cases}
\,|\ssf k_{\ssf t,\ssf\ell}^{\ssf\sigma\ssb,\vsf0}\ssb(r)\vsf|
\lesssim(1\!+\msb\ell\ssf)^\gamma\,t^{\vsf-\frac32}\,r\msf,\\
\,|\ssf k_{\ssf t,\ssf\ell}^{\ssf\sigma\ssb,\vsf2}\ssb(r)\vsf|
\lesssim(1\!+\msb\ell\ssf)^\gamma\,
t^{-\frac{2\ssf\sigma-\vsf n}{2\ssf(1-\vsf\alpha)}}\,
r^{\ssf\frac{2\ssf\sigma-\vsf n}{2\ssf(1-\vsf\alpha)}-\frac12}\ssf,\\
\,|\ssf k_{\ssf t,\ssf\ell}^{\ssf\sigma\ssb,\vsf\infty}\ssb(r)\vsf|
\lesssim(1\!+\msb\ell\ssf)^\gamma\,t^{\vsf-N}\ssf,
\end{cases}\end{equation}
provided that \ssf$\alpha\ssf N\!>\msb\frac{n\vsf+1}2\msb-\msb\sigma$, hence
\begin{equation}\label{eq:Estimateeq_KernelSigma}
|\ssf k_{\ssf t}^{\ssf\sigma}\ssb(r)\vsf|
\lesssim\bigl(\tfrac rt\bigr)^{
\min\ssf\left\{\ssb\frac32,\ssf\frac{2\ssf\sigma-\vsf n}{2\ssf(1-\vsf\alpha)}\ssb\right\}}\,
r^{\vsf-\vsb\frac12}\,e^{-\rho\ssf r}\ssf.
\end{equation}

\begin{remark}\label{rm:RemarkLowerBoundOnR}
All results so far, which have been proved under the assumption \,$r\!\ge\!1$\vsf,
hold actually for \,$r\!\ge\msb r_0$ with \,$r_0\msb>\msb0$ fixed\ssf.
\end{remark}

\subsubsection{Case 2 - small spatial scale} \label{sect:Case2}
If \ssf$r$ \vsf is bounded above, let us say by $1$\vsf,
we use two expressions of the spherical functions  \ssf$\phi_\lambda(r)$\vsf, {namely Harish-Chandra integral formula \eqref{eq:HCformula}, with} \ssf$H(a_{\ssf\pm\ssf r}\ssf k)\msb\in\msb[-r,r\ssf]$ {and Stanton-Tomas-Ionescu formula \eqref{eq:STI} (see also Lemma \ref{smallscaleExp}).}\\

\textbf{Subcase 2.1}.
Assume that \vsf$1\!<\ssb\alpha\ssb<\ssb2$ {\ssf and \ssf$0\ssb\le\ssb r\msb\le\msb1$\vsf, \ssf$t\msb>\msb0$ \ssf (see Figure~\ref{fig:RegionsAlphaLarge}, cases \circled{3} -- \circled{4} -- \circled{5}).}
\smallskip

{\textit{Subcase 2.1.1}. Consider first the range} \ssf$r\msb\le \ssb t^{\vsf\frac1\alpha}$ {(see Figure~\ref{fig:RegionsAlphaLarge}, {cases \circled{3} -- \circled{4})}}.
By using \eqref{eq:HCformula},
\eqref{eq:KernelSigma} becomes
\begin{equation}\label{eq:KernelSigmaAverage}
k_{\ssf t}^{\ssf\sigma}\ssb(r)
=\ssf\const\int_K\!\hspace{1mm}e^{\vsf-\ssf\rho\ssf H(a_{\vsf-r}k)}\,
\tilde{k}_{\ssf t}^{\ssf\sigma}\ssb(H(a_{\vsf-r}\ssf k)\ssb)\msf\; \mathrm{d}k,
\end{equation}
where
\begin{equation}\label{eq:KernelSigmaH}
\tilde{k}_{\ssf t}^{\ssf\sigma}\ssb(H)
=\!\int_{-\infty}^{\vsf+\infty}\hspace{-1mm}\,
|\vsf\bc(\lambda)|^{-2}\,(\lambda^{\vsb2}\!+\msb\rho^{\ssf2}\vsf)^{-\frac\sigma2}\,
e^{\ssf i\ssf t\ssf(\lambda^2+\ssf\rho^2)^{\vsf\alpha\vsb/2}\vsf-\ssf i\ssf\lambda\ssf H}\ssf\; \mathrm{d}\lambda.
\end{equation}
We estimate \eqref{eq:KernelSigmaH} when \ssf$H\!\in\msb[-r,\vsb r\ssf]$
\ssf by resuming the analysis in Subcase 1.1
and by dealing separately with the cases \ssf$|H|\msb\le\msb1\!\le\ssb t$ \vsf
and \ssf$|H|\msb\le\ssb t^{\vsf\frac1\alpha}\!\le\msb1$\vsf.
Let us elaborate.
\smallskip

$\bullet$
\,If \ssf$|H|\msb\le\msb1\!\le\ssb t$ {(see Figure~\ref{fig:RegionsAlphaLarge}, case \circled{3})},
the stationary point \ssf{$\lambda_1$}
of the phase \eqref{eq:phaseHn},
with \ssf$R=\ssb\tfrac Ht\msb\in\msb[-1,\vsb1\vsf]$\ssf,
{remains bounded, according to Lemma \ref{lem:phase1Hn},
say \ssf$|\lambda_1|\ssb\le\ssb c$ \ssf for some constant \ssf$c\ssb>\ssb0$\ssf.}
Let us split up
\vspace{1mm}

\centerline{$\displaystyle
\int_{-\infty}^{\vsf+\infty}\hspace{-1.5mm}\; \mathrm{d}\lambda\,
=\int_{-\infty}^{\vsf+\infty}\hspace{-1.5mm}\hspace{1mm}
\chi\bigl(\vsb\tfrac\lambda{2\ssf{c}}\ssb\bigr)\; \mathrm{d}\lambda
+\int_{-\infty}^{\vsf+\infty}\hspace{-1.5mm}\,
\bigl[\ssf1\!-\msb\chi\bigl(\vsb\tfrac\lambda{2\ssf{c}}\ssb\bigr)\bigr]\; \mathrm{d}\lambda
$}
in \eqref{eq:KernelSigmaH} and
\begin{equation}\label{eq:KernelDecomposition1}
\tilde{k}_{\ssf t}^{\ssf\sigma}\ssb(H)
=\ssf\tilde{k}_{\ssf t}^{\ssf\sigma\ssb,\vsf0}(H)
+\ssf\tilde{k}_{\ssf t}^{\ssf\sigma\ssb,\vsf\infty}(H)
\end{equation}
accordingly,
where \ssf$\chi\!\in\msb C_c^{\vsf\infty\ssb}(\R)$ \vsf
is an even bump function  such that
\vspace{1mm}

\centerline{$
0\msb\le\msb\chi\msb\le\msb1\msf,\quad
\chi\msb=\!1\msf\text{ on }\,[-\vsb1,\ssb1\vsf]\,,\quad
\chi\msb=\msb0\,\text{ outside of }\,(\ssb-2,\ssb2\vsf)\msf.
$}\vspace{1mm}
On the one hand,
after an integration by parts based on \eqref{eq:IBP0},
the first term in \eqref{eq:KernelDecomposition1} becomes

\centerline{$\displaystyle
\tilde{k}_{\ssf t}^{\ssf\sigma\ssb,\vsf0}(H)\vsb
=C\msf i\,\tfrac Ht\ssb\int_{-\infty}^{\vsf+\infty}\hspace{-1mm}\hspace{1mm}
e^{\hspace{.4mm}i\hspace{.4mm}t\msf\psi_{\ssb\frac Ht}\msb(\lambda)}\msf a_{\vsf0}(\lambda)\;\mathrm{d}\lambda
-\ssf\tfrac Ct\ssb\int_{-\infty}^{\vsf+\infty}\hspace{-1mm}\hspace{1mm}
e^{\hspace{.4mm}i\hspace{.4mm}t\msf\psi_{\ssb\frac Ht}\msb(\lambda)}\msf
\tfrac\partial{\partial\lambda}\msf a_{\vsf0}(\lambda)\msf \;\mathrm{d}\lambda,
$}

where

\centerline{$\displaystyle
a_{\vsf0}(\lambda)\vsb
=\chi\bigl(\vsb\tfrac\lambda{2\ssf{c}}\ssb\bigr)\msf
\tfrac{\Gamma(i\vsf\lambda\vsf+\vsf\rho\vsf)}{\Gamma(i\vsf\lambda\vsf+1)}\msf
\tfrac{\Gamma(-\vsf i\vsf\lambda\vsf+\vsf\rho\vsf)}{\Gamma(-\vsf i\vsf\lambda)}\msf
(\lambda^{\vsb2}\!+\msb\rho^{\vsf2}\vsf)^{1-\frac\alpha2-\frac\sigma2}
$}\vspace{1mm}

is a smooth function with compact support.
By using Lemma \ref{lem:vanderCorput} with \ssf${L}\msb=\ssb2$\ssf,
we deduce that
\vspace{-1mm}
\begin{equation}\label{eq:EstimateKernelTildeSigma0Tlarge}
|\ssf\tilde{k}_{\ssf t}^{\ssf\sigma\ssb,\vsf0}(H)\vsf|
\lesssim t^{\vsf-\frac 32}\msf.
\end{equation}
On the other hand, after \ssf$N$ integrations by parts based on
\begin{equation}\label{eq:IBP2}
e^{\hspace{.4mm}i\hspace{.4mm}t\msf\psi_{\ssb\frac Ht}\msb(\lambda)}\ssb
=\tfrac{-\ssf i\vphantom{|}}{t\,\psi_{\msb\frac Ht}^{\msf\prime}\msb(\lambda)}\,
\tfrac{\partial\vphantom{|}}{\partial\lambda\vphantom{|}}\,
e^{\hspace{.4mm}i\hspace{.4mm}t\msf\psi_{\ssb\frac Ht}\msb(\lambda)}\ssf,
\end{equation}
\vspace{-4mm}

the second term in \eqref{eq:KernelDecomposition1} becomes

\centerline{$\displaystyle
\tilde{k}_{\ssf t}^{\ssf\sigma\ssb,\vsf\infty}(H)\vsb
=C\msf\bigl(\tfrac{i\vphantom{|}}t\bigr)^{\ssb N}\!
\int_{-\infty}^{\vsf+\infty}\hspace{-1.5mm}\,
e^{\hspace{.4mm}i\hspace{.4mm}t\msf\psi_{\ssb\frac Ht}\msb(\lambda)}\msf
a_{\vsf\infty}(\lambda)\msf\; \mathrm{d}\lambda,
$}

where
\vspace{1.5mm}

\centerline{$\displaystyle
a_{\vsf\infty}(\lambda)
=\bigl\{\underbrace{\ssf
\tfrac\partial{\partial\lambda}\ssb\circ\ssb
\tfrac{1\vphantom{|}}{\psi_{\msb\frac Ht}^{\msf\prime}\msb(\lambda)}
\ssb\circ\ssf\dots\ssf\circ\ssb
\tfrac\partial{\partial\lambda}\ssb\circ\ssb
\tfrac{1\vphantom{|}}{\psi_{\msb\frac Ht}^{\msf\prime}\msb(\lambda)}
\ssf}_{N\text{ times}}\bigr\}\ssf\bigl\{
\bigl[\vsf1\!-\msb\chi\bigl(\vsb\tfrac\lambda{2\ssf{c}}\ssb\bigr)\bigr]\,
|\vsf\bc(\lambda)|^{-2}\,
(\lambda^{\vsb2}\!+\msb\rho^{\ssf2}\vsf)^{-\frac\sigma2}
\vsf\bigr\}
$}\vspace{1.5mm}

is an inhomogeneous symbol of order \msf$n\msb-\!1\hspace{-.8mm}-\msb\sigma\!-\!N\alpha$\ssf,
according to Lemma \ref{lem:phase1Hn}.
Hence
\begin{equation}\label{eq:EstimateKernelTildeSigmaInftyTlarge}
|\ssf\tilde{k}_{\ssf t}^{\ssf\sigma\ssb,\vsf\infty}(H)\vsf|
\lesssim t^{\vsf-N}\ssf,
\end{equation}
provided that \msf$N\!>\msb\frac{n\ssf-\ssf\sigma}\alpha$\ssf.
By taking \,$N\msb=\ssb\max\msf\bigl\{2\ssf,
\lfloor\frac{n\ssf-\ssf\sigma}\alpha\rfloor\msb+\msb1\bigr\}$\ssf,
we obtain finally the bound \msf$\text{O}\bigl(\vsf t^{\ssf-\frac 32}\bigr)$ \ssf
for \eqref{eq:KernelSigmaH} and hence {
\begin{equation}\label{eq:EstimateKernel3}
|\ssf k_{\ssf t}^{\ssf\sigma}\ssb(r)\vsf|
\lesssim t^{-\frac32}\,e^{-\rho\ssf r}
\end{equation}
}when \ssf$r\msb\le\msb1\!\le\ssb t$\ssf.
\smallskip

$\bullet$
\,We proceed similarly in the case
\ssf$|H|\msb\le\ssb t^{\vsf\frac1\alpha}\!\le\msb1$ {(see Figure~\ref{fig:RegionsAlphaLarge}, {case \circled{4}})},
with the following few differences.
{T}he stationary point \ssf{$\lambda_1$} of the phase \eqref{eq:phaseHn}, with \ssf$R\ssb=\msb\tfrac Ht\msb\in\msb[-\ssf t^{\vsf-\frac{\alpha-1}\alpha},t^{\vsf-\frac{\alpha-1}\alpha}\vsf]\ssf,$ {satisfies now \ssf$|\lambda_1|\ssb\le\ssb c\ssf t^{-1/\alpha}$, for some constant \ssf$c\ssb>\ssb0$}\ssf. {After splitting} up
\vspace{1mm}

\centerline{$\displaystyle
\int_{-\infty}^{\vsf+\infty}\!\mathrm{d}\lambda\,
=\int_{-\infty}^{\vsf+\infty}\hspace{-1mm}
\chi\bigl(\vsb\tfrac{t^{\vsf1\vsb\vsb/\ssb\alpha}\lambda}{2\ssf{c}}\ssb\bigr)
\;\mathrm{d}\lambda
+\int_{-\infty}^{\vsf+\infty}\bigl[\ssf1\!-\msb
\chi\bigl(\vsb\tfrac{t^{\vsf1\vsb\vsb/\ssb\alpha}\lambda}{2\ssf{c}}\bigr)\bigr]\;\mathrm{d}\lambda$}\vspace{1mm}

in \eqref{eq:KernelSigmaH}, the contribution of the first integral is estimated easily, while the contribution of the second integral is handled as above. Specifically,
as \ssf$|\vsf\bc(\lambda)|^{-2}\ssb\lesssim\ssb(1\!+\msb\lambda)^{n\vsf-1}$\vsf,
we have on the one hand
\begin{equation}\label{eq:EstimateKernelTildeSigma0Tsmall}
|\ssf\tilde{k}_{\ssf t}^{\ssf\sigma\ssb,\vsf0}(H)\vsf|
\lesssim\!\int_{\vsf|\lambda|\vsf\le\ssf4\ssf{c}\ssf t^{-\vsb1\vsb/\ssb\alpha}}
\hspace{-1mm}\,(1\!+\msb\lambda)^{n\vsf\vsf-1-\ssf\sigma}\;\mathrm{d}\lambda
\lesssim\ssf t^{\ssf-\frac{n\vsf\vsf-\ssf\sigma}\alpha}\ssf.
\end{equation}
On the other hand, after \ssf$N$ integrations by parts based on \eqref{eq:IBP2},
with \ssf$N\!>\msb\tfrac{n\vsf\vsf-\ssf\sigma}\alpha$\ssf, we get
\begin{equation}\label{eq:EstimateKernelTildeSigmaInftyTsmall}
|\ssf\tilde{k}_{\ssf t}^{\ssf\sigma\ssb,\vsf\infty}(H)\vsf|
\lesssim\ssf t^{\vsf-N}\!
\int_{\vsf|\lambda|\vsf\vsf\ge\ssf2\ssf{c}\ssf t^{-\vsb1\vsb/\ssb\alpha}}
\hspace{-1mm}\,|\lambda|^{n-1-\sigma-\vsb N\alpha}\; \mathrm{d}\lambda
\lesssim\ssf t^{\ssf-\frac{n\vsf-\ssf\sigma}\alpha}\ssf.
\end{equation}
In conclusion, \eqref{eq:KernelSigmaH} and hence \eqref{eq:KernelSigma} are
\ssf$\text{O}\bigl(\ssf t^{\ssf-\frac{n\vsf\vsf-\ssf\sigma}\alpha}\vsf\bigr)$\vsf,
when \ssf$r\msb\le\ssb t^{\vsf\frac1\alpha}\!\le\msb1$\vsf.
\smallskip

{\textit{Subcase 2.1.2}. Consider next the range \ssf$0\msb<\ssb t^{\vsf\frac1\alpha}\!\le\ssb r\msb\le\!1$ (see Figure~\ref{fig:RegionsAlphaLarge}, case \circled{5}).}
According to Lemma \ref{lem:phase1Hn},
there exists \ssf${c}\ssb>\msb0$ \ssf such that
the stationary point of the phase \eqref{eq:phaseHn},
with \ssf$R\ssb=\ssb\frac rt\msb\ge\msb1$\vsf,
satisfies
\ssf$\lambda_1\msb>\ssb{c}\,(\frac rt)^{1/(\alpha\ssf-1)}$\vsf.
Let us split up
\begin{equation}\label{eq:SplittedIntegral}
\int_{-\infty}^{\vsf+\infty}\hspace{-1.5mm}\mathrm{d}\lambda\,
=\int_{-\infty}^{\vsf+\infty}\hspace{-1.5mm}\hspace{1mm}\chi_0(\vsf r\vsf\lambda)\,\mathrm{d}\lambda\;
+\int_{-\infty}^{\vsf+\infty}\hspace{-1.5mm}\hspace{1mm}\chi_1\vsb(\lambda_1^{-1}\lambda)\,\mathrm{d}\lambda\;
+\int_{-\infty}^{\vsf+\infty}\hspace{-1.5mm}\hspace{1mm}\chi_\infty(\lambda)\; \mathrm{d}\lambda
\end{equation}
in \eqref{eq:KernelSigma} and
\begin{equation}\label{eq:KernelDecomposition2}
k_{\ssf t}^{\ssf\sigma}(r)
=\ssf k_{\ssf t}^{\ssf\sigma\ssb,\vsf0}\ssb(r)
+\ssf k_{\ssf t}^{\ssf\sigma\ssb,1}\ssb(r)
+\ssf k_{\ssf t}^{\ssf\sigma\ssb,\vsf\infty\ssb}(r)
\end{equation}
accordingly, where \ssf$\chi_0\vsf,\vsb\chi_1\!\in\!C_c^{\vsf\infty\ssb}(\R)$
are even functions such that
\vspace{1mm}

\centerline{$\begin{cases}
\,0\msb\le\msb\chi_0\msb\le\!1\vsf,\;
\chi_0\msb=\!1\vsf\text{ on }\ssf[-\frac{c}8,\ssb\frac{c}8]\ssf,\;
\supp\chi_0\ssb\subset\ssb[-\frac{c}4,\ssb\frac{c}4]\ssf,\\
\,0\msb\le\msb\chi_1\!\le\!1\vsf,\;
\chi_1\!=\!1\vsf\text{ on }\ssf
[-2\vsf,\ssb-\frac12\vsf]\ssb\cup\ssb[\vsf\frac12,\ssb2\ssf]\ssf,\;
\supp\chi_1\msb\subset\ssb
[-\vsf4\vsf,\ssb-\frac14\ssf]\ssb\cup\ssb[\vsf\frac14,\ssb4\ssf]\ssf,
\end{cases}$}\vspace{1mm}

and \ssf$\chi_\infty(\lambda)\msb
=\msb1\!-\ssb\chi_0(\vsf r\vsf\lambda)\msb-\ssb\chi_1\vsb(\lambda_1^{-1}\lambda)$\ssf.
Notice that the cutoff functions \ssf$\chi_0(\vsf r\,\cdot\msf)$ \vsf
and \ssf$\chi_1\vsb(\lambda_1^{-1}\msf\cdot\msf)$ \vsf have disjoint supports.
By using \ssf$|\vsf\bc(\lambda)|^{-2}\ssb\lesssim\ssb(1\!+\msb\lambda)^{n\vsf-1}$
\vsf and \ssf$|\ssf\phi_\lambda(r)\vsf|\msb\le\msb1$\vsf,
we estimate easily
\vspace{-2mm}

\centerline{$\displaystyle
|\ssf k_{\ssf t}^{\ssf\sigma\ssb,\vsf0}(r)\vsf|
\lesssim\!\int_{\vsf|\lambda|\ssf\le\ssf c_1\vsb/4\ssf r}
\hspace{-1mm}\,(1\!+\msb\lambda)^{n\vsf-\ssf\sigma\vsf-1}
\;\mathrm{d}\lambda\lesssim\ssf r^{-(n-\sigma)}\msf.
$}\vspace{1mm}

Let us turn to the last two terms in \eqref{eq:KernelDecomposition2}.
By using the small scale asymptotics \eqref{eq:STI}
with $\Lambda\msb=\msb\frac{c}8$,
we obtain the expressions
\begin{equation}\label{eq:KernelSigmaOne}\begin{aligned}
k_{\ssf t}^{\ssf\sigma\ssb,1}\ssb(r)
&=\ssf{2}\,C\msb\overbrace{\int_{-\infty}^{\vsf+\infty}\hspace{-1.5mm}\hspace{1mm}
\chi_1\vsb(\lambda_1^{-1}\lambda)\,|\ssf\bc(\lambda)\vsf|^{-2}\,
(\lambda^{\vsb2}\!+\msb\rho^{\ssf2}\vsf)^{-\frac\sigma2}\,
b_{\vsf M}(-\lambda,r)\,
e^{\hspace{.4mm}i\hspace{.4mm}t\msf\psi_{\vsb\frac rt\ssb}(\lambda)}\;\mathrm{d}\lambda}^{I\vphantom{g}}\\
&\ssf+\ssf C\msb\underbrace{\int_{-\infty}^{\vsf+\infty}\hspace{-1.5mm}\hspace{1mm}
\chi_1\vsb(\lambda_1^{-1}\lambda)\,|\ssf\bc(\lambda)\vsf|^{-2}\,
(\lambda^{\vsb2}\!+\msb\rho^{\ssf2}\vsf)^{-\frac\sigma2}\,
R_{\ssf M}(\lambda,r)\,
e^{\hspace{.4mm}i\hspace{.4mm}t\ssf(\lambda^2+\ssf\rho^2\ssf)^{\alpha/2}}\;\mathrm{d}\lambda
}_{II\vphantom{0^0}}
\end{aligned}\end{equation}
and
\begin{equation}\label{eq:KernelSigmaInfty}\begin{aligned}
k_{\ssf t}^{\ssf\sigma\ssb,\infty}\ssb(r)
&=\ssf{2}\,C\msb\overbrace{\int_{-\infty}^{\vsf+\infty}\hspace{-1.5mm}\hspace{1mm}
\chi_\infty(\lambda)\,|\ssf\bc(\lambda)\vsf|^{-2}\,
(\lambda^{\vsb2}\!+\msb\rho^{\ssf2}\vsf)^{-\frac\sigma2}\,
b_{\vsf M}(-\lambda,r)\,
e^{\hspace{.4mm}i\hspace{.4mm}t\msf\psi_{\vsb\frac rt\ssb}(\lambda)}\;\mathrm{d}\lambda
}^{I\ssb I\ssb I\vphantom{g}}\\
&\ssf+C\msb\underbrace{\int_{-\infty}^{\vsf+\infty}\hspace{-1.5mm}\hspace{1mm}
\chi_\infty(\lambda)\,|\ssf\bc(\lambda)\vsf|^{-2}\,
(\lambda^{\vsb2}\!+\msb\rho^{\ssf2}\vsf)^{-\frac\sigma2}\,
R_{\ssf M}(\lambda,r)\,
e^{\hspace{.4mm}i\hspace{.4mm}t\ssf(\lambda^2+\ssf\rho^2\ssf)^{\alpha/2}
}\;\mathrm{d}\lambda}_{IV\vphantom{0^0}}.
\end{aligned}\end{equation}
\vspace{-1mm}

The main contribution arises from the first integral,
which is estimated by using Lemma \ref{lem:vanderCorput} with \ssf${L}\msb=\ssb2$\ssf,
together with Lemma \ref{lem:phase1Hn}~(i) and \eqref{eq:EstimateBM}.
This way we obtain
\vspace{.5mm}

\centerline{$
|\vsf I\vsf|\vsf\lesssim\ssf t^{\vsf-\frac12}\,r^{\vsf-\frac{n\vsf-1}2}\msf
\lambda_1^{\ssb\frac{n\vsf-\vsf2\vsf\vsf\sigma}2\vsf-\vsf\frac{\alpha\vsf-1}2}
\asymp\ssf t^{\vsf-\frac{n\vsf-\vsf2\vsf\vsf\sigma}{2\ssf(\alpha\vsf-1)}}\,
r^{\vsf\frac{n\vsf-\vsf2\vsf\vsf\sigma}{2\ssf(\alpha\vsf-1)}\vsf-\vsf\frac n2}\ssf.
$}\vspace{1mm}

On the other hand, by using \eqref{eq:EstimateRM}
with \vsf$M\msb>\msb\frac{\ssf n\ssf+\vsf\vsf1}2\msb+\ssb\sigma$\ssf,
we estimate
\begin{align*}
|\vsf I\msb I\vsf|\vsf
&\lesssim\ssf r^{\vsf-\frac{n\vsf-1}2\vsf-M}\!
\int_{\frac14\lambda_1\le\ssf|\lambda|\ssf\le\ssf4\ssf\lambda_1}
\hspace{-1.5mm}\,|\lambda|^{\frac{n-1}2-\ssf\sigma\vsf-M}\; \mathrm{d}\lambda\\
&\asymp\ssf r^{\vsf-\frac{n\vsf-1}2\vsf-M}\msf
\lambda_1^{\ssb\frac{n\vsf-\vsf2\vsf\vsf\sigma}2\vsf-\vsf\frac{2\vsf M\vsb\vsb-1}2}
\asymp\underbrace{\bigl(\tfrac t{r^{\vsf\alpha}}\bigr)^{\ssb\frac{2\vsf M\vsb\vsb-1}{2\ssf(\alpha\vsf-1)}}}_{\le\ssf1}
\ssf t^{\vsf-\frac{n\vsf-\vsf2\vsf\vsf\sigma}{2\ssf(\alpha\vsf-1)}}\,
r^{\vsf\frac{n\vsf-\vsf2\vsf\vsf\sigma}{2\ssf(\alpha\vsf-1)}\vsf-\vsf\frac n2}
\end{align*}
\vspace{-5mm}
and
\vspace{-.5mm}
\begin{equation*}
|\vsf I\vsb V|\vsf\lesssim\ssf r^{\vsf-\frac{n\vsf-1}2\vsf-M}\!
\int_{\vsf|\lambda|\ssf\ge\ssf c_1\vsb/\vsf8\ssf r}
\hspace{-1mm}\,|\lambda|^{\frac{n-1}2-\ssf\sigma\vsf-M} \;\mathrm{d}\lambda\,
\asymp\ssf r^{\vsf-\frac{n\vsf-1}2\vsf-M}\,
r^{\vsf-\frac{n\vsf+1}2+\ssf\sigma\vsf+M}
=\ssf r^{-(n-\sigma)}\ssf,
\end{equation*}
Finally, the third integral in \eqref{eq:KernelSigmaInfty} is estimated
by performing \vsf$N$ integrations by parts based on \eqref{eq:IBP1},
with \ssf$N\!>\msb\tfrac{n\vsf+1}2\msb-\ssb\sigma\vsf$,
by using Lemma \ref{lem:phase1Hn}, together with \eqref{eq:EstimateBM},
and by splitting up the integral

\centerline{$\displaystyle
\int_{-\infty}^{\vsf+\infty}\hspace{-1.5mm}\mathrm{d}\lambda\ssf
=\!\int_{\vsf|\lambda|\vsf\vsf<\ssf\lambda_1}\hspace{-1.5mm}\mathrm{d}\lambda\ssf
+\!\int_{\vsf|\lambda|\ssf\ge\ssf\lambda_1}\hspace{-1.5mm}\mathrm{d}\lambda\,.
$}

This way we obtain
\vspace{-5mm}
\begin{align*}
|\vsf I\msb I\msb I\vsf|
&\lesssim\ssf t^{\vsf-N}\ssf r^{\vsf-\frac{n\vsf-1}2}\ssf\lambda_1^{-(\alpha-\vsb1)N}\!\overbrace{
\int_{\frac{c_1}{8\vsf r}\vsf\le\ssf|\lambda|\vsf\vsf<\ssf\lambda_1}\hspace{-1.5mm}\,
|\lambda|^{\frac{n-1}2-\ssf\sigma\vsf-N}\;\mathrm{d}\lambda}^{\lesssim\,r^{\vsf N+\vsf\sigma-\frac{n+\vsb1}2}\vphantom{\frac00}}
+\,t^{\vsf-N}\ssf r^{\vsf-\frac{n\vsf-1}2}\!\overbrace{
\int_{\vsf|\lambda|\ssf\ge\ssf\lambda_1}\hspace{-1.5mm}\,
|\lambda|^{\frac{n-1}2-\ssf\sigma\vsf-\alpha\vsf N}
\;\mathrm{d}\lambda}^{\asymp\,\lambda_1^{\ssb\frac{n+\vsb1}2-\vsf\sigma-\vsf\alpha\vsf N}\vphantom{\frac00}}\; \\
&\lesssim\ssf r^{-(n-\sigma)}
+\underbrace{\bigl(\tfrac t{r^{\vsf\alpha}}\bigr)^{\frac{2\vsf N\vsb\vsb-1}{2\vsf(\alpha\vsf-1)}}}_{\le\ssf1}
\msf t^{\vsf-\frac{n\vsf-\vsf2\vsf\vsf\sigma}{2\ssf(\alpha\vsf-1)}}\,
r^{\vsf\frac{n\vsf-\vsf2\vsf\vsf\sigma}{2\ssf(\alpha\vsf-1)}\vsf-\vsf\frac n2}\ssf.
\end{align*}
As
\begin{equation*}
r^{-(n-\sigma)}
=(\vsf r^{\ssf\alpha})^{-\frac1{\alpha\vsf-1}\vsf\left(\ssb\frac n2-\ssf\sigma\ssb\right)}\,
r^{{+}\frac1{\alpha\vsf-1}\vsf\left(\ssb\frac n2-\ssf\sigma\ssb\right)-\frac n2}
\le\ssf t^{\vsf-\frac{n\vsf-\vsf2\vsf\vsf\sigma}{2\ssf(\alpha\vsf-1)}}\,
r^{\vsf\frac{n\vsf-\vsf2\vsf\vsf\sigma}{2\ssf(\alpha\vsf-1)}\vsf-\vsf\frac n2}\msf,
\end{equation*}
we conclude that

\centerline{$\displaystyle
|\ssf k_{\ssf t}^{\ssf\sigma}(r)\vsf|
\lesssim\ssf t^{\vsf-\frac{n\vsf-\vsf2\vsf\vsf\sigma}{2\ssf(\alpha\vsf-1)}}\,
r^{\vsf\frac{n\vsf-\vsf2\vsf\vsf\sigma}{2\ssf(\alpha\vsf-1)}\vsf-\vsf\frac n2}\ssf.
$}\medskip

\textbf{Subcase 2.2}.
Assume that \ssf$0\ssb<\ssb\alpha\ssb<\msb1$ {\ssf and \ssf$0\ssb\le\ssb r\msb\le\msb1$\vsf, \ssf$t\msb>\msb0$ \ssf (see Figure~\ref{fig:RegionsAlphaSmall}, cases \circled{4} -- \circled{5} -- \circled{6}).}
\smallskip

{\textit{Subcase 2.2.1\/} (range \circled{4} in Figure~\ref{fig:RegionsAlphaSmall}).}
Assume that \ssf$0\ssb\le\ssb r\msb<\msb1\msb\le t$\ssf,
hence \vsf$\frac rt\msb<\msb1$\vsf.
Recall that \eqref{eq:theta} reaches its maximum \ssf$\theta_0$ \vsf at
\ssf$\lambda_0\ssb=\msb\tfrac{\rho}{\sqrt{1\vsf-\ssf\alpha}}\msb>\ssb0$\ssf.
As in Subcase 1.2,
{let \msf$0\msb<\msb c_{\vsf1}\!<\!1\!<\msb c_{\ssf2}\!<\!\smash{\sqrt{\ssf3\ssf}}$ \ssf such that \msf$\theta(c_{\vsf1}\lambda_{\vsf0})\msb=\ssb\theta_0/2\ssb=\ssb\theta(c_{\ssf2}\ssf\lambda_{\vsf0})$\ssf.}
We may assume that $\frac rt\msb<\msb\smash{\theta_0/2}$
by reducing the range \ssf$0\ssb\le\ssb r\msb<\msb r_0$\ssf,
according to Remark \ref{rm:RemarkLowerBoundOnR}.
Let us estimate $k_{\ssf t}^{\ssf\sigma}(r)$ by considering again several cases.
\smallskip

$\bullet$
\,Assume first that \ssf$r\msb=\ssb0$\ssf.
According to Lemma \ref{lem:phase2Hn}~(iv),
the phase \eqref{eq:phaseHn}, with \ssf$R\msb=\ssb0$\ssf,
has a single stationary point at the origin.
{Given an even bump function \ssf$\chi_0\msb\in\msb C_c^{\ssf\infty}(\R)$ \ssf such that
\begin{equation*}\begin{cases}
\,\chi_0\msb=\!1\text{ \,on \,}
[\ssf-\ssf b_{\vsf\vsf1}\lambda_{\vsf0}\ssf,b_{\vsf\vsf1}\lambda_{\vsf0}\ssf]\\
\,\supp\chi_0\msb\subset\msb
[-\ssf b_{\vsf\vsf2}\lambda_{\vsf0}\ssf,b_{\vsf\vsf2}\vsf\lambda_{\vsf0}\ssf]
\end{cases}\end{equation*}
where \ssf$c_{\vsf1}\!<\ssb b_{\vsf\vsf1}\!<\ssb b_{\vsf\vsf2}\!<\msb1$\vsf,}
let us split up {the integral} {as in \eqref{eq:integral_split}}
{and the kernel}
\begin{equation}\label{KernelSplittedInTwo}
k_{\ssf t}^{\ssf\sigma}(0)=k_{\ssf t}^{\ssf\sigma,0}(0)+k_{\ssf t}^{\ssf\sigma,\infty}(0)
\end{equation}
accordingly. On the one hand, after an integration by parts based on \eqref{eq:IBP0}, we obtain
\begin{equation*}
k_{\ssf t}^{\ssf\sigma,0}(0)
=\tfrac{i\msf C}\alpha\msf t^{-1}\int_{-\infty}^{\vsf+\infty}\hspace{-1mm}
e^{\hspace{.4mm}i\hspace{.4mm}t\hspace{.4mm}(\lambda^{\vsb2}
+\ssf\rho^{\vsf2}\vsf)^{\vsf\alpha\vsb/2}}
\tfrac\partial{\partial\lambda}\ssf
\bigr\{{\chi_{\vsf0}}(\lambda)\ssf
(\lambda^{\vsb2}\!+\msb\rho^{\ssf2}\vsf)^{1-\frac{\alpha+\sigma}2}
\lambda^{-1}\ssf|\vsf\bc(\lambda)|^{-2}\bigr\}\,\mathrm{d}\lambda
\end{equation*}
and deduce that
\begin{equation*}
\bigl|\vsf k_{\ssf t}^{\ssf\sigma,0}(0)\bigr|\lesssim t^{-\frac32}
\end{equation*}
by applying Lemma \ref{lem:vanderCorput} with \vsf$L\msb=\ssb2$\ssf.
On the other hand,
after performing $N$ integrations by parts based on \eqref{eq:IBP0},
we obtain
\begin{align*}
k_t^{\sigma,\infty}(0)&=\ssf C\msf\bigl(\tfrac i{\alpha\ssf t}\bigr)^N\!\int_{-\infty}^{\vsf+\infty}\hspace{-1mm}
e^{\hspace{.4mm}i\hspace{.4mm}t\hspace{.4mm}(\lambda^{\vsb2}+\ssf\rho^{\vsf2}\vsf)^{\vsf\alpha\vsb/2}}\\
&\ssf\times\underbrace{
\bigl\{\tfrac\partial{\partial\lambda}\ssb\circ\vsb\lambda^{-1}\ssf
(\lambda^{\vsb2}\!+\msb\rho^{\ssf2}\vsf)^{1-\frac\alpha2}\bigr\}^{\ssb N}
\bigr\{[\vsf1\!-\msb{\chi_{\vsf0}}(\lambda)\vsf]\ssf
(\lambda^{\vsb2}\!+\msb\rho^{\ssf2}\vsf)^{-\frac\sigma2}\msf
|\vsf\bc(\lambda)|^{-2}\bigr\}
}_{\text{O}(|\lambda|^{n-\sigma-\alpha N-1})}
\msf\mathrm{d}\lambda\msf,
\end{align*}
which is \ssf$\textrm{O}\bigl(t^{-N}\bigr)$ \ssf if \ssf$N\msb>\msb\tfrac{n\vsf-\vsf\sigma}\alpha$\ssf.
{As a first conclusion, we obtain
$$
\bigl|\vsf k_{\ssf t}^{\ssf\sigma}(0)\bigr|\lesssim t^{-\frac32}
$$
when \ssf$r\msb=\ssb0$ \ssf and \ssf$t\ssb\ge\msb1$\vsf.}
\smallskip

$\bullet$
\,Assume next that \vsf$0\ssb<\ssb r\ssb<\ssb\min\ssf\{1,r_0\}$\ssf.
According to Lemma \ref{lem:phase2Hn}~(iii),
the phase \eqref{eq:phaseHn}, with \ssf$R=\ssb\tfrac rt\msb\in\msb(\vsf0\ssf,\smash{\theta_0/2})$\ssf,
has two stationary points\,:
\ssf$\lambda_1\!\in\msb(\vsf0,\vsb c_{\vsf1}\lambda_{\vsf0})$\vsf,
which is comparable to~$\frac rt$,
and \ssf$\lambda_{\vsf2}\msb\in\msb(\vsf c_{\ssf2}\ssf\lambda_{\vsf0}\vsf,\ssb+\infty)$\vsf,
which is comparable to $(r\ssb/t)^{-1/(1-\alpha)}$.
Given another even bump function \ssf$\chi_2\msb\in\msb C_c^{\ssf\infty}(\R)$ \ssf such that
\begin{equation*}
\begin{cases}
\,\chi_2\ssb=\!1\text{ \msf on \,}{[\vsf-\ssf b_{\ssf3}\vsf,\ssb-\ssf b_{\vsf3}^{-1}\vsf]\cup[\ssf b_{\vsf3}^{-1},\vsb b_{\ssf3}\ssf]\ssf,}\\
\,\supp\chi_2\msb\subset\msb{[\vsf-\ssf b_{\vsf4}\vsf,\ssb-\ssf b_4^{-1}\vsf]\cup[\ssf b_4^{-1},\vsb b_{\vsf4}\ssf]\ssf,}
\end{cases}\end{equation*}
where \ssf{$1\msb<\ssb b_{\vsf3}\msb<\ssb b_{\vsf4}\msb<\ssb c_2$}\ssf,
let us split up
\begin{align*}
\int_{-\infty}^{\vsf+\infty}\hspace{-1mm}\mathrm{d}\lambda\,
&=\int_{-\infty}^{\vsf+\infty}\hspace{-1mm}\chi_0(\lambda)\;\mathrm{d}\lambda{
+\int_{\ssf|\lambda|\ssf<\ssf\lambda_2}\![\vsf1\!-\hspace{-.6mm}\chi_0(\lambda)\msb-\hspace{-.6mm}\chi_2(\lambda_{\vsf2}^{\ssb-1\ssb}\lambda)\ssf]\;\mathrm{d}\lambda}\\
&+\int_{-\infty}^{\vsf+\infty}\hspace{-1mm}\chi_2(\lambda_{\vsf2}^{\ssb-1\ssb}\lambda)\;\mathrm{d}\lambda
{+} \int_{\ssf|\lambda|\ssf>\ssf\lambda_2}\![\vsf1\!-\hspace{-.6mm}\chi_2(\lambda_{\vsf2}^{\ssb-1\ssb}\lambda)\ssf]\;\mathrm{d}\lambda
\end{align*}
in \eqref{eq:KernelSigma} and
\begin{equation}\label{eq:KernelSplittedInFour}
k_{\ssf t}^{\ssf\sigma}(r)
=k_{\ssf t}^{\ssf\sigma,0}(r)
{+k_{\ssf t}^{\ssf\sigma,1}(r)}
+k_{\ssf t}^{\ssf\sigma,2}(r)
+k_{\ssf t}^{\ssf\sigma,\infty}(r)
\end{equation}
accordingly. {
As far as the first term {in \eqref{eq:KernelSplittedInFour}} is concerned, we obtain
\begin{equation*}
\bigl|\vsf k_{\ssf t}^{\ssf\sigma,0}(r)\bigr|\lesssim t^{-\frac32}
\end{equation*}
either by using the phase \eqref{eq:phaseHn} with \vsf$R\ssb=\ssb0$ \ssf as above,
or by using \eqref{eq:HCformula}
and the phase \eqref{eq:phaseHn} with \vsf$R\ssb=\ssb\frac Ht$
as in the proof of \eqref{eq:EstimateKernelTildeSigma0Tlarge}.
We claim that the second term in \eqref{eq:KernelSplittedInFour} is \ssf$\text{O}(t^{-N})$\ssf, for every integer \ssf$N\!>\msb\frac{n-\vsf\sigma}\alpha$\ssf.
{This is achieved by substituting \eqref{eq:HCformula} in
\begin{align*}
k_{\ssf t}^{\ssf\sigma,1}(r)
=\ssf C\msb\int_{\vsf b_1\vsb\lambda_0\vsf<\ssf|\lambda|\ssf<\ssf b_3^{-1}\ssb\lambda_0}\!
&\bigl[1\!-\msb\chi_0(\lambda)\msb-\msb\chi_2(\lambda_2^{-1}\lambda)\bigr]\\
\times&\,e^{\hspace{.4mm}i\hspace{.4mm}t\msf(\lambda^2+\rho^2)^{\alpha/2}}\ssf
\phi_\lambda(r)\msf
(\lambda^{\vsb2}\!+\msb\rho^{\ssf2}\vsf)^{-\frac\sigma2}\,
|\vsf\bc(\lambda)|^{-2}\,\mathrm{d}\lambda
\end{align*}}
and by performing \ssf$N$ integrations by parts based on \eqref{eq:IBP2}
with \ssf$H\ssb=\ssb H(a_{-r}k)\msb\in\msb[\vsf-\ssf r\vsf,r\ssf]$\ssf,
after observing that the stationary points of the phase \eqref{eq:phaseHn},
with \ssf$R\vsb=\ssb\tfrac Ht$,
remain outside \ssf$\{\ssf\lambda\msb\in\ssb\R\,|\,\lambda_1\msb<\msb|\lambda|\msb<\msb\lambda_2\vsf\}$\ssf.
Let us turn to the third term in \eqref{eq:KernelSplittedInFour}, which reads}
\begin{align*}
k_{\ssf t}^{\ssf\sigma,2}(r)
&=\ssf2\,C\ssb\overbrace{
\int_{b_4^{-1}\ssb\lambda_{\vsf2}\vsf\le\vsf|\lambda|\vsf\le\ssf b_{\vsf4}\lambda_{\vsf2}}\hspace{-1mm}
\chi_2(\lambda_2^{-1}\lambda)\,e^{\hspace{.4mm}i\hspace{.4mm}t\msf\psi_{\ssb\frac rt}\msb(\lambda)}\,
b_{\vsf M}(\vsb-\vsb\lambda,r)\,(\lambda^{\vsb2}\!+\msb\rho^{\ssf2}\vsf)^{-\frac\sigma2}\,|\vsf\bc(\lambda)|^{-2}\,
\mathrm{d}\lambda}^I\\
&\ssf+C\ssb\underbrace{
\int_{b_4^{-1}\ssb\lambda_{\vsf2}\vsf\le\vsf|\lambda|\vsf\le\ssf b_{\vsf4}\lambda_{\vsf2}}\hspace{-1mm}
\chi_2(\lambda_2^{-1}\lambda)\,e^{\hspace{.4mm}i\hspace{.4mm}t{(\lambda^2+\rho^2)^{\alpha/2}}}\ssf
R_{\vsf M}(\lambda,r)\,(\lambda^{\vsb2}\!+\msb\rho^{\ssf2}\vsf)^{-\frac\sigma2}\,|\vsf\bc(\lambda)|^{-2}\,
\mathrm{d}\lambda}_{I\!I}
\end{align*}
{after substituting \eqref{eq:STI}.}
We estimate the main term
\begin{equation*}
|I|\lesssim r^{-\frac{n-1}2}\msf t^{-\frac12}\msf\lambda_2^{\ssb\frac{n+1-\alpha}2-\vsf\sigma}
\asymp\ssf t^{-\frac12\frac{2\vsf\sigma-n}{1-\vsf\alpha}}\msf
r^{\ssf\frac12\frac{2\vsf\sigma-\vsf n}{1-\vsf\alpha}-\frac n2}
\end{equation*}
by using Lemma \ref{lem:vanderCorput} with \vsf$L\msb=\ssb2$\ssf,
together with Lemma \ref{lem:phase2Hn}~(iii) and \eqref{eq:EstimateBM},
and the remainder
\begin{equation*}
|I\!I|\lesssim r^{-\frac{n-1}2-M}\msf\lambda_2^{\ssb\frac{n+1}2-\vsf\sigma-M}
\asymp\underbrace{\bigl(\tfrac{r^\alpha}t\bigr)^{\frac{2M-1}{2\vsf(1-\alpha)}}}_{\le1}
t^{-\frac12\frac{2\vsf\sigma-n}{1-\vsf\alpha}}\msf
r^{\ssf\frac12\frac{2\vsf\sigma-\vsf n}{1-\vsf\alpha}-\frac n2}
\quad{\qquad\forall\,M\!>\msb\tfrac{n+1}2\msb-\ssb\sigma}
\end{equation*}
by using \eqref{eq:EstimateRM}.
{Consider finally the last term in \eqref{eq:KernelSplittedInFour}, which reads similarly}
\begin{align*}
k_{\ssf t}^{\ssf\sigma,\infty}(r)
&=\ssf2\,C\ssb\overbrace{\int_{\ssf|\lambda|\vsf>\ssf b_{\vsf3}\vsf\lambda_2}\!
\bigl[1\!-\msb\chi_2(\lambda_2^{-1}\lambda)\bigr]\,e^{\hspace{.4mm}i\hspace{.4mm}t\msf\psi_{\frac rt}(\lambda)}\,
b_M(-\lambda\vsf,\ssb r)\msf(\lambda^{\vsb2}\!+\msb\rho^{\ssf2}\vsf)^{-\frac\sigma2}\,|\vsf\bc(\lambda)|^{-2}\,
\mathrm{d}\lambda}^{I\!I\!I}\\
&\ssf+\ssf C\ssb\underbrace{\int_{\ssf|\lambda|\vsf>\ssf b_{\vsf3}\vsf\lambda_2}\!
\bigl[1\!-\msb\chi_2(\lambda_2^{-1}\lambda)\bigr]\,e^{\hspace{.4mm}i\hspace{.4mm}t{(\lambda^2+\rho^2)^{\alpha/2}}}\ssf
R_M(\lambda\vsf,\ssb r)\msf(\lambda^{\vsb2}\!+\msb\rho^{\ssf2}\vsf)^{-\frac\sigma2}\,|\vsf\bc(\lambda)|^{-2}\,
\mathrm{d}\lambda}_{I\ssb V}\msf.
\end{align*}
We estimate
\begin{equation*}
|I\!I\!I|\lesssim\ssf r^{-\frac{n-1}2}\msf t^{-N}\ssf\lambda_2^{\frac{n+1}2-\vsf\sigma-\vsf\alpha\vsf N}
\asymp\underbrace{\bigl(\tfrac{r^{\vsf\alpha}}t\bigr)^{\frac{2N-1}{2\vsf(1-\alpha)}}}_{\le\ssf1}
t^{-\frac12\frac{2\vsf\sigma-n}{1-\vsf\alpha}}\msf
r^{\ssf\frac12\frac{2\vsf\sigma-\vsf n}{1-\vsf\alpha}-\frac n2}
\qquad\forall\,N\!>\msb\tfrac{n+1}2\msb-\ssb\sigma
\end{equation*}
by performing \ssf$N$ integrations by parts based on \eqref{eq:IBP1}
and by using Lemma \ref{lem:phase2Hn}~(iii) and (v), together with \eqref{eq:EstimateBM},
and
\begin{equation*}
|I\ssb V|\lesssim r^{-\frac{n-1}2-M}\ssf\lambda_2^{\frac{n+1}2-\vsf\sigma-M}
\asymp\underbrace{\bigl(\tfrac{r^{\vsf\alpha}}t\bigr)^{\frac{2M-1}{2\vsf(1-\alpha)}}}_{\le\ssf1}
t^{-\frac12\frac{2\vsf\sigma-n}{1-\vsf\alpha}}\msf
r^{\ssf\frac12\frac{2\vsf\sigma-\vsf n}{1-\vsf\alpha}-\frac n2}
\qquad\forall\,M\!>\msb\tfrac{n+1}2\msb-\ssb\sigma
\end{equation*}
by using \eqref{eq:EstimateRM}.
In conclusion,
\begin{equation*}
\bigl|\vsf k_{\ssf t}^{\ssf\sigma}(r)\bigr|
\lesssim\ssf t^{-\frac32}
+\ssf t^{-\frac12\frac{2\vsf\sigma-n}{1-\vsf\alpha}}\msf
r^{\ssf\frac12\frac{2\vsf\sigma-\vsf n}{1-\vsf\alpha}-\frac n2}
\end{equation*}
when \ssf$r\ssb<\ssb\min\ssf\{1,r_0\}$ and \ssf$t\ssb\ge\msb1$\vsf.

\begin{remark}
We obtain in particular
\,$|\ssf k_{\ssf t}^{\ssf\sigma}(r)\vsf|\msb\lesssim\ssb t^{-\frac32}\msb+\ssb t^{-\frac n2}$
 when \,$\sigma\ssb=\ssb(1\!-\msb\tfrac\alpha2)\ssf n$\ssf.
\end{remark}


{\textit{Subcase 2.2.2\/} (range \circled{5} in Figure~\ref{fig:RegionsAlphaSmall}).} Assume that \ssf$0\ssb\le\ssb r\msb<\msb1$ \vsf and \ssf$0\ssb<\ssb t\ssb<\msb1$ \vsf satisfy \ssf$r\ssb<\ssb t^{1/\alpha}$, hence \vsf$\frac rt\msb<\msb1$\vsf.
We argue as in Subcase {2.2.1} with a few differences.
By reducing \ssf$0\ssb\le\ssb r\msb<\ssb\min\ssf\{1,r_0\}\ssf$,
we may assume again that
\vsf$\frac rt\msb<\ssb\smash{\theta_0/2}$\vsf.
\smallskip

$\bullet$
\,When \ssf$r\ssb=\ssb0$\ssf,
we split again \ssf$k_{\ssf t}^{\ssf\sigma}(0)$ \vsf as in \eqref{KernelSplittedInTwo}.
On the one hand, we estimate trivially
\begin{equation*}
\bigl|\vsf k_{\ssf t}^{\ssf\sigma,\ssf0}(0)\bigr|\lesssim1\msf.
\end{equation*}
On the other hand, we estimate
\begin{equation*}
\bigl|\vsf k_{\ssf t}^{\ssf\sigma,\ssf\infty}(0)\bigr|\lesssim\ssf t^{-\frac{n-\sigma}\alpha}
\end{equation*}
by splitting up
\begin{equation*}
\int_{-\infty}^{+\infty}\!\bigl[1\!-\msb\chi(\lambda)\bigr]\,\mathrm{d}\lambda
=\int_{-\infty}^{+\infty}\!\bigl[1\!-\msb\chi(\lambda)\bigr]\msf
\chi\bigl(\tfrac{t^{1/\alpha}\lambda}2\bigr)\,\mathrm{d}\lambda
+\int_{-\infty}^{+\infty}\!\bigl[1\!-\msb\chi\bigl(\tfrac{t^{1/\alpha}\lambda}2\bigr)\bigr]\,\mathrm{d}\lambda\,.
\end{equation*}
More precisely, the contribution of the first integral is trivially bounded by
\begin{equation*}
\int_{\vsf1\vsf\lesssim\vsf|\lambda|\vsf\lesssim\ssf t^{\vsf-1/\alpha}}\!
|\lambda|^{\vsf n-\vsf\sigma-1}\,\mathrm{d}\lambda
\lesssim\ssf t^{\vsf-\frac{n-\sigma}\alpha}
\end{equation*}
while the contribution of the second integral is bounded,
after \ssf$N\msb>\msb\tfrac{n\vsf-\vsf\sigma}\alpha$ \ssf integrations by parts based on \eqref{eq:IBP0},
by
\begin{align*}
&\msf t^{-N}\!
\int_{\ssf|\lambda|\ssf\gtrsim\ssf t^{-1/\alpha}}\msb
\overbrace{\bigl|\bigl\{\tfrac\partial{\partial\lambda}\ssb\circ\vsb\lambda^{-1}\ssf
(\lambda^{\vsb2}\!+\msb\rho^{\ssf2}\vsf)^{1-\frac\alpha2}\bigr\}^N
\bigr\{\bigl[\vsf1\!-\msb\chi\bigl(\tfrac{t^{1/\alpha}\lambda}2\bigr)\bigr]\ssf
(\lambda^{\vsb2}\!+\msb\rho^{\ssf2}\vsf)^{-\frac\sigma2}\msf
|\vsf\bc(\lambda)|^{-2}\bigr\}\bigr|}^{\lesssim\,|\lambda|^{n-\sigma-\alpha N\vsb-1}}\mathrm{d}\lambda\\
&\lesssim\ssf t^{\vsf-\frac{n-\sigma}\alpha}\ssf.
\end{align*}
{As first conclusion, we obtain $\bigl|\vsf k_{\ssf t}^{\ssf\sigma}(0)\bigr|\lesssim\ssf t^{-\frac{n-\sigma}\alpha}$ when $r=0$ and $t \geq 1$.}

$\bullet$
\,When \ssf$0\ssb<\ssb r\ssb<\ssb\min\ssf\{1,r_0\vsf,t^{1/\alpha}\}$\vsf,
we split again \ssf$k_{\ssf t}^{\ssf\sigma}(r)$ \vsf as in \eqref{eq:KernelDecomposition2}.
This time, we estimate
\begin{equation*}
\bigl|\vsf k_{\ssf t}^{\ssf\sigma,0}(r)\bigr|\lesssim1
\end{equation*}
trivially and both \ssf$k_{\ssf t}^{\ssf\sigma,2}(r)$\vsf, $IV$ by
\begin{equation*}
t^{-\frac12\frac{2\vsf\sigma-n}{1-\vsf\alpha}}\msf
r^{\ssf\frac12\frac{2\vsf\sigma-\vsf n}{1-\vsf\alpha}-\frac n2}\ssf,
\end{equation*}
as in Sub{case} 2.2.1. Finally, we split up
\begin{equation*}
I\!I\!I=V\!I\!I+V\!I\!I\!I
\end{equation*}
according to
\begin{align*}
\int_{\ssf|\lambda|\vsf<\vsf\lambda_2}\!
\bigl[1\!-\msb\chi_0(\lambda)\msb-\msb\chi_2(\lambda_2^{-1}\lambda)\bigr] {\,\mathrm{d}\lambda}
&=\int_{\ssf|\lambda|\vsf<\vsf\lambda_2}\!
\bigl[1\!-\msb\chi_0(\lambda)\msb-\msb\chi_2(\lambda_2^{-1}\lambda)\bigr]\msf
\chi\bigl(t^{1/\alpha}\lambda\bigr)\,\mathrm{d}\lambda\\
&+\int_{\ssf|\lambda|\vsf<\vsf\lambda_2}\!
\bigl[1\!-\msb\chi_0(\lambda)\msb-\msb\chi_2(\lambda_2^{-1}\lambda)\bigr]
\bigl[1\!-\msb\chi\bigl(t^{1/\alpha}\lambda\bigr)\bigr]\,\mathrm{d}\lambda\,.
\end{align*}
On the one hand, we estimate
\begin{equation*}
|V\!I\!I|
\lesssim\int_{\vsf1\lesssim\vsf|\lambda|\vsf\lesssim\ssf t^{-1/\alpha}}\!
|\lambda|^{n-\vsf\sigma-1}\,\mathrm{d}\lambda
\lesssim\ssf t^{-\frac{n-\vsf\sigma}\alpha}\ssf,
\end{equation*}
by using \vsf$|\phi_\lambda(r)|\msb\le\msb1$\vsf, and
\begin{equation*}
|V\!I\!I{\! I}|\lesssim t^{-N}\!\int_{\ssf|\lambda|\vsf\gtrsim\ssf t^{-1/\alpha}}\!
|\lambda|^{n-\vsf\sigma-\vsf\alpha\vsf N-1}\mathrm{d}\lambda
\lesssim\ssf t^{-\frac{n-\vsf\sigma}\alpha}\ssf,
\end{equation*}
by substituting \eqref{eq:HCformula},
by performing \ssf$N\msb>\ssb\tfrac{n-\vsf\sigma}\alpha$ \vsf integrations by parts based on \eqref{eq:IBP2}
and by using \eqref{eq:Estimate2eq_Psiprime}.
\smallskip
\smallskip

In conclusion,
\begin{equation*}
\bigl|\vsf k_{\ssf t}^{\ssf\sigma}(r)\bigr|
\lesssim t^{-\frac{n-\vsf\sigma}\alpha}
+\ssf t^{-\frac12\frac{2\vsf\sigma-n}{1-\vsf\alpha}}\msf
r^{\ssf\frac12\frac{2\vsf\sigma-\vsf n}{1-\vsf\alpha}-\frac n2}
\end{equation*}
when \ssf$0\ssb\le\ssb r\ssb<\ssb\min\ssf\{1,r_0\vsf,t^{1/\alpha}\}$\vsf.

\begin{remark}
We obtain in particular
\,$|\ssf k_{\ssf t}^{\ssf\sigma}(r)\vsf|\msb
\lesssim\ssb t^{-\frac n2}$
when \,$\sigma\ssb=\ssb(1\!-\msb\tfrac\alpha2)\ssf n$\ssf.
\end{remark}


{\textit{Subcase 2.2.3} (range \circled{6} in Figure~\ref{fig:RegionsAlphaSmall}).}
Assume that \ssf$0\ssb<\ssb t\ssb\le\ssb r^{\vsf\alpha}\msb<\msb1$\vsf.
Notice that
\begin{equation*}
\bigl(\tfrac rt\bigr)^{-\frac1{1-\alpha}}\le\ssf t^{-1/\alpha}\,.
\end{equation*}
According to Lemma \ref{lem:phase2Hn}, there exists \ssf$c\ssb>\ssb0$ \ssf such that all critical of the phase \eqref{eq:phaseHn}, with \ssf$R\ssb=\msb\tfrac rt\msb\in\msb[\ssf 0\vsf,\theta_0\ssf]$\ssf, are contained in $[\ssf0\vsf,c\ssf R^{-1/(1-\alpha)}\ssf]\msb\subset\msb[\ssf0\vsf,c\ssf t^{-1/\alpha}]$\ssf.
Given a smooth even bump function \ssf$\chi$ \ssf on \ssf$\R$ \ssf such that \ssf$\chi\ssb=\msb1$ \vsf on \ssf$[-1,1\vsf]$ \ssf and \ssf$\supp\chi\msb\subset\msb[-2\vsf,2\ssf]$\ssf, let us split up
\begin{equation*}
\int_{-\infty}^{+\infty}\!\mathrm{d}\lambda
=\!\int_{-\infty}^{+\infty}\!\chi\bigl(\tfrac{t^{\vsf1/\alpha}}{2\ssf c}\lambda\bigr)\,\mathrm{d}\lambda
+\!\int_{-\infty}^{+\infty}\!\bigl[1\!-\msb\chi\bigl(\tfrac{t^{\vsf1/\alpha}}{2\ssf c}\lambda\bigr)\bigr]\,\mathrm{d}\lambda
\end{equation*}
and
\begin{equation*}
    k_{\ssf t}^{\ssf\sigma}(r)=\ssf k_{\ssf t}^{\ssf\sigma,\vsf0}(r)+\ssf k_{\ssf t}^{\ssf\sigma,\vsf\infty}(r)
\end{equation*}
accordingly. On the one hand, we estimate
\begin{equation*}
    \bigl|\vsf k_{\ssf t}^{\ssf\sigma,\vsf0}(r)\bigr|\lesssim\!\int_{\ssf|\lambda|\vsf\lesssim\ssf t^{-1/\alpha}}\!(|\lambda|\msb+\!1)^{n-\vsf\sigma-1}\,\mathrm{d}\lambda\lesssim\ssf t^{-\frac{n-\sigma}\alpha}
\end{equation*}
by using \vsf$|\phi_\lambda(r)|\msb\le\msb1$\vsf.
On the other hand,
we estimate
\begin{equation*}
    \bigl|\vsf k_{\ssf t}^{\ssf\sigma,\vsf\infty}(r)\bigr|\lesssim t^{-N}\int_{\ssf|\lambda|\vsf\gtrsim\ssf t^{-1/\alpha}}\!
    |\lambda|^{n-\vsf\sigma-\alpha\vsf N-1}\mathrm{d}\lambda
    \lesssim\ssf t^{-\frac{n-\vsf\sigma}\alpha}\ssf,
\end{equation*}
by substituting \eqref{eq:HCformula},
by performing \ssf$N\msb>\ssb\tfrac{n-\vsf\sigma}\alpha$ \vsf integrations by parts based on \eqref{eq:IBP2}
and by using \eqref{eq:Estimate2eq_Psiprime}.
\smallskip


In conclusion,
\begin{equation*}
    \bigl|\vsf k_{\ssf t}^{\ssf\sigma}(r)\bigr|
    \lesssim\ssf t^{-\frac{n-\vsf\sigma}\alpha}
\end{equation*}
when \ssf$0\ssb<\ssb t\ssb\le\ssb r^\alpha\msb<\msb1$\vsf.
\hfill$\square$

{\begin{remark}
We obtain again
\,$|\ssf k_{\ssf t}^{\ssf\sigma}(r)\vsf|\msb
\lesssim\ssb t^{-\frac n2}$
when \,$\sigma\ssb=\ssb(1\!-\msb\tfrac\alpha2)\ssf n$\ssf.
\end{remark}}

\section{Dispersive and Strichartz estimates on hyperbolic spaces} \label{sect:estimates_hyperbolic_spaces}

We now deduce from Theorem \ref{thm:KernelEstimate} dispersive and Strichartz estimates,
following the standard strategy of Ginibre and Velo,
combined with the Kunze--Stein phenomenon,
as in \cite{AnkerPierfelice2009,
IonescuStaffilani2009,
AnkerPierfeliceVallarino2011,
AnkerPierfeliceVallarino2012,
AnkerPierfelice2014,
AnkerPierfeliceVallarino2015}.
{More precisely, we will use the {following} version of the Kunze--Stein phenomenon {(see \cite[Thm.~4.2]{AnkerPierfeliceVallarino2011})}.}

{\begin{lemma}\label{lem:KunzeStein}
Let \,$2\msb\le\msb q\msb<\!\infty$\ssf.
Then there exists a positive constant \,$C$ such that,
for every \,$f\!\in\!L^{q^{\ssf\prime}}\msb(\Hn)$
and for every measurable radial function \,$k$ on \,$\Hn$,
\begin{equation*}
\|\ssf f\ssb*k\msf\|_{L^q\vphantom{\frac|0}}
\le C\,\|\ssf f\ssf\|_{L^{q^{\vsf\prime}}\vphantom{\frac|0}}\ssf
\Bigl[\msf\int_{\,0}^{+\infty}\hspace{-1mm}|\ssf k(r)\vsf|^{\frac q2}\,
\phi_0(r)\,(\sinh r)^{n-1}\,\mathrm{d}\ssf r\,\Bigr]^{\frac2q}\ssf.
\end{equation*}
\end{lemma}

\begin{remark}
Notice that
\begin{equation}\label{Weight}
\phi_0(r)\,(\sinh r)^{n-1}\asymp\ssf
\tfrac{r^{\vsf n-1}}{(1\ssf+\msf r)^{n-2}}\;e^{\vsf\frac{n-1}2\ssf r}
\qquad\forall\;r\msb\ge\ssb0\msf.
\end{equation}
\end{remark}}

The proof of the Strichartz estimates uses the standard $TT^*$ {argument, Young's inequality extended to weak type spaces (see for instance \cite[Thm.~1.4.25]{Grafakos2014}), the Christ-Kiselev Lemma \cite{ChristKiselev2001} for the non-endpoint estimates and the Bourgain or Keel-Tao trick \cite{KeelTao1998} for the endpoint estimates, as in \cite[Sect.~6]{AnkerPierfeliceVallarino2011} or \cite[Sect.~5]{AnkerPierfelice2014} for instance. Therefore we generally omit proofs of the Strichartz estimates in this paper.}


\subsection{Case \vsf$1\msb<\ssb\alpha\ssb<\ssb2$}

\begin{theorem}[{Dispersive estimates}]
\label{thm:DispersiveEstimateHn_CaseAlphaLarge}
Let \,$1\!<\msb\alpha\msb<\msb2$\ssf,
{$0\ssb\le\msb\sigma\msb\le\msb\frac n2$\vsf,}
$2\msb<\msb q\msb\le\msb\infty$ and set
\begin{equation}\label{eq:DefinitionMinimumM}
m=\max\ssf\bigl\{2\ssf\tfrac{n\vsf-\vsf\sigma}\alpha,
\tfrac{n\vsf-\vsf2\ssf\sigma}{\alpha\vsf-1}\bigr\}
=\begin{cases}
\,2\ssf\frac{n\vsf-\vsf\sigma}\alpha
&\text{if \,$\sigma\msb\ge\msb(1\msb-\msb\frac\alpha2)\ssf n$\ssf,}\\
\msf\tfrac{n\vsf-\vsf2\ssf\sigma}{\alpha\vsf-1}
&\text{if \,$\sigma\msb\le\msb(1\msb-\msb\frac\alpha2)\ssf n$\ssf.}
\end{cases}
\end{equation}
Then the following dispersive estimates hold for \,$t\ssb\in\ssb\R^*:$
\begin{equation*}
\bigl\|\ssf(-\Delta)^{-\ssf(\frac12-\frac1q)\ssf\sigma}\msf
e^{\hspace{.4mm}i\hspace{.4mm}t\ssf(-\Delta)^{\alpha/2}}\ssf
\bigr\|_{L^{q^{\vsf\prime}}\!\to\vsf L^q}
\lesssim\ssf|t|^{-\ssf(\frac12-\frac1q)\ssf m}
\end{equation*}
for \,$t$ small, say \,$0\msb<\msb|t|\msb<\msb1$\vsf, and
\begin{equation*}
\bigl\|\ssf(-\Delta)^{-\ssf\sigma/2}\msf
e^{\hspace{.4mm}i\hspace{.4mm}t\ssf(-\Delta)^{\alpha/2}}\ssf
\bigr\|_{L^{q^{\vsf\prime}}\!\to\vsf L^q}
\lesssim\ssf|t|^{-\frac32}
\end{equation*}
for \,$t$ large{, say \,$|t|\msb\ge\msb1$\vsf}.
\end{theorem}

\begin{remark}
These estimates become in particular
\begin{equation*}
\bigl\|\ssf(-\Delta)^{-\ssf\sigma/2}\msf
e^{\hspace{.4mm}i\hspace{.4mm}t\ssf(-\Delta)^{\alpha/2}}\ssf
\bigr\|_{L^{q^{\vsf\prime}}\!\to\vsf L^q}
\lesssim\ssf\begin{cases}
\,|t|^{-\vsf(\frac12-\frac1q)\ssf n}
&\text{if \,$0\msb<\msb|t|\msb<\msb1$}\\
\,|t|^{-\frac32}
&\text{if \,$|t|\msb\ge\msb1$}\\
\end{cases}
\end{equation*}
when \,$\sigma\msb=\msb(\frac12\msb-\msb\frac1q)\vsf(2\msb-\msb\alpha)\ssf n$\ssf.
\end{remark}

\begin{proof}
All estimates rely on the kernel estimates in Theorem \ref{thm:KernelEstimate}~{(i)} {extended straightforwardly to the vertical strip \ssf$0\ssb\le\ssb\Re\sigma\msb\le\msb\frac n2$ \vsf in \ssf$\C$\ssf.}
More precisely, the estimates for \ssf$t$ \ssf small are obtained by interpolation {for an analytic family of operators (see for instance \cite[Ch.~V, Thm.~4.1]{SteinWeiss1971} or \cite[Thm.~1.3.7]{Grafakos2014})} between
\begin{equation}\label{eq:ConservationL2}
\bigl\|\ssf(-\Delta)^{-\ssf\sigma/2}\msf
e^{\hspace{.4mm}i\hspace{.4mm}t\ssf(-\Delta)^{\alpha/2}}\ssf
\bigr\|_{L^2\to\vsf L^2}
=1
\qquad\forall\;\sigma\msb\in\ssb i\ssf\R
\end{equation}
and

\centerline{$
\bigl\|\ssf(-\Delta)^{-\ssf\sigma/2}\msf
e^{\hspace{.4mm}i\hspace{.4mm}t\ssf(-\Delta)^{\alpha/2}}\ssf
\bigr\|_{L^1\to\vsf L^\infty}
\lesssim\ssf{\begin{cases}
\msf|t|^{-\frac{n\vsf-\Re\sigma}\alpha}
&\text{if \ssf$\Re\sigma\msb\ge\msb(1\msb-\msb\frac\alpha2)\ssf n$\ssf,}\\
\msf|t|^{-\frac12\frac{n\vsf-\vsf2\Re\sigma}{\alpha\vsf-1}}
&\text{if \ssf$\Re\sigma\msb\le\msb(1\msb-\msb\frac\alpha2)\ssf n$\ssf.}\\
\end{cases}}$}
\vspace{1mm}
While the estimate for \ssf$t$ \ssf large follows from
\vspace{1mm}

\centerline{$\displaystyle
\Bigl[\msf\int_{\,0}^{+\infty}\hspace{-1mm}
|\ssf k_{\ssf t}^{\ssf\sigma}(r)\vsf|^{\frac q2}\,
\phi_0(r)\,(\sinh r)^{n-1}\,
\mathrm{d}\ssf r\msf\Bigr]^{\frac2q}
\lesssim\msf|t|^{-\frac32}
$}\vspace{1mm}

and from Lem{ma}~\ref{lem:KunzeStein}.
\end{proof}

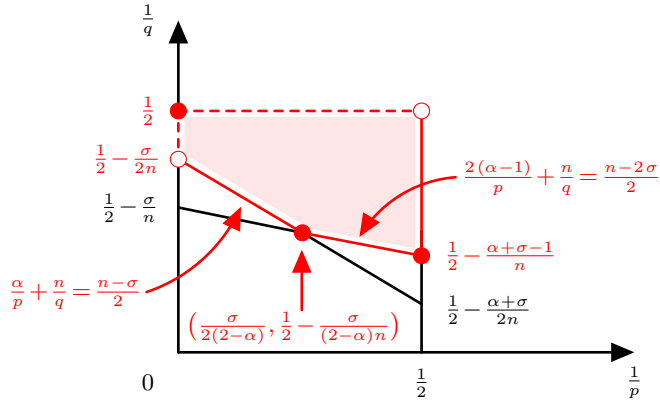
\begin{figure}[h!]
\centering
\begin{tikzpicture}[line cap=round,line join=round,>=triangle 45,x=8mm,y=8mm]
\clip(-3.,-1.) rectangle (8.,6.);
\draw [->,line width=1.pt,color=black] (0.,0.) -- (7.5,0.);
\draw [line width=1.pt,color=black] (0.,0.) -- (0.,3.2);
\draw [->,line width=1.pt,color=black] (0.,4.1) -- (0.,5.5);
\draw [dashed,line width=1.pt,color=red] (0.,4.) -- (4.,4.);
\draw [line width=1.pt,color=black] (4.,0.) -- (4.,2.);
\draw [line width=1.pt,color=red] (4.,1.6) -- (4.,4.);
\draw [line width=1.pt,color=red] (0.,3.2) -- (2.04,1.98);
\draw [line width=1.pt,color=black] (2.04,1.98) -- (4.,0.8);
\draw [line width=1.pt,color=black] (0.,2.4) -- (2.14,1.96);
\draw [line width=1.pt,color=red] (2.14,1.96) -- (4.,1.6);
\draw [dashed,line width=1.pt,color=red] (0.,3.4) -- (0.,3.8);
\fill [line width=0pt,color=red,fill=red,fill opacity=0.1] (.1,3.9) -- (3.9,3.9) -- (3.9,1.72) -- (2.16,2.05) -- (0.1,3.28) -- cycle;
\begin{scriptsize}
\draw [color=red,fill=red] (0.,4.) circle (3.0pt);
\draw [color=red,fill=white] (0.,3.2) circle (3.0pt);
\draw[color=red,fill=red] (4.,1.6) circle (3.0pt);
\draw[color=red,fill=white] (4.,4.) circle (3.0pt);
\draw[color=red,fill=red] (2.04,1.98) circle (3.0pt);
\draw[color=black](-.5,-.5) node {$0$};
\draw[color=black](7.5,-.5) node {$\frac1p$};
\draw[color=black](4,-.5) node {$\frac12$};
\draw[color=black](-.5,5.5) node {$\frac1q$};
\draw[color=red](-.5,4) node {$\frac12$};
\draw[color=red](-0.85,3.2) node {$\frac12\msb-\msb\frac\sigma{2\vsf n}$};
\draw[color=black](-0.8,2.4) node {$\frac12\msb-\msb\frac\sigma n$};
\draw[color=red](5.3,1.6) node {$\frac12\msb-\msb\frac{\alpha\vsf+\vsf\sigma-1}n$};
\draw[color=red](1.9,0.4) node {$\bigl(\frac\sigma{2\vsf(2-\alpha)},\vsb\frac12\msb-\msb\frac\sigma{(2-\alpha)\vsf n}\bigr)$};
\draw[color=black](5.1,0.7) node {$\frac12\msb-\msb\frac{\alpha\vsf+\vsf\sigma}{2\vsf n}$};
\draw[color=red](-1.7,1.) node {$\frac{\alpha\vphantom{|}}p\msb+\msb\frac{n\vphantom{|}}q\msb=\msb\frac{n\vsf-\vsf\sigma\vphantom{|}}2$};
\draw [<-, line width=1pt, color=red] (2.03,1.7)-- (2.03,0.7);
\draw [->, shift={(4.713476207838939,0.8982610910339334)},line width=1pt, color=red]  
plot[domain=1.6477573690202114:2.613235003003553,variable=\t]({1*2.009339437222573*cos(\t r)+0*2.009339437222573*sin(\t r)},{0*2.009339437222573*cos(\t r)+1*2.009339437222573*sin(\t r)});
\draw[color=red] (6.3,2.9) node {$\frac{2\ssf(\alpha\vsf-1)\vphantom{|}}p\msb+\msb\frac{n\vphantom{|}}q\msb=\msb\frac{n\vsf-\vsf2\ssf\sigma\vphantom{|}}2$};
\draw [->,shift={(-1.3607121714240082,3.479473881613361)},line width=1pt, color=red]  plot[domain=5.037886224969313:5.863639798909469,variable=\t]({1*2.61898177088956*cos(\t r)+0*2.61898177088956*sin(\t r)},{0*2.61898177088956*cos(\t r)+1*2.61898177088956*sin(\t r)});
\end{scriptsize}
\end{tikzpicture}
\caption{Admissible region \ssf$\mathcal{R}_{\vsf\alpha}$ \vsf for fixed \ssf$1\!<\msb\alpha\msb<\msb2$ \ssf and \ssf$0\msb\le\ssb\sigma\msb\le\ssb \frac n2$}
\label{fig:AdmissibleRegionAlphaLarge}
\end{figure}

\begin{theorem}[Strichartz inequalities]
\label{thm:StrichartzEstimateHnAlphaLarge}
Assume that \,$1\msb<\ssb\alpha\ssb<\ssb2$ and let
\,$I\ssb=\ssb(-T,\ssb+\vsf T\ssf)$ be an open interval with \,$T\msb>\ssb0$\ssf.
Then the following Strichartz inequalities hold for so\-lu\-tions of \eqref{eq:NLSgeneral} on \,$I\msb\times\hspace{-.4mm}\Hn:$
\begin{align*}
\|(-\Delta_{\ssf x})^{-\ssf\sigma/4}\msf u(t,x)
\|_{L^{p}(\vsb I \ssf;\ssf L^q(\mathbb{H}^n\ssb)\vsb)}
\hspace{-.9mm}\le C\,\bigl\{\hspace{.4mm}\|f\|_{L^2(\mathbb{H}^n\ssb)}
+\ssf\|(-\Delta_{\ssf x})^{\vsf\tilde{\sigma}/4}\ssf F(t,x)\|_{
L^{\vsf\tilde{p}^{\ssf\prime}\hspace{-.6mm}}(\vsb I\ssf;
\ssf L^{\tilde{q}^{\ssf\prime}\hspace{-.6mm}}(\mathbb{H}^n\ssb)\vsb)}
\ssf\bigr\}\,.
\end{align*}
Here \msf$\bigl(\frac{1\vphantom{|}}p,\ssb\frac{1\vphantom{|}}q,\sigma\bigr)$ \,and \msf$\bigl(\frac{1\vphantom{|}}{\tilde{p}},\ssb\frac{1\vphantom{|}}{\tilde{q}},\tilde{\sigma}\bigr)$
belong to the {admissible} region $(${see Figure~\ref{fig:AdmissibleRegionAlphaLarge}}\msf$)$
\begin{align*}
\mathcal{R}_{\vsf\alpha}
&=\text{\small{$\bigl\{\ssf\bigl(\tfrac{1\vphantom{|}}p,\ssb\tfrac{1\vphantom{|}}q,\sigma\bigr)\msb\in\msb\bigl(\vsf0\ssf,\ssb\tfrac12\ssf\bigr]\!\times\!\bigl[\ssf0\ssf,\ssb\tfrac12\vsf\bigr)\!\times\!\bigl[\ssf0\ssf,\ssb\tfrac n2\vsf\bigr]\msf\big|\,\tfrac{\alpha\vphantom{|}}p\msb+\msb\tfrac{n\vphantom{|}}q\msb\ge\msb\tfrac{n\vsf-\vsf\sigma\vphantom{|}}2\text{ and \,}\tfrac{2\ssf(\alpha\vsf-1)\vphantom{|}}p\msb+\msb\tfrac{n\vphantom{|}}q\msb\ge\msb\tfrac{n\vsf-\vsf2\ssf\sigma\vphantom{|}}2\ssf\bigr\}\,\cup\,\bigl\{\bigl(\vsf0\ssf,\ssb\tfrac12\vsb,\ssb0\vsf\bigr)\bigr\}\,,$}}
\end{align*}
where \,$m$ is defined in \eqref{eq:DefinitionMinimumM}.
Moreover \,$C\!\ge\msb0$ depends on \,$\alpha$\vsf,
$(\vsf p,\vsb q,\vsb\sigma)$ \,and \ssf$(\vsf\tilde{p},\vsb\tilde{q},\vsf\vsb\tilde{\sigma})$ but not on \,$T\msb>\msb0$ and \,$u$\ssf.
\end{theorem}

\begin{remark}
As already observed for the Schr\"odinger equation $(\alpha\ssb=\ssb2)$ and for the wave equation $(\alpha\ssb=\msb1)$, the admissible region is much larger on hyperbolic spaces than on Euclidean spaces.
\end{remark}

\subsection{Case \ssf$0\ssb<\ssb\alpha\ssb<\msb1$}
\label{sect:case_alpha_small}
From the kernel estimates in Theorem \ref{thm:KernelEstimate} (ii), we deduce similarly the following inequalities.
\medskip

\begin{theorem}[Dispersive estimates]
\label{thm:DispersiveEstimateHn_CaseAlphaSmall}
\textnormal{(i)}
Assume that \,$n\ssb\ge\ssb3$\ssf,
\ssf$0\!<\msb\alpha\msb<\msb1$\vsf,
$(1\!-\msb\frac\alpha2)\ssf n\ssb\le\msb\sigma\msb\le\msb n$
and \,$2\ssb<\ssb q\ssb\le\ssb\infty$\ssf.
Then the following dispersive estimates hold\,$:$
\begin{equation}\label{eq:DispersiveAlphaSmallTSmall}
\bigl\|\ssf(-\Delta)^{-\ssf(\frac12-\frac1q)\ssf\sigma}\msf e^{\hspace{.4mm}i\hspace{.4mm}t\ssf(-\Delta)^{\alpha/2}}\ssf\bigr\|_{L^{q^{\vsf\prime}}\!\to\vsf L^q}\ssb\lesssim|t|^{-\vsf(\frac12-\frac1q)\ssf2\vsf\frac{n-\vsf\sigma}\alpha}
\end{equation}
for \,$t$ small, say \,$0\ssb<\ssb|t|\ssb<\msb1$\vsf, and
\begin{equation}\label{eq:DispersiveAlphaSmallTLarge}
\bigl\|\ssf(-\Delta)^{-\ssf\sigma/2}\msf
e^{\hspace{.4mm}i\hspace{.4mm}t\ssf(-\Delta)^{\alpha/2}}
\ssf\bigr\|_{L^{q^{\vsf\prime}}\!\to\vsf L^q}\ssb
\lesssim|t|^{-\frac32}  
\end{equation}
for \,$t$ large, say \,$|t|\ssb\ge\msb1$\vsf.

\textnormal{(ii)}
In dimension \,$n\ssb=\ssb2$\ssf,
the small time estimate is the same,
while the large time estimate reads
\begin{align*}
\bigl\|\ssf(-\Delta)^{-\ssf\sigma/2}\msf
e^{\hspace{.4mm}i\hspace{.4mm}t\ssf(-\Delta)^{\alpha/2}}
\ssf\bigr\|_{L^{q^{\vsf\prime}}\!\to\vsf L^q}\ssb
&\lesssim|t|^{-\min\ssf\{\frac32,\frac{\sigma-1}{1-\vsf\alpha}\}}\\
&{=\begin{cases}
\msf|t|^{-\frac{\sigma-1}{1-\vsf\alpha}}
&\text{if \,$0\ssb<\ssb\alpha\ssb\le\msb\frac13$ and \,$2\ssb-\ssb\alpha\ssb\le\ssb\sigma\msb\le\ssb2$\ssf,}\\
\msf|t|^{-\frac{\sigma-1}{1-\vsf\alpha}}
&\text{if \,$\frac13\msb\le\ssb\alpha\ssb<\msb1$ and \,$2\ssb-\ssb\alpha\ssb\le\ssb\sigma\msb\le\msb\frac{5\ssf-\ssf3\ssf\alpha}2$\vsf,}\\
\msf|t|^{-\frac32}
&\text{if \,$\frac13\msb\le\ssb\alpha\ssb<\msb1$ and
\,$\frac{5\ssf-\ssf3\ssf\alpha}2\msb\le\ssb\sigma\msb\le\ssb2$\ssf.}
\end{cases}}
\end{align*}
\end{theorem}

{The estimate \eqref{eq:DispersiveAlphaSmallTLarge} is fine for \ssf$q$ \ssf large but, as \ssf$q$ \ssf tends to \ssf$2$\ssf, the smoothness factor \vsf$(-\Delta)^{-\ssf\sigma/2}$ \vsf doesn't vanish, as might be expected. Let us therefore refine \eqref{eq:DispersiveAlphaSmallTLarge} as follows.}

\begin{corollary}
Assume that \,$n\ssb\ge\ssb3$\ssf,
\ssf$0\!<\msb\alpha\msb<\msb1$\vsf,
$(1\!-\msb\frac\alpha2)\ssf n\ssb\le\msb\sigma\msb\le\msb n$
and let \,$2\ssb<\ssb q\ssb\le\ssb Q\ssb\le\ssb\infty\ssf$.
Then the following dispersive estimate holds
for \,$t$ large, say \,$|t|\ssb\ge\msb1:$
\begin{equation}\label{eq:ModifiedDispersiveAlphaSmallTLarge}
\bigl\|\ssf
(-\Delta)^{-\frac{1/2\vsf-1/q}{1/2\vsf-1/Q}\frac\sigma2}\msf
e^{\hspace{.4mm}i\hspace{.4mm}t\ssf(-\Delta)^{\alpha/2}}
\ssf\bigr\|_{L^{q^{\vsf\prime}}\!\to\vsf L^q}\ssb
\lesssim|t|^{-\frac{1/2\vsf-1/q}{1/2\vsf-1/Q}\frac32}\,.
\end{equation}
\end{corollary}

\begin{proof}
This result is obtained by interpolation between the estimate \eqref{eq:ConservationL2}
for \ssf$q\ssb=\ssb2$\ssf 
 and \ssb$\sigma\msb\in\ssb i\msf\R$\ssf,
and the estimate \eqref{eq:DispersiveAlphaSmallTLarge} for \ssf$q\ssb=\ssb Q$ \ssf
and \ssf$\sigma\msb\in\ssb\C$ \ssf with
$(1\!-\msb\frac\alpha2)\ssf n\ssb\le\ssb\Re\sigma\msb\le\ssb n$\ssf.
\end{proof}

\begin{theorem}[Strichartz inequalities]\label{thm:StrichartzEstimateHnAlphaSmall}
Assume that \,$n\ssb\ge\ssb3$\ssf, $0\msb<\msb\alpha\msb<\msb1$ and let
\,$I\msb=\ssb(-T,\ssb+\vsf T\ssf)$ be an open interval with \,$T\msb>\ssb0$\ssf.
Then the following Strichartz inequalities hold for so\-lu\-tions of \eqref{eq:NLSgeneral} on \,$I\msb\times\hspace{-.4mm}\Hn:$
\begin{align*}
\|(-\Delta_{\ssf x})^{-\frac\sigma2}\ssf u(t,x)
\|_{L^{p}(\vsb I \ssf;\ssf L^q(\mathbb{H}^n\ssb)\vsb)}
\hspace{-.9mm}\le C\,\bigl\{\hspace{.4mm}\|f\|_{L^2(\mathbb{H}^n\ssb)}
+\ssf\|(-\Delta_{\ssf x})^{\frac{\tilde{\sigma}}2}\vsf F(t,x)\|_{
L^{\vsf\tilde{p}^{\ssf\prime}\hspace{-.6mm}}(\vsb I\ssf;
\ssf L^{\tilde{q}^{\ssf\prime}\hspace{-.6mm}}(\mathbb{H}^n\ssb)\vsb)}
\ssf\bigr\}\,.
\end{align*}
Here
\begin{equation*}\begin{cases}
\,\sigma\ssb=\ssb\sigma(\beta,q)\ssb=\msb\bigl(\frac12\msb-\msb\frac1q\bigr)(1\!-\msb\frac\beta2)\ssf n\\
\,\tilde{\sigma}\ssb=\ssb\sigma(\tilde{\beta},\tilde{q})\ssb=\msb\bigl(\frac12\msb-\msb\frac1{\tilde{q}}\bigr)(1\!-\msb\frac{\tilde{\beta}}2)\ssf n\\
\end{cases}
\quad\text{with}\quad\beta,\tilde{\beta}\msb\in\msb[\ssf0\vsf,\vsb\alpha\ssf]
\end{equation*}
and \,\,$\bigl(\frac{1\vphantom{|}}p,\ssb\frac{1\vphantom{|}}q,\beta\vsf\bigr)$,
$\bigl(\frac{1\vphantom{|}}{\tilde{p}},\ssb\frac{1\vphantom{|}}{\tilde{q}},\tilde{\beta}\vsf\bigr)$
belong to the following \textnormal{admissible} region $($see Figures \ref{fig:AdmissibleRegionAlphaSmallBetaSmall} and \ref{fig:AdmissibleRegionAlphaSmallBetaLarge}$\msf)\ssb:$
\begin{equation}\label{eq:AdmissibilityAlphaSmall}\begin{aligned}
\mathcal{R}_{\vsf\alpha}
&=\text{\small{$\bigl\{\bigl(\tfrac1p,\ssb\tfrac1q,\beta\vsf\bigr)\msb\in\msb\bigl[\ssf0\ssf,\ssb\tfrac12\ssf\bigr]\!\times\!\bigl[\ssf0\ssf,\ssb\tfrac12\vsf\bigr]\!\times\!\bigl[\ssf0\ssf,\ssb\alpha\ssf\bigr]\msf\big|\msf\bigl(\tfrac12\msb-\msb\tfrac1q\bigr)\ssf\tfrac\beta\alpha\ssf\tfrac n2\msb\le\msb\tfrac1p\msb\le\msb\bigl(\tfrac12\msb-\msb\tfrac1q\bigr)\ssf\tfrac{2\vsf-\vsf\beta}{2\vsf-\vsf\alpha}\ssf\tfrac32\msf\bigr\}\,\smallsetminus\msf
\bigl\{\vsb\bigl(\tfrac12,\vsb0\vsf,\ssb\tfrac2n\vsf\alpha\bigr)\vsb\bigr\}\msf.$}}
\end{aligned}\end{equation}
Moreover \,$C\!\ge\msb0$ depends on \,$n$\ssf, $\alpha$\vsf,
$(\vsf p,\vsb q,\vsb\beta\vsf)$ \,and \ssf$(\vsf\tilde{p},\vsb\tilde{q},\vsf\vsb\tilde{\beta}\vsf)$ but not on \,$T\msb>\msb0$ and \,$u$\ssf.
\end{theorem}

\begin{remark}
In this statement, the interval \,$[\ssf0\vsf,\ssb\alpha\ssf]$ must be actually reduced to the smaller interval \,$[\ssf0\vsf,\ssb\widehat{\alpha}\ssf]$\ssf, where
\vspace{1mm}

\centerline{$
\widehat{\alpha}\ssf
=\tfrac{6\ssf\alpha}{(2\ssf-\ssf\alpha)\ssf n\ssf+\ssf3\ssf\alpha}
\,\begin{cases}
=\alpha
&\text{if \,}n\ssb=\ssb3\ssf,\\
\ssf\in(0\vsf,\vsb\alpha)
&\text{if \,}n\ssb>\ssb3\ssf.\\
\end{cases}$}\vspace{1mm}
When \,$n\ssb>\ssb3$ and \,$\widehat{\alpha}\ssb<\ssb\beta\ssb\le\ssb\alpha$\ssf, we have indeed\,$\tfrac\beta\alpha\ssf\tfrac n2\msb>\msb\tfrac{2\vsf-\vsf\beta}{2\vsf-\vsf\alpha}\ssf\tfrac32$ and the admissibility condition boils down to the trivial endpoint \,$p\ssb=\ssb\infty$\ssf, $q\ssb=\ssb2$\ssf.
\end{remark}

\begin{figure}[h!]
\centering
\begin{tikzpicture}[line cap=round,line join=round,>=triangle 45,x=8mm,y=8mm]
\clip(-3.,-1.) rectangle (8.,6.);
\draw [->,line width=1.pt,color=black] (4.1,0.) -- (7.,0.);
\draw [->,line width=1.pt,color=black] (0.,0.) -- (0.,5.5);
\draw [line width=1.pt,color=black] (0.,4.) -- (4.,4.);
\draw [line width=1.pt,color=red] (4.,0.) -- (4.,1.);
\draw [line width=1.pt,color=red] (0.,4.) -- (2.5,0.);
\draw [line width=1.pt,color=black] (0.,0.) -- (2.4,0.);
\draw [line width=1.pt,color=black] (4.,1.) -- (4.,4.);
\draw [line width=1.pt,color=red] (0.,4.) -- (4.,1.);
\draw [line width=1.pt,color=red] (2.5,0.) -- (4.,0.);
\fill [line width=0pt,color=red,fill=red,fill opacity=0.1] (.33,3.65) -- (3.9,0.95) -- (3.9,0.1) -- (2.55,0.1) -- cycle;
\begin{scriptsize}
\draw [color=red,fill=red] (0.,4.) circle (3.0pt);
\draw [color=red,fill=red] (4.,0.) circle (3.0pt);
\draw[color=red,fill=red] (4,1.) circle (3.0pt);
\draw[color=red,fill=red] (2.5,0.) circle (3.0pt);
\draw[color=black](-.5,-.65) node {$0$};
\draw[color=black](7.,-.65) node {$\frac{1\vphantom{|}}{p\vphantom{|}}$};
\draw[color=red](4,-.65) node {$\frac{1\vphantom{|}}{2\vphantom{|}}$};
\draw[color=red](2.5,-.65) node {$\frac{\beta\vphantom{|}}\alpha\vsf\frac{n\vphantom{|}}{4\vphantom{|}}$};
\draw[color=black](-.5,5.5) node {$\frac{1\vphantom{|}}{q\vphantom{|}}$};
\draw[color=red](-.5,4) node {$\frac{1\vphantom{|}}{2\vphantom{|}}$};
\draw[color=red](5.6,1.) node {$\frac{1\vphantom{|}}{Q_1}\msb=\ssb\frac{1\vphantom{|}}{2\vphantom{|}}\msb-\msb\frac{1\vphantom{|}}{3\vphantom{|}}\vsf\frac{2-\alpha}{2-\beta}$};
\draw[color=red](-1.7,1.) node {$\frac{1\vphantom{|}}{p\vphantom{|}}\ssb=\ssb(\frac{1\vphantom{|}}{2\vphantom{|}}\msb-\msb\frac{1\vphantom{|}}{q\vphantom{|}}\bigr)\vsb\frac{\beta\vphantom{|}}{\alpha\vphantom{|}}\vsf\frac{n\vphantom{|}}{2\vphantom{|}}$};
\draw [->,shift={(-.5,4.)},line width=1pt, color=red]  plot[domain=4.83:5.33,variable=\t]({1*3*cos(\t r)+0*3*sin(\t r)},{0*3*cos(\t r)+1*3*sin(\t r)});
\draw[color=red](3.,4.8) node {$\frac{1\vphantom{|}}{p\vphantom{|}}\ssb=\ssb(\frac{1\vphantom{|}}{2\vphantom{|}}\msb-\msb\frac{1\vphantom{|}}{q\vphantom{|}}\bigr)\vsb\frac{2-\beta\vphantom{|}}{2-\alpha\vphantom{|}}\vsf\frac{3\vphantom{|}}{2\vphantom{|}}$};
\draw [->,shift={(-1.,5.)},line width=1pt, color=red]  plot[domain=6.13:5.65,variable=\t]({1*4*cos(\t r)+0*4*sin(\t r)},{0*4*cos(\t r)+1*4*sin(\t r)});
\end{scriptsize}
\end{tikzpicture}
\caption{Admissible region \ssf$\mathcal{R}_{\vsf\alpha}$ \vsf for \ssf$n\ssb\ge\ssb3$ \ssf and \ssf$0\ssb<\ssb\alpha\ssb<\msb1$\vsf, $0\ssb<\ssb\beta\ssb<\ssb\frac2n\ssf\alpha$ \ssf fixed.}
\label{fig:AdmissibleRegionAlphaSmallBetaSmall}
\end{figure}
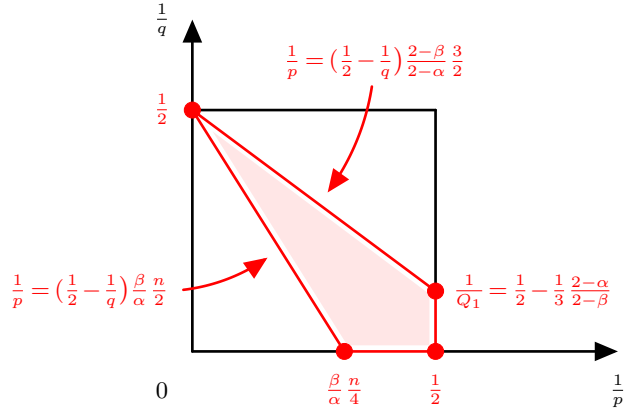

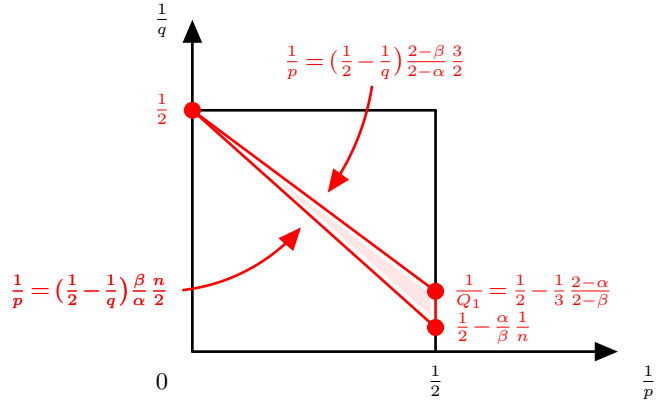
\begin{figure}[h!]
\centering
\begin{tikzpicture}[line cap=round,line join=round,>=triangle 45,x=8mm,y=8mm]
\clip(-3.,-1.) rectangle (8.,6.);
\draw [->,line width=1.pt,color=black] (0.,0.) -- (7.,0.);
\draw [->,line width=1.pt,color=black] (0.,0.) -- (0.,5.5);
\draw [line width=1.pt,color=black] (0.,4.) -- (4.,4.);
\draw [line width=1.pt,color=red] (0.,4.) -- (4.,0.4);
\draw [line width=1.pt,color=red] (4.,1.) -- (4.,0.4);
\draw [line width=1.pt,color=black] (4.,1.) -- (4.,4.);
\draw [line width=1.pt,color=black] (4.,0.) -- (4.,0.3);
\draw [line width=1.pt,color=red] (0.,4.) -- (4.,1.);
\fill [line width=0pt,color=red,fill=red,fill opacity=0.1] (1.3,2.93) -- (3.92,0.95) -- (3.92,0.6) -- cycle;
\begin{scriptsize}
\draw [color=red,fill=red] (0.,4.) circle (3.0pt);
\draw[color=red,fill=red] (4,1.) circle (3.0pt);
\draw[color=red,fill=red] (4.,0.4) circle (3.0pt);
\draw[color=black](-.5,-.5) node {$0$};
\draw[color=black](7.5,-.5) node {$\frac{1\vphantom{|}}{p\vphantom{|}}$};
\draw[color=black](4,-.5) node {$\frac{1\vphantom{|}}{2\vphantom{|}}$};
\draw[color=black](-.5,5.5) node {$\frac{1\vphantom{|}}{q\vphantom{|}}$};
\draw[color=red](-.5,4) node {$\frac{1\vphantom{|}}{2\vphantom{|}}$};
\draw[color=red](5.6,1.) node {$\frac{1\vphantom{|}}{Q_1}\msb=\ssb\frac{1\vphantom{|}}{2\vphantom{|}}\msb-\msb\frac{1\vphantom{|}}{3\vphantom{|}}\vsf\frac{2-\alpha}{2-\beta}$};
\draw[color=red](-1.7,1.) node {$\frac{1\vphantom{|}}{p\vphantom{|}}\ssb=\ssb(\frac{1\vphantom{|}}{2\vphantom{|}}\msb-\msb\frac{1\vphantom{|}}{q\vphantom{|}}\bigr)\vsb\frac{\beta\vphantom{|}}{\alpha\vphantom{|}}\vsf\frac{n\vphantom{|}}{2\vphantom{|}}$};
\draw[color=red](4.95,0.4) node {$\frac12\msb-\msb\frac{\alpha\vphantom{|}}{\beta\vphantom{|}}\vsf\frac{1\vphantom{|}}{n\vphantom{|}}$};
\draw[color=red](-1.7,1.) node {$\frac{1\vphantom{|}}{p\vphantom{|}}\ssb=\ssb(\frac{1\vphantom{|}}{2\vphantom{|}}\msb-\msb\frac{1\vphantom{|}}{q\vphantom{|}}\bigr)\vsb\frac{\beta\vphantom{|}}{\alpha\vphantom{|}}\vsf\frac{n\vphantom{|}}{2\vphantom{|}}$};
\draw [->,shift={(-.5,4.)},line width=1pt, color=red]  plot[domain=4.83:5.58,variable=\t]({1*3*cos(\t r)+0*3*sin(\t r)},{0*3*cos(\t r)+1*3*sin(\t r)});
\draw[color=red](3.,4.8) node {$\frac{1\vphantom{|}}{p\vphantom{|}}\ssb=\ssb(\frac{1\vphantom{|}}{2\vphantom{|}}\msb-\msb\frac{1\vphantom{|}}{q\vphantom{|}}\bigr)\vsb\frac{2-\beta\vphantom{|}}{2-\alpha\vphantom{|}}\vsf\frac{3\vphantom{|}}{2\vphantom{|}}$};
\draw [->,shift={(-1.,5.)},line width=1pt, color=red]  plot[domain=6.13:5.65,variable=\t]({1*4*cos(\t r)+0*4*sin(\t r)},{0*4*cos(\t r)+1*4*sin(\t r)});
\end{scriptsize}
\end{tikzpicture}
\caption{Admissible region \ssf$\mathcal{R}_{\vsf\alpha}$ \vsf for \ssf$n\ssb\ge\ssb3$ \ssf and \ssf$0\ssb<\ssb\alpha\ssb<\msb1$\vsf, $\frac2n\ssf\alpha\ssb<\ssb\beta\ssb<\ssb\widehat{\alpha}$ \ssf fixed.}
\label{fig:AdmissibleRegionAlphaSmallBetaLarge}
\end{figure}

\begin{proof}
Referring to the proofs of \cite[Thm.~6.3]{AnkerPierfeliceVallarino2011}
and \cite[Thm.~5.2]{AnkerPierfelice2014},
we shall be content to explain and comment the admissibility conditions
\begin{equation}\label{eq:TwoConditionsAdmissibility}
\begin{cases}
\msf\frac1p\ssb\ge\ssb\bigl(\tfrac12\msb-\msb\tfrac1q\bigr)\ssf\tfrac\beta\alpha\ssf\tfrac n2\\
\msf\frac1p\ssb\le\ssb\bigl(\frac12\msb-\msb\frac1q\bigr)\ssf\frac{2\vsf-\vsf\beta}{2\vsf-\vsf\alpha}\ssf\frac32\\
\ssf\bigl(\frac1p,\ssb\frac1q,\vsb\beta\bigr)\ssb\ne\ssb\bigl(\tfrac12,\vsb0\vsf,\ssb\tfrac2n\vsf\alpha\bigr)
\end{cases}
\end{equation}
occurring in \eqref{eq:AdmissibilityAlphaSmall}.
Recall that the above-mentioned proofs consist mainly in estimating
\begin{equation}\label{eq:StrichartzMainEstimate1}\begin{gathered}
\Bigl\|\,\int_{\ssf0\vsf<\vsf|t-s|\vsf<1}\vsf
\bigl\|\ssf(-\ssf\Delta_{\ssf x})^{\vsb-\vsf\sigma(\beta,q)}\msf
e^{\ssf i\ssf(t-s)\vsf(-\vsf\Delta_{\vsf x})^{\alpha/2}}
F(s,x)\ssf\bigr\|_{L_x^q\vphantom{\big|}}\msf ds\,\Bigr\|_{L_t^p}\\
\lesssim\msf\Bigl\|\,\int_{-\infty}^{\ssf+\infty}\!
\|F(s,x)\|_{L_x^{q^{\vsf\prime}}}\msf ds\,\Bigr\|_{L_s^{p^{\ssf\prime}}}
\end{gathered}\end{equation}
and
\begin{equation}\label{eq:StrichartzMainEstimate2}\begin{gathered}
\Bigl\|\,\int_{\ssf|t-s|\vsf\ge1}\vsf
\bigl\|\ssf(-\ssf\Delta_{\ssf x})^{\vsb-\vsf\sigma(\beta,q)}\msf
e^{\ssf i\ssf(t-s)\vsf(-\vsf\Delta_{\vsf x})^{\alpha/2}}
F(s,x)\ssf\bigr\|_{L_x^q\vphantom{\big|}}\msf ds\,\Bigr\|_{L_t^p}\\
\lesssim\msf\Bigl\|\,\int_{-\infty}^{\ssf+\infty}\!
\|F(s,x)\|_{L_x^{q^{\vsf\prime}}}\msf ds\,\Bigr\|_{L_s^{p^{\ssf\prime}}}\,.
\end{gathered}\end{equation}
On the one hand, we deduce \eqref{eq:StrichartzMainEstimate1} from the dispersive estimate \eqref{eq:DispersiveAlphaSmallTSmall}, which yields
\begin{equation*}
\bigl\|\ssf(-\Delta)^{-\vsf\sigma(\beta,q)}\msf e^{\ssf i\ssf(t-s)\vsf(-\Delta)^{\alpha/2}}\ssf\bigr\|_{L^{q^{\vsf\prime}}\!\to\vsf L^q}\ssb
\lesssim|\ssf t\ssb-\ssb s\ssf|^{\vsf-\ssf(\frac12-\frac1q)\ssf\frac\beta\alpha\ssf n}\,,
\end{equation*}
and from Young's inequality (see for instance \cite[Thm.~1.4.25]{Grafakos2014})
provided that
\linebreak
\ssf$(\frac12\msb-\msb\frac1q)\ssf\frac\beta\alpha\ssf n$ \ssf is \ssf$<\msb1$ \vsf and \ssf$\le\msb\frac2p$\ssf.
This way we obtain \eqref{eq:StrichartzMainEstimate1} under the assumptions
$$
0\ssb\le\ssb\beta\ssb\le\ssb\alpha\ssf,\quad
2\ssb\le\ssb p\ssb\le\ssb\infty\ssf,\quad
2\ssb<\ssb q\ssb\le\ssb\infty\ssf,\quad
\tfrac1p\ssb\ge\ssb(\tfrac12\ssb-\ssb\tfrac1q)\ssf\tfrac\beta\alpha\ssf\tfrac n2
$$
and except for the endpoint
\begin{equation}\label{eq:Endpoints}
\bigl(\tfrac1p\vsf,\vsb\tfrac1q\bigr)
=\begin{cases}
\ssf\bigl(\frac12\vsf,\frac12\ssb-\ssb\frac\alpha\beta\frac1n\bigr)
&\text{when \,}\beta\ssb>\msb\frac2n\ssf\alpha\ssf,\\
\ssf\bigl(\frac12\vsf,\vsb0\vsf\bigr)
&\text{when \,}\beta\ssb=\msb\frac2n\ssf\alpha\ssf.\\
\end{cases}\end{equation}
The first case is handled by the refined analysis in \cite{KeelTao1998}
while the second one is excluded.
\smallskip

On the other hand, we prove \eqref{eq:StrichartzMainEstimate2} under the assumptions
$$
0\ssb\le\ssb\beta\ssb\le\ssb\alpha\ssf,\quad
2\ssb\le\ssb p\ssb\le\ssb\infty\ssf,\quad
2\ssb<\ssb q\ssb\le\ssb\infty\ssf,\quad
\tfrac1p\ssb\le\ssb\bigl(\tfrac12\msb-\msb\tfrac1q\bigr)\ssf
\tfrac{2\vsf-\vsf\beta}{2\vsf-\vsf\alpha}\ssf\tfrac32
$$
by considering separately the ranges \ssf$2<q\ssb\le\ssb Q_1$\vsf, $Q_1\msb<\ssb q\ssb<\ssb Q_2$ \vsf and \ssf$Q_2\ssb\le\ssb q\ssb\le\ssb\infty$\ssf, 
where
$$
\tfrac1{Q_1}\msb=\msb\tfrac12\msb-\msb\tfrac13\vsf\tfrac{2\vsf-\vsf\alpha}{2\vsf-\vsf\beta}\msb\in\msb\bigl[\vsf\tfrac16,\ssb\tfrac{1+\vsf\alpha}6\vsf\bigr]
\quad\text{and}\quad
\tfrac1{Q_2}\msb=\msb\tfrac12\msb-\msb\tfrac12\vsf\tfrac{2\vsf-\vsf\alpha}{2\vsf-\vsf\beta}\msb\in\msb\bigl[\ssf0\vsf,\ssb\tfrac\alpha4\vsf\bigr]\vsf.
$$

$\bullet$
\,Assume first that \ssf$q\ssb\ge\ssb Q_2$\vsf.
Then 
{\small$$
2\ssf\sigma(\beta,q)\ssb=\msb\smash{\Bigl(\frac12\msb-\msb\frac1q\Bigr)}(2\ssb-\ssb\beta\vsf)\ssf n\ssb\ge\msb\smash{\Bigl(\frac12\msb-\msb\frac1{Q_2}\Bigr)}(2\ssb-\ssb\beta\vsf)\ssf n\ssb=\msb\smash{\Bigl(1\!-\msb\frac\alpha2\Bigr)}\vsf n
$$}
and we deduce from \eqref{eq:DispersiveAlphaSmallTLarge} that \eqref{eq:StrichartzMainEstimate2} holds for every \ssf$2\ssb\le\ssb p\ssb\le\ssb\infty$\ssf.
\smallskip

$\bullet$
\,Assume next that \ssf$Q_1\msb<\ssb q\ssb<\ssb Q_2$\vsf.
Then
{\small$$
\frac{1/2\msf-\ssf1/q}{1/2\msf-\ssf1/Q_2}\ssf\frac32\msb>\msb\frac{1/2\msf-\ssf1/Q_1}{1/2\msf-\ssf1/Q_2}\ssf\frac32\msb=\msb1
$$}
and we deduce from \eqref{eq:ModifiedDispersiveAlphaSmallTLarge} with \ssf$Q\ssb=\ssb Q_2$ \vsf that \eqref{eq:StrichartzMainEstimate2} holds for every \ssf$2\ssb\le\ssb p\ssb\le\ssb\infty$\ssf.
\smallskip

$\bullet$
\,Assume finally that \ssf$2\ssb<\ssb q\ssb\le\ssb Q_1$\vsf.
Then we obtain \eqref{eq:StrichartzMainEstimate2} under the assumption
\ssf$\smash{\frac1p}\ssb\le\ssb\bigl(\frac12\msb-\msb\frac1q\bigr)\ssf
\frac{2\vsf-\vsf\beta}{2\vsf-\vsf\alpha}\ssf\frac32$
by using \eqref{eq:ModifiedDispersiveAlphaSmallTLarge} with \ssf$Q\ssb=\ssb Q_2$ \vsf together with Young' inequality (see for instance \cite[Thm.~1.4.25]{Grafakos2014}), except for the endpoint \vsf$(p\vsf,\ssb q)\ssb=\ssb(2\vsf,\ssb Q_1)\vsf$, where we use the refined analysis in \cite{KeelTao1998}. Notice that this case is new compared to \cite{AnkerPierfelice2009,AnkerPierfeliceVallarino2012,AnkerPierfelice2014}.

\smallskip

In conclusion,
\eqref{eq:StrichartzMainEstimate1} and \eqref{eq:StrichartzMainEstimate2}
hold under the conditions \eqref{eq:TwoConditionsAdmissibility},
which define a non-empty subset of $\bigl[\ssf0\vsf,\ssb\frac12\vsf\bigr]\!\times\!\bigl[\ssf0\vsf,\ssb\frac12\vsf\bigl)$ \vsf provided that \vsf$\bigl(\frac12\msb-\msb\frac1q\bigr)\vsf\frac\beta\alpha\vsf\frac n2\msb\le\msb\bigl(\frac12\msb-\msb\frac1q\bigr)\vsf\frac{2\vsf-\vsf\beta}{2\vsf-\vsf\alpha}\vsf\frac32$\vsf,
i.e., $\beta\ssb\le\ssb\widehat{\alpha}$\ssf.
\end{proof}


\begin{remark}
In dimension \,$n\ssb=\ssb2$\ssf, Theorem \ref{thm:StrichartzEstimateHnAlphaSmall} holds for the following admissible region\,$:$
\begin{align*}
{\mathcal{R}_{\vsf\alpha}}
&=\bigl\{\bigl(\tfrac1p,\ssb\tfrac1q,\beta\vsf\bigr)\msb\in\msb\bigl[\ssf0\ssf,\ssb\tfrac12\ssf\bigr]\!\times\!\bigl[\ssf0\ssf,\ssb\tfrac12\vsf\bigr]\!\times\!\bigl[\ssf0\ssf,\ssb\alpha\ssf\bigr]\msf\big|\msf\bigl(\tfrac12\msb-\msb\tfrac1q\bigr)\ssf\tfrac\beta\alpha\msb\le\msb\tfrac1p\msb\le\msb\bigl(\tfrac12\msb-\msb\tfrac1q\bigr)\msf{\min\ssf\bigl(\tfrac32\vsf,\vsb\tfrac{1-\vsf\beta}{1-\vsf\alpha}\bigr)}\ssf\tfrac{2\vsf-\vsf\beta}{2\vsf-\vsf\alpha}\msf\bigr\}\\
&\;\;\;\;\;\msf\smallsetminus\msf\bigl\{\vsb\bigl(\tfrac12,\vsb0\vsf,\vsb\alpha\bigr)\vsb\bigr\}\msf\\
&=
\left\{\begin{array}{@{}lr@{}}
    \text{\vphantom{\Big|}\scriptsize{$\bigl\{\bigl(\tfrac1p,\ssb\tfrac1q,\beta\vsf\bigr)\big|\msf\bigl(\tfrac12\msb-\msb\tfrac1q\bigr)\ssf\tfrac\beta\alpha\msb\le\msb\tfrac1p\msb\le\msb\bigl(\tfrac12\msb-\msb\tfrac1q\bigr)\ssf\tfrac32\ssf\tfrac{2\vsf-\vsf\beta}{2\vsf-\vsf\alpha}\msf\bigr\}$}} & \text{\scriptsize{if \,$\alpha\ssb\in\msb\bigl[\frac13\vsf,\ssb1\bigr)$ and \,$\beta\ssb\in\msb\bigl[0\ssf,\ssb\frac{3\ssf\alpha\vsf-1}2\bigr]\vsf,$}}\\
    \multirow{2}{*}{\text{\scriptsize{$\bigl\{\bigl(\tfrac1p,\ssb\tfrac1q,\beta\vsf\bigr)|\msf\bigl(\tfrac12\msb-\msb\tfrac1q\bigr)\ssf\tfrac\beta\alpha\msb\le\msb\tfrac1p\msb\le\msb\bigl(\tfrac12\msb-\msb\tfrac1q\bigr)\ssf\tfrac{1-\vsf\beta}{1-\vsf\alpha}\ssf\tfrac{2\vsf-\vsf\beta}{2\vsf-\vsf\alpha}\msf\bigr\}
\smallsetminus\bigl\{\vsb\bigl(\tfrac12,\vsb0\vsf,\vsb\alpha\bigr)\vsb\bigr\}$}}} 
&
\text{\scriptsize{if \,$\alpha\ssb\in\msb\bigl(0\ssf,\ssb\frac13\bigr]$  and  \,$\beta\ssb\in\msb[\ssf0\ssf,\ssb\alpha\ssf]$}}\\
& \text{\scriptsize{or if \,$\alpha\ssb\in\msb\bigl[\frac13\vsf,\ssb1\bigr)$
 and  \,$\beta\ssb\in\msb\bigl[\frac{3\ssf\alpha\vsf-1}2,\vsb\alpha\vsf\bigr]\vsf$.}}
\end{array} \right.
\end{align*}


For fixed \,$\beta\msb\in\msb[\ssf0\vsf,\vsb\alpha\vsf)$\ssf,
the admissible set of couples \vsf$\bigl(\frac1p,\ssb\frac1q\bigr)$ looks like Figure \ref{fig:AdmissibleRegionAlphaSmallBetaSmall}\ssf,
with
\begin{align*}
\tfrac1{Q_1}\ssb
&=\ssb\tfrac12\ssb-\ssb\tfrac12\max\ssf\bigl\{\tfrac23\vsf,\tfrac{1-\vsf\alpha}{1-\vsf\beta}\bigr\}\ssf\tfrac{2\vsf-\vsf\alpha}{2\vsf-\vsf\beta}\\
&=\begin{cases}
\ssf\frac12-\frac13\frac{2-\alpha}{2-\beta}
&\text{if \,$\alpha\ssb\in\msb\bigl[\frac13\vsf,\ssb1\bigr)$ and \,$\beta\ssb\in\msb\bigl[0\ssf,\ssb\frac{3\ssf\alpha\vsf-1}2\bigr]$\vsf,}\\
\ssf\frac12\msb-\msb\frac12\ssf\frac{1-\vsf\alpha}{1-\vsf\beta}\ssf\frac{2\vsf-\vsf\alpha}{2\vsf-\vsf\beta}
&\text{if \,$\alpha\ssb\in\msb\bigl(0\ssf,\ssb\frac13\bigr]$ and \,$\beta\ssb\in\msb[\ssf0\ssf,\ssb\alpha\ssf]$ or if \,$\alpha\ssb\in\msb\bigl[\frac13\vsf,\ssb1\bigr)$
and \,$\beta\ssb\in\msb\bigl[\frac{3\ssf\alpha\vsf-1}2,\vsb\alpha\vsf\bigr]$\vsf,}\\
\end{cases}    
\end{align*}
{
and, in the limit case \,$\beta\ssb=\ssb\alpha$\ssf, this set boils down to the diagonal
\begin{equation*}
\bigl\{\vsf\bigl(\vsf\tfrac1p\ssf,\ssb\tfrac1q\vsf\bigr)\msb\in\msb\bigl[\ssf0\ssf,\ssb\tfrac12\ssf\bigr)\msb\times\msb\bigl(\ssf0\ssf,\ssb\tfrac12\ssf\bigr]\msf\big|\,\tfrac1p\msb+\msb\tfrac1q\msb=\msb\tfrac12\msf\bigr\}\msf.
\end{equation*}}
\end{remark}

As a general observation, we would like to emphasize that the admissible range of exponents when the power of the Laplacian is below one, i.e., close to a very small diffusion, is smaller than the one for powers closer to the standard diffusion $\alpha >1$. This is due to a combination of two effects: one is due to the necessary loss of derivatives $\sigma$ which cannot be made arbitrarily small but also the behaviour of the kernel in this low diffusive case which is more similar to an Euclidean one. On the other hand, in the regime of higher diffusion, one observes a behaviour much more influenced by the negative curvature.   
From the point of view of nonlinear applications, this introduces substantial difficulties to prove well-posedness.




\section{Homogeneous trees} \label{sect:estimates_trees}
In this section, we consider the discrete analogs of hyperbolic spaces $\H^n$ which are homogeneous trees and more precisely $0-$hyperbolic space according to Gromov \cite{Gromov1987}.
Specifically, for \ssf$Q\msb\ge\msb2$\ssf,  a homogeneous tree of degree $Q+1$ is an infinite connected graph with no loops, in which every vertex is adjoint to $Q+1$ other vertices (see Figure~\ref{fig:Tree}). We denote by
\ssf$\TQ$ the set of vertices in the homogeneous tree
with \ssf$Q\msb+\!1$ edges,
equipped with the counting measure.

\begin{figure}[h!]
    \centering
     \includegraphics[height=40mm]{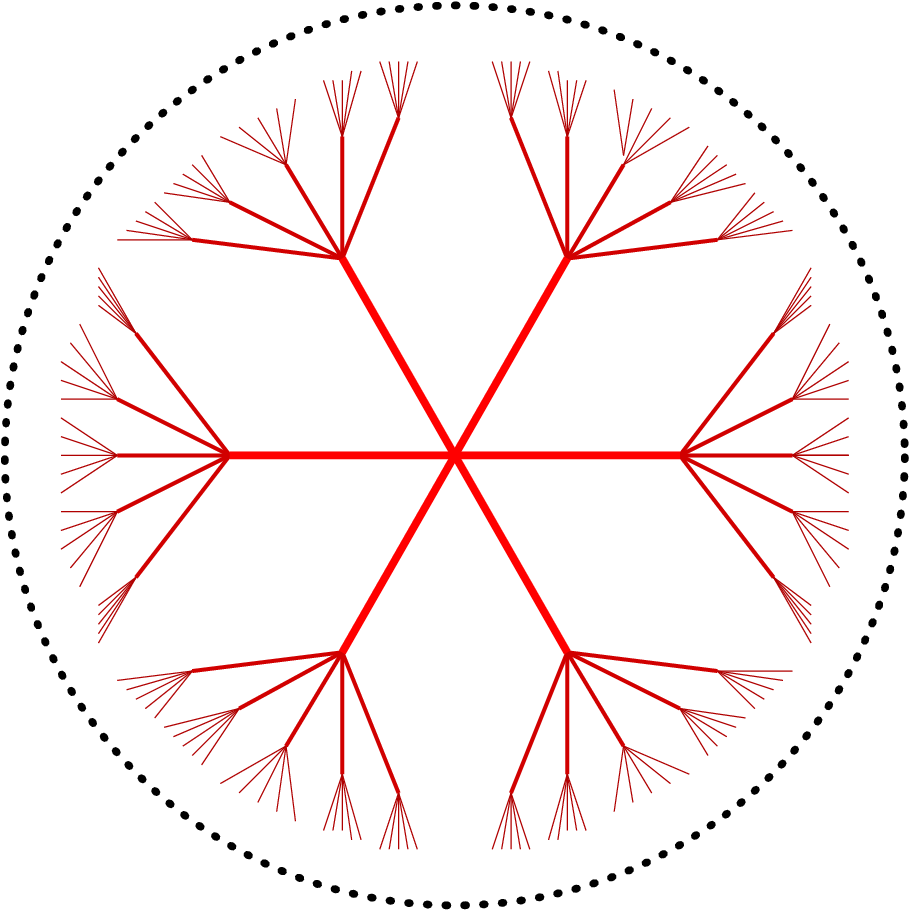}
    \caption{The homogeneous tree \ssf$\mathbb{T}_5$}
    \label{fig:Tree}
\end{figure}

Recall that the combinatorial Laplacian on \ssf$\TQ$ is defined by
\begin{equation*}
\Delta f(x)=\tfrac1{Q\ssf+\vsf1}\sum\nolimits_{\ssf d(y,x)=1}f(y)-f(x)
\end{equation*}
and that its \ssf$\ell^{\ssf2}$ spectrum is equal to
\msf$[-\ssf\gamma_{\ssf0}\!-\!1\vsf,\ssb\gamma_{\ssf0}\!-\!1\vsf]$\ssf,
where \msf$\gamma_{\ssf0}\msb
=\ssb\frac2{Q^{1/2}\vsf+\hspace{.4mm}Q^{-1/2}}
\msb\in\msb(\vsf0\vsf,\ssb1)$\ssf.
Here $d(y,x)$ is the number of edges of the shortest path joining $y$ and $x$. We refer to \cite{Cartier1973, FigaTalamancaNebbia1991, CowlingMedaSetti1998} for some basic tools of harmonic analysis on $\TQ$.\\

Among earlier works {about \eqref{eq:NLSgeneral} on homogeneous trees}, let us mention
\begin{itemize}
\item
\cite{Setti1998}, which is devoted to
the heat equation with continuous time associated with
the Laplacian \ssf$\Delta$ \ssf on \ssf$\TQ$\vsf,
\item
\cite{Stos2011}, which is devoted to
the heat equation with continuous time associated with
the fractional Laplacian
\ssf$(-\Delta)^{\alpha/2}$ on \ssf$\TQ$\vsf,
\item
\cite{MedollaSetti1999}, which is devoted to
the wave equation with continuous time associated with
the shifted Laplacian \ssf$\Delta\msb+\!1\!-\ssb\gamma_{\ssf0}$
\ssf on \ssf$\TQ$\vsf,
\item
\cite{Jamaleddine2013a}, which is devoted to
the Schr\"odinger equation with continuous time associated with
the Laplacian \ssf$\Delta$ \ssf on \ssf$\TQ$\vsf.
\end{itemize}

The equation \eqref{eq:NLSgeneral} on \ssf$\TQ$ is solved and analyzed
as the corresponding equation on \ssf$\Hn$.
The main differences lie in the local (in time) analysis, which is trivial,
and in the spectrum, which is compact. 
More precisely, we have again Duhamel's formula \eqref{eq:SolutionHn}
where
\begin{equation*}
e^{\hspace{.4mm}i\hspace{.4mm}t\ssf(-\Delta)^{\alpha/2}}\msb f
=f\ssb*k_{\ssf t}
\end{equation*}
is the convolution operator defined by the radial kernel
\begin{equation}\label{eq:KernelTQ0}
k_{\ssf t}(r)=\ssf\const\int_{\,0}^{\ssf\tau/2}\hspace{-1mm}
e^{\hspace{.4mm}i\hspace{.4mm}t\msf
[\ssf1-\ssf\gamma(\lambda)\vsf]^{\vsf\alpha/2}}
\hspace{.4mm}\phi_\lambda(r)\,\tfrac{\mathrm{d}\lambda}{|\vsf\bc(\lambda)|^2}.
\end{equation}
Here \msf$\tau\ssb
=\ssb\tfrac{2\vsf\pi\vphantom{|}}{\log Q\vphantom{Q^{1/2}}}$\ssf,
\ssf$\mathbf{c}(\lambda)\ssb
=\ssb\tfrac{1\vphantom{|}}{Q^{\vsf1/2}\ssf+\,Q^{-1/2}}\msf
\tfrac{Q^{\vsf1/2\ssf+\ssf i\vsf\lambda}\ssf
-\,Q^{-1/2\ssf-\ssf i\vsf\lambda}\vphantom{|}}
{Q^{\ssf i\vsf\lambda}\ssf-\,Q^{-i\vsf\lambda}\vphantom{Q^{1/2}}}$\ssf,
\ssf$\gamma(\lambda)\ssb
=\ssb\tfrac{Q^{\ssf i\lambda}\ssf+\,Q^{-i\lambda}\vphantom{|}}
{Q^{\vsf1/2}\ssf+\,Q^{-1/2}}$\msf 
and
\begin{equation}\label{eq:SphericalFunctionsTq}
\phi_\lambda(r)=\ssf
\mathbf{c}\vsf(\lambda)\,Q^{\ssf(-1/2\ssf+\ssf i\vsf\lambda)\ssf r}
+\ssf\mathbf{c}\vsf(\ssb-\lambda)\,Q^{\,(-1/2\ssf-\ssf i\vsf\lambda)\ssf r}\,
\end{equation}
is the spherical function of index $\lambda \in \C$.
By substituting \eqref{eq:SphericalFunctionsTq} in \eqref{eq:KernelTQ0},
we obtain
\begin{align}
k_{\ssf t}(r)&=\ssf\const\msf{Q}^{-\ssf r/2}
\int_{\ssf\R\vsf/\vsb\tau\Z}\!\hspace{1mm}
\bc(\lambda)^{-1}\,e^{\hspace{.4mm}i\hspace{.4mm}t\msf
[\ssf1-\ssf\gamma(\lambda)\vsf]^{\vsf\alpha/2}}\msf
Q^{-\ssf i\ssf\lambda\ssf r}\;\mathrm{d}\lambda\nonumber\\
&=\ssf C\hspace{1mm}Q^{-\ssf r/2}
\int_{\ssf\R\vsf/2\vsf\pi\Z}\!\hspace{1mm}
\tfrac{\sin\lambda\vphantom{|}}
{Q^{\vsf1/2}\ssf e^{\ssf i\vsf\lambda}\ssf
-\,Q^{-1/2}\ssf e^{-i\vsf\lambda}}\hspace{1mm}
e^{\ssf i\hspace{.4mm}t\msf\psi(\lambda)}\msf
e^{-\ssf i\msf r\vsf\lambda}\;\mathrm{d}\lambda\,,
\label{eq:KernelTQ1}\end{align}
where \,$\psi(\lambda)\ssb=\vsb
[\ssf1\!-\ssb\gamma_{\ssf0}\cos\vsb\lambda\msf]^{\vsf\alpha/2}$.

\begin{lemma}[{Stationary phase analysis}]\label{lem:phaseTQ}
Let \,$0\msb<\msb\alpha\msb<\msb2$
and assume that \,$t\msb>\msb0$\ssf.

{\rm(i)}
The phase function \,$\psi$ has two stationary points
on the circle \,$\R\vsf/2\ssf\pi\Z:$
\vspace{.5mm}

\centerline{$
\lambda_1\ssb=2\ssf\pi\Z
\quad\text{and}\quad
\lambda_{\vsf2}\ssb=\pi\ssb+\ssb2\ssf\pi\Z\msf.
$}\vspace{.5mm}

Moreover,
\,$\psi^{\ssf\prime\prime}(\lambda_1)\msb>\msb0$ \ssf and
\,$\psi^{\ssf\prime\prime}(\lambda_{\vsf2})\msb<\msb0$\ssf.

{\rm(ii)}
There exists \,$M\!>\msb0$ such that
the phase function \,$\psi_{\ssf t,r}(\lambda)\msb
:=\ssb t\msf\psi(\lambda)\msb-\ssb r\vsf\lambda$
has
\vspace{.5mm}

\centerline{$\begin{cases}
\,\text{no stationary point}
&\text{if \,}\frac rt\msb>\!M,\\
\,\text{one stationary point \,}\lambda_{\vsf0}
&\text{if \,}\frac rt\msb=\!M,\\
\,\text{two stationary points \,}\lambda_1\vsf,\,\lambda_{\vsf2}
&\text{if \,}\frac rt\msb<\!M.\\
\end{cases}$}\vspace{1mm}

{\rm (iii)}
Mor{e}over, we have the following additional information about the last case.
For every \,$\epsilon\msb>\msb0$\ssf,
there exist open subsets \,$U,V$ \ssb in \,$\R\vsf/2\ssf\pi\Z$ \ssf
and a constant \,$c\msb>\msb0$ such that,
whenever \,$0\msb\le\msb\frac rt\msb
\le\msb M\hspace{-1mm}-\msb\epsilon$,

$\bullet$
\msf$\{\lambda_1,\vsb\lambda_{\vsf2}\}\msb
\subset\ssb U\!\subset\ssb\overline{U}\!\subset\ssb V$\ssb,

$\bullet$
\,$|\ssf\psi_{\ssf t,r}^{\,\prime}\vsf|\ssb\ge\ssb c\,t$
outside \,$U$ and
\,$|\ssf\psi_{\ssf t,r}^{\,\prime\prime}\vsf|\ssb\ge\ssb c\,t$
on \,$V$\ssb.

\end{lemma}

\begin{proof}
Let us compute the first derivatives
\begin{align} 
\psi^{\ssf\prime}(\lambda)
&=\vsf\tfrac\alpha2\,\gamma_{\ssf0}\msf
[\ssf1\!-\ssb\gamma_{\ssf0}\cos\vsb\lambda\msf]^{\vsf\alpha/2\vsf-1}
\ssf\sin\ssb\lambda\,, \nonumber\\
\label{eq:SecondDerivativePsi}
\psi^{\ssf\prime\prime}(\lambda)
&=\vsf\tfrac\alpha2\,\gamma_{\ssf0}\msf
[\ssf1\!-\ssb\gamma_{\ssf0}\cos\vsb\lambda\msf]^{\vsf\alpha/2\vsf-2}\,
[\msf-\msf\tfrac\alpha2\,\gamma_{\ssf0}\ssf\cos^2\!\lambda\ssb
+\cos\ssb\lambda\ssb-\ssb(1\!-\msb\tfrac\alpha2)
\msf\gamma_{\ssf0}\msf]\msf.
\end{align}
Consider the expression
\begin{equation*}\label{eq:ThetaTQ}
\theta(\lambda)
=-\msf\tfrac\alpha2\,\gamma_{\ssf0}\ssf\cos^2\!\lambda\ssb
+\cos\ssb\lambda\ssb
-\ssb(1\!-\msb\tfrac\alpha2)\msf\gamma_{\ssf0}\,,
\end{equation*}
which occurs in \eqref{eq:SecondDerivativePsi}
and which is a \vsf$2\ssf\pi$\vsf--\ssf periodic even function on \ssf$\R$\ssf,
with
\begin{equation*}
\theta(0)=1\ssb-\gamma_{\ssf0}>0\,,
\quad
\theta(\tfrac\pi2)=-\msf
(1\msb-\ssb\tfrac\alpha2)\,\gamma_{\ssf0}<0\,,
\quad\text{and}\quad
\theta(\pi)=-\msf1\ssb-\gamma_{\ssf0}<0.
\end{equation*}
{On \ssf$[\ssf0\ssf,\ssb\pi\ssf]$\ssf, the function \ssf$\theta$ \vsf may increase before decreasing, as its derivatives
\begin{equation*}\label{eq:ThetaTQ_Derivation}
\theta^{\ssf\prime}(\lambda)
=[\ssf\alpha\msf\gamma_{\ssf0}\cos\ssb\lambda\ssb
-\ssb1\ssf]\msf\sin\ssb\lambda
\end{equation*}
vanishes at \ssf$\lambda\ssb=\ssb0$\ssf, at \ssf$\lambda\ssb=\ssb\pi$ \ssf and at most once on \vsf$(\vsf0\ssf,\ssb\tfrac\pi2)$\vsf.}
In par\-ti\-cu\-lar, \ssf$\theta$ \ssf and hence \ssf$\psi^{\ssf\prime\prime}$
\begin{itemize}
\item
$\,$vanishes at a single point \ssf$\lambda_{\vsf0}$
in \ssf$[\ssf0\ssf,\ssb\pi\ssf]$\ssf,
which belongs to \ssf$(\ssf0\ssf,\ssb\frac\pi2\ssf)$\ssf,
\item
$\,$is strictly positive on \ssf$[\ssf0\ssf,\ssb\lambda_{\vsf0}\vsf)$\ssf,
\item
$\,$is strictly negative on \ssf$(\lambda_{\vsf0}\vsf,\ssb\pi\ssf]$\ssf.
\end{itemize}
Thus \msf$\psi^{\ssf\prime}$\vsb,
which is a \vsf$2\ssf\pi$\vsf--\ssf periodic odd function on \ssf$\R$\ssf,
increases (strictly) on \ssf$[\ssf0\ssf,\ssb\lambda_{\vsf0}\ssf]$\ssf,
from \ssf$\psi^{\ssf\prime}(0)\msb=\msb0$ \ssf to \ssf$M\!
=\ssb\psi^{\ssf\prime}(\lambda_{\vsf0})\msb>\msb0$\ssf,
and decreases back on \ssf$[\ssf\lambda_{\vsf0}\vsf,\ssb\pi\ssf]$\ssf,
from \ssf$M$ \vsf to \ssf$0$\ssf.
Con\-se\-quently,

for every \ssf$\mu\!\in\![\ssf0\ssf,\ssb M\vsf)$\vsf, the equation
\msf$\psi^{\ssf\prime}(\lambda)\msb=\msb\mu$
\msf has exactly two solutions in \ssf$(-\ssf\pi\ssf,\ssb\pi\ssf]$\,:
\vspace{.5mm}

\centerline{
$\lambda_1\vsb(\mu)\msb
\in\msb[\ssf0\ssf,\ssb\lambda_{\vsf0}\vsf)$
\quad and\quad
$\lambda_{\vsf2}(\mu)\msb
\in\msb(\vsf\lambda_{\vsf0}\ssf,\ssb\pi\ssf]$\msf.
}\vspace{.5mm}

Let \msf$\epsilon\msb\in\msb(0\ssf,\ssb\frac M3)$
\ssf and assume that \msf$\tfrac rt\msb\le
\msb M\hspace{-1mm}-\msb3\msf\epsilon$\ssf.
Then the phase function \ssf$\psi_{\ssf t,r}$ has two stationary points
in \ssf$(\vsb-\ssf\pi\ssf,\ssb\pi\ssf]\ssf$,
namely $\lambda_1\!=\msb\lambda_1\vsb(\tfrac rt)$ and
$\lambda_{\vsf2}\msb=\msb\lambda_{\vsf2}(\tfrac rt)$\vsf.
Moreover,
\begin{itemize}

\item
$\ssf U_1\msb=\msb
\bigl(-\lambda_1\vsb(\epsilon)\vsf,
\vsb\lambda_1\vsb(M\hspace{-1mm}-\msb2\ssf\epsilon)\bigr)
\msb\subset\ssb V_1\msb=\msb
\bigl(-\lambda_1\vsb(2\ssf\epsilon)\vsf,
\vsb\lambda_1\vsb(M\hspace{-1mm}-\msb\epsilon)\bigr)$
\ssf are neighborhoods of \ssf$\lambda_1$ in
\ssf$(\ssb-\lambda_{\vsf0}\ssf,\ssb\lambda_{\vsf0}\vsf)$ \ssf such that
\ssf $|\ssf\psi^{\ssf\prime}(\lambda)\msb-\msb\tfrac rt\ssf|
\ssb\ge\ssb\epsilon$ \ssf on
\ssf$[\vsb-\lambda_{\vsf0}\ssf,\ssb\lambda_{\vsf0}\ssf]\msb\smallsetminus\msb{U}_{\ssb1}$ and \ssf$\min_{\vsf\lambda\in\overline{V_1}}\ssf
\psi^{\ssf\prime\prime}(\lambda)\msb>\msb0$\ssf,

\item
$\ssf U_{\vsf2}\ssb=\msb
\bigl(\lambda_{\vsf2}(M\hspace{-1mm}-\msb2\ssf\epsilon)\vsf,
\vsb2\ssf\pi\!-\msb\lambda_{\vsf2}(\epsilon)\bigr)
\msb\subset\ssb V_{\vsf2}\ssb=\msb
\bigl(\lambda_{\vsf2}(M\hspace{-1mm}-\msb\epsilon)\vsf,
\vsb2\ssf\pi\!-\msb\lambda_{\vsf2}(2\ssf\epsilon)\bigr)$
\ssf are neighborhoods of \ssf$\lambda_{\vsf2}$ in
\ssf$(\lambda_{\vsf0}\ssf,\ssb2\ssf\pi\!-\msb\lambda_{\vsf0}\vsf)$
\ssf such that
\ssf$|\ssf\psi^{\ssf\prime}(\lambda)\msb-\msb\tfrac rt\ssf|
\ssb\ge\ssb\epsilon$ \ssf on
\ssf$[\vsf\lambda_{\vsf0}\ssf,\ssb2\ssf\pi\!-\msb\lambda_{\vsf0}\ssf]
\msb\smallsetminus\msb{U}_{\vsb2}$
and \ssf$\min_{\vsf\lambda\in\overline{V_{\vsf2}}}
|\ssf\psi^{\ssf\prime\prime}(\lambda)\vsf|\ssb>\ssb0\ssf$.

\end{itemize}
\end{proof}

\begin{remark}
In the limit case \,$\alpha\ssb=\ssb2$\ssf, we have
\vspace{.5mm}

\centerline{
$\psi(\lambda)\ssb
=\ssb1\msb-\vsb\gamma_{\ssf0}\vsf\cos\vsb\lambda$\ssf,
\,$\psi^{\ssf\prime}(\lambda)\ssb
=\vsb\gamma_{\ssf0}\ssf\sin\ssb\lambda$\ssf,
\,$\psi^{\ssf\prime\prime}(\lambda)\ssb
=\vsb\gamma_{\ssf0}\vsf\cos\vsb\lambda\ssf,
$}\vspace{-.5mm}

hence
\vspace{-1mm}

\centerline{
$\lambda_{\vsf0}\ssb=\ssb\frac\pi2$\ssf,
\msf$M\msb=\vsb\gamma_{\vsf0}$\msf,
\,$\lambda_1(\mu)\ssb
=\vsb\arcsin\frac\mu{\gamma_{\vsf0}}$\ssf,
\,$\lambda_{\vsf2}(\mu)\ssb
=\ssb\pi\msb-\ssb\arcsin\frac\mu{\gamma_{\vsf0}}$\ssf.
}

\end{remark}

\begin{theorem}[{Kernel estimate}]\label{KernelEstimateTQ}
Assume that \,$0\msb<\msb
\alpha\msb\le\msb2$\ssf.
Then the following kernel estimate holds\,$:$
\begin{equation}\label{eq:KernelEstimateTQ1}
|\vsf k_{\ssf t}(r)|\lesssim Q^{\ssf-\frac r2}
\qquad\forall\;t\msb\in\msb\R^*,
\,\forall\hspace{1mm}r\!\in\msb\N\msf.
\end{equation}
Moreover, there exists \,$C\!>\msb0$ such that 
\begin{equation}\label{eq:KernelEstimateTQ2}
|\vsf k_{\ssf t}(r)|\lesssim
|t|^{-\frac32}\msf(1\!+\msb r)\,Q^{\ssf-\frac r2}
\end{equation}
if \,$1\!+\hspace{-.4mm}r\msb\le\ssb C\msf|t|$\ssf.
\end{theorem}

\begin{remark}${}$

\begin{itemize}
\item
In the limit case \,$\alpha\ssb=\ssb2$\ssf,
the slightly weaker estimate
\begin{equation*}
|\vsf k_{\ssf t}(r)|\ssf\lesssim\ssf
\begin{cases}
\,Q^{\ssf-\ssf r/2}
&\text{if }\;0\msb<\!|t|\!<\!1\\
\,|t|^{-\ssf3/2}\msf(1\!+\msb r)^2\,Q^{\ssf-\ssf r/2}
&\text{if }\;|t|\!\ge\!1
\end{cases}
\end{equation*}
was obtained in {\cite[Prop.~3.1]{Jamaleddine2013a}}. Note that a 
small error in \cite[Prop.~\;3.1]{Jamaleddine2013a}
was corrected in \cite[Prop.~\;3.1]{Jamaleddine2013b}.
\item
The estimate \eqref{eq:KernelEstimateTQ1} may be further improved,
when \,$|t|$ is large and \,$\frac r{|t|}$ is bounded from below,
but we will not need it.
\end{itemize}

\end{remark}

\begin{proof}[Proof of Theorem \ref{KernelEstimateTQ}]
Without loss of generality, we may assume that
\msf$t\msb>\msb0$\ssf.
The estimate \eqref{eq:KernelEstimateTQ1} follows immediately
from the expression \eqref{eq:KernelTQ1}, where the integrand is bounded.
Let us improve \eqref{eq:KernelEstimateTQ1}
when \msf$t\ssb\ge\!1$ \ssf and \ssf$\frac rt$ \ssf is small,
so that Lemma \ref{lem:phaseTQ}~(iii) applies.
First, we perform an integration by parts based on
\begin{equation*}
(\ssf\sin\ssb\lambda)\,
e^{\ssf i\hspace{.4mm}t\msf\psi(\lambda)}
=\tfrac2{\alpha\msf\gamma_{\vsf0}}\,\tfrac1t\,
[\ssf1\!-\ssb\gamma_{\ssf0}\cos\vsb\lambda\msf]^{\vsf1-\ssf\alpha/2}\,
\bigl(\ssb-\ssf i\ssf\tfrac\partial{\partial\lambda}\bigr)\msf
e^{\ssf i\hspace{.4mm}t\msf\psi(\lambda)}\,.
\end{equation*}
This way, \eqref{eq:KernelTQ1} becomes
\begin{equation}\label{eq:KernelTQ2}
k_{\ssf t}(r)=\ssf C\,\tfrac{1\vphantom{|}}{t\vphantom{|}}\,Q^{-\ssf r/2}
\int_{\ssf\R\vsf/2\vsf\pi\Z}\!\hspace{1mm}
e^{\msf i\msf\psi_{\vsf t,r}\vsb(\lambda)}\,
\bigl\{\ssf r\ssf a(\lambda)\msb
+\vsb i\msf a^{\ssf\prime}(\lambda)\vsf\bigr\}\;\mathrm{d}\lambda\,,
\end{equation}
where \,$a(\lambda)\ssb=\ssb\tfrac
{[\ssf1-\ssf\gamma_{\vsf0}\cos\vsb\lambda\msf]^{\vsf1-\vsf\alpha/2}
\vphantom{\frac00}}{Q^{\vsf1/2}\ssf e^{\ssf i\vsf\lambda}\ssf
-\,Q^{-1/2}\ssf e^{-i\vsf\lambda}}$
\msf is bounded, as well as its derivatives.
Next, we estimate \eqref{eq:KernelTQ2}
by stationary phase analysis based on Lemma \ref{lem:phaseTQ}~(iii).
Specifically, given a smooth function \ssf$\chi$ on \ssf$\R\vsf/2\ssf\pi\Z$
\ssf such that \ssf$\chi\msb=\msb1$ on \ssf$\overline{U}$
and \msf$\supp\chi\msb\subset\msb V$,
we split up the integral in \eqref{eq:KernelTQ2} as follows\,:
\vspace{-1mm}

\centerline{$\displaystyle
\int_{\ssf\R\vsf/2\vsf\pi\Z}\!\;\mathrm{d}\lambda\,
=\int_{\ssf V}\chi(\lambda)\;\mathrm{d}\lambda\msf
+\int_{(\R\vsf/2\vsf\pi\Z)\ssf\smallsetminus\ssf U}\!
\,[\vsf1\!-\msb\chi(\lambda)\vsf]\,\;\mathrm{d}\lambda.
$}\vspace{.5mm}

On the one hand, the main estimate
\begin{equation*}
\Bigl|\msf\int_{\ssf V}\ssb \chi(\lambda)\,
e^{\msf i\msf\psi_{\vsf t,r}\vsb(\lambda)}\msf
\bigl\{\ssf r\ssf a(\lambda)\msb
+\vsb i\msf a^{\ssf\prime}(\lambda)\vsf\bigr\}\;\mathrm{d}\lambda\hspace{.4mm}\Bigr|
\lesssim\ssf t^{\vsf-\frac12}\msf(1\msb+\ssb r)
\end{equation*}
is obtained by applying Lemma \ref{lem:vanderCorput} with \ssf${L}\msb=\ssb2$\ssf.
On the other hand, the remainder estimate
\begin{equation*}
\Bigl|\msf\int_{(\R\vsf/2\vsf\pi\Z)\ssf\smallsetminus\ssf U}\!\,[\vsf1\!-\msb\chi(\lambda)\vsf]\,
e^{\msf i\msf\psi_{\vsf t,r}\vsb(\lambda)}\msf
\bigl\{\ssf r\ssf a(\lambda)\msb
+\vsb i\msf a^{\ssf\prime}(\lambda)\vsf\bigr\}\;\mathrm{d}\lambda\hspace{.4mm}\Bigr|
\lesssim\ssf t^{\vsf-N}\ssf(1\msb+\ssb r)
\end{equation*}
is obtained after \ssf$N$ integrations by parts based on
\vspace{1mm}

\centerline{$
e^{\msf i\msf\psi_{\vsf t,r}\vsb(\lambda)}
=\tfrac{1}{\psi_{\vsf t,r}^{\msf\prime}\vsb(\lambda)}
\msf\bigl(\ssb-\ssf i\ssf\tfrac\partial{\partial\lambda}\bigr)\msf
e^{\msf i\msf\psi_{\vsf t,r}\vsb(\lambda)}\ssf.
$}\vspace{.5mm}

This concludes the proof of \eqref{eq:KernelEstimateTQ2}.
\end{proof}

{Let us turn to \ssf$\ell^{\ssf q^{\ssf\prime}}\msb(\mathbb{T}_Q)\!\to\msf\ell^{\ssf\tilde{q}}(\mathbb{T}_Q)$ mapping properties of the Schr\"{o}dinger operator $e^{it(-\Delta)^{\alpha/2}}$.}
As in {\cite[Thm.~3.4 and Cor.~3.5]{Jamaleddine2013a}},
let us deduce the following result from Theorem~\ref{KernelEstimateTQ}.

\begin{corollary}[{Dispersive estimate}] \label{cor:DispersiveEstimateTQ}
Let \,$0\msb<\msb\alpha\msb\le\msb2$
and \,$2\msb<\msb q,{\tilde{q}}\msb\le\msb\infty$\ssf.
Then the following dispersive estimate holds for
\,$t\msb\in\msb\R^*:$
\begin{equation*}
\bigl\|\msf
e^{\hspace{.4mm}i\hspace{.4mm}t\hspace{.4mm}(-\Delta)^{\alpha/2}}
\ssf\bigr\|_{\ssf\ell^{\ssf q^{\ssf\prime}}\msb{(\mathbb{T}_Q)}\ssf\to\ssf\ell^{\ssf\tilde{q}}{(\mathbb{T}_Q)}}
\lesssim\ssf(1\!+\msb|t|\vsf)^{-\frac32}\msf.
\end{equation*}
\end{corollary}

{Similar, we can conclude from Theorem~\ref{KernelEstimateTQ} and \cite[Thm.~3.6]{Jamaleddine2013a} the following result for the inhomogeneous case.}

\begin{corollary}[{Strichartz estimates}]\label{cor:StrichartzEstimateTQ}
{Let \,$0\msb<\msb\alpha\msb\le\msb2$
and \,$I\msb=\msb(-T,\ssb+T\vsf)$ with \,$T\msb>\ssb0$\ssf.}
Then the following Strichartz estimates hold
for solutions of \eqref{eq:NLSgeneral}
on \,${I}\msb\times\ssb\TQ:$
\begin{equation*}
\|\ssf u(t,x)\|_{L^\infty({I,}\msf\ell^{\ssf2}(\TQ))}+\ssf\|\ssf u(t,x)\|_{L^{\ssf\tilde{p}}({I,}\msf\ell^{\ssf\tilde{q}}(\TQ))}\ssb
\le C\,\bigl\{\ssf\|f\|_{\ssf\ell^{\ssf2\vsb}(\TQ)}
+\ssf\|\ssf F(t,x)\|_{L^{\ssf p^{\ssf\prime}}\!({I,}\msf\ell^{\ssf q^{\ssf\prime}}\!(\TQ))}\ssf\bigr\}\,.
\end{equation*}
Here \,$(\frac{1\vphantom{|}}p,\ssb\frac{1\vphantom{|}}q)$ and
\,$(\frac{1\vphantom{|}}{\tilde{p}},\ssb\frac{1\vphantom{|}}{\tilde{q}}\ssf)$
belong to the square {$($see Figure \ref{fig:AdmissiblePairsTQ}\ssf$)$}
\vspace{.5mm}

\centerline{$
{\bigl[}\vsf0\ssf,\ssb\frac12\ssf\bigr]\msb
\times\msb\bigl[\ssf0\ssf,\ssb\frac12\vsf\bigr)
\,\cup\,\bigl\{\bigl(\vsf0\ssf,\ssb\frac12\vsf\bigr)\bigr\}
$}\vspace{1.5mm}

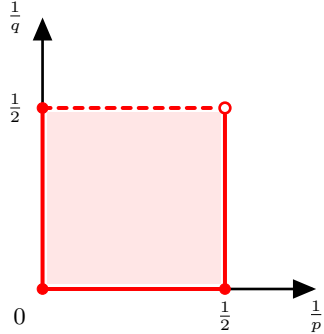
\begin{figure}[h!]
\centering
\begin{tikzpicture}[line cap=round,line join=round,>=triangle 45,x=6mm,y=6mm]
\clip(-1,-1) rectangle (7,7);
\fill[line width=0pt,color=red,fill=red,fill opacity=0.1] (0.1,3.9) -- (3.9,3.9) -- (3.9,0.1) -- (0.1,0.1) -- cycle;
\draw [line width=1.5pt,color=red] (4,0) -- (4,3.85);
\draw [line width=1.5pt,color=red] (0,0) -- (4,0);
\draw [line width=1.5pt,color=red] (0,0) -- (0,4);
\draw [line width=1.5pt,dashed,color=red] (0,4) -- (3.85,4);
\draw [line width=1pt,->] (0,4.1) -- (0,6);
\draw [line width=1pt,->] (4.1,0) -- (6,0);
\begin{scriptsize}
\draw[color=black] (6,-0.6) node {$\frac1p$};
\draw[color=black] (-.6,6) node {$\frac1q$};
\draw[color=red,fill=red] (0,0) circle (2pt);
\draw[color=black] (-0.5,-0.6) node {$0$};
\draw[color=red,fill=red] (0,4) circle (2pt);
\draw[color=black] (-0.6,4) node {$\frac12$};
\draw[color=red,fill=red] (4,0) circle (2pt);
\draw[color=black] (4,-0.6) node {$\frac12$};
\draw[line width=1pt,color=red] (4,4) circle (2pt);
\end{scriptsize}
\end{tikzpicture}
\caption{Admissible pairs for $\mathbb{T}_Q$}
\label{fig:AdmissiblePairsTQ}
\end{figure}

and \,$C\!\ge\msb0$ depends on \,$\alpha$\vsf,
$(\vsf p,\vsb q)$\vsf, $(\vsf\tilde{p},\vsb\tilde{q}\ssf)$
but not on \,$T$ and \,$u$\ssf.
\end{corollary}

\begin{remark}
Notice that, in the discrete setting and contrary to the continuous setting, the case \,$\alpha\ssb=\msb1$ $($\vsf half-wave equation\ssf$)$ is similar to the general case \,$0<\alpha\le2$\ssf.
\end{remark}

Let us mention that for the nonlinear Schr\"{o}dinger (NLS) equation on homogenous trees
\begin{equation} \label{eq:NLS_trees}
    \begin{cases}
    i \partial_t u(x,t)+(-\Delta_x)^{\alpha/2}u(x,t)=\widetilde{F}(u(x,t)) \quad (x,t) \in \TQ \times \R\\
    u(x,0)=u_0
\end{cases}
\end{equation}
where
$$
\begin{cases}
    |\widetilde{F}(u)| \lesssim |u|^\eta\\
    |\widetilde{F}(u)-\widetilde{F}(v)| \lesssim \{|u|^{\eta-1} +|v|^{\eta-1}\} |u-v|
\end{cases}
$$
for some exponent $\eta >1$, we get similar local and global well-posedness results as in \cite[Thm.~4.1.]{Jamaleddine2013a} for $\alpha=2$.
\begin{theorem}
Let $1 < \eta < \infty.$ Then the nonlinear Schr\"odinger equation \eqref{eq:NLS_trees} is
\begin{itemize}
\item
\,locally well-posed for arbitrary initial data in \,$\ell^{\ssf2}$,
\item
\,globally well-posed for small initial data in \,$\ell^{\ssf2}$,
\item
\,globally well-posed for arbitrary initial data in \,$\ell^{\ssf2}$ under the additional gauge invariant condition \,$\mathrm{Im}\{\widetilde{F}(u)\overline{u}\}=0$\ssf.
\end{itemize}
\end{theorem}

\section*{Acknowledgments}
G. Palmirotta receives funding from
the German Research Foundation via the grant SFB-TRR 358/1 2023 — 491392403. Y. Sire is partially supported by NSF DMS Grant $2154219$, "Regularity {\sl vs} singularity formation in elliptic and parabolic equations". 

\normalem
\bibliographystyle{amsalpha}
\bibliography{literature}

\end{document}